\documentclass{amsart}
\usepackage[top=1.5in, bottom=1in, left=0.9in, right=0.9in]{geometry}

\usepackage[utf8]{inputenc} 
\usepackage[all]{xy}
\usepackage[dvipsnames]{xcolor}
\usepackage{hyperref}

\usepackage{amsmath, mathtools}
\hypersetup{colorlinks=true}
\usepackage{amssymb}
\usepackage{amsfonts}
\usepackage{graphicx}
\usepackage{tikz}
\usepackage{mathrsfs}
\usepackage{amscd}
\usepackage[all]{xy}
\usepackage{amsthm}
\usepackage{enumerate}
\usepackage{tikz-cd} 
\usepackage{cleveref}
\usepackage{verbatim}
\usepackage{mathtools}
\usepackage{comment}

\usepackage{comment}

\newtheorem{thm}{\bf Theorem}[section]

\newtheorem{prop}[thm]{\bf Proposition}
\newtheorem{lemma}[thm]{\bf Lemma}
\newtheorem{cor}[thm]{\bf Corollary}
\theoremstyle{definition}
\newtheorem{definition}[thm]{\bf Definition}
\theoremstyle{remark}
\newtheorem{remark}[thm]{\bf Remark}

\newtheorem{conjecture}[thm]{\bf Conjecture}
\newtheorem{question}[thm]{\bf Question}
\newtheorem{example}[thm]{\bf Example}
\numberwithin{equation}{section}

\newcommand{\fps[1]}{[\![#1]\!]}

\DeclareMathOperator{\height}{{ht}}

\DeclareMathOperator{\depth}{{depth}}

\DeclareMathOperator{\hilb}{{Hilb}}

\DeclareMathOperator{\chara}{{{char}}}

\DeclareMathOperator{\gr}{{{gr}}}

\newcommand{\modulo[1]}{\text{ (mod} \, #1)}

\DeclareMathOperator{\cs}{{cs}}
\DeclareMathOperator{\ms}{{ms}}
\DeclareMathOperator{\mcn}{{mcn}}

\def\f0{\mathbf{0}}

\def\fm{\mathfrak{m}}

\def \QQ{\mathbb Q}
\def \CC{\mathbb C}
\def \RR{\mathbb R}

\def \ZZ{\mathbb Z}

\def \cc{\mathcal C}

\begin{document}

\title[Some new invariants of Noetherian local rings related to square of the maximal ideal]{Some new invariants of Noetherian local rings related to square of the maximal ideal} 

\author{Dylan C. Beck and Souvik Dey}
\address{Department of Mathematics \\ University of Kansas \\ 405 Snow Hall, 1460 Jayhawk Blvd. \\ Lawrence, KS 66045}
\email{dylan.beck@ku.edu and souvik@ku.edu}


\begin{abstract}
We introduce two new invariants of a Noetherian (standard graded) local ring $(R, \fm)$ that measure the number of generators of certain kinds of reductions of $\fm,$ and we study their properties. Explicitly, we consider the minimum among the number of generators of ideals $I$ such that either $I^2 = \fm^2$ or $I \supseteq \fm^2$ holds. We investigate subsequently the case that $R$ is the quotient of a polynomial ring $k[x_1, \dots, x_n]$ by an ideal $I$ generated by homogeneous quadratic forms, and we compute these invariants. We devote specific attention to the case that $R$ is the quotient of a polynomial ring $k[x_1, \dots, x_n]$ by the edge ideal of a finite simple graph $G.$
\end{abstract}

\maketitle

\section{Introduction}\label{Introduction}

Our work began as a simple curiosity: given ideals $I$ and $J$ in a commutative unital ring $R$ such that $I^2 = J^2,$ how ``close'' must $I$ and $J$ be? For simplicity, suppose that $J$ is a prime ideal of $R.$ Localizing at $J$ reduces to the case that $(R, J)$ is a local ring with $I^2 = J^2.$ Unless otherwise specified, we will henceforth assume that $(R, \fm)$ is a commutative unital Noetherian local ring with a unique maximal ideal $\fm.$ If $R = \bigoplus_{i \geq 0} R_i$ is standard graded, we assume that $R_0$ is a field, $R = R_0[R_1],$ and $\fm = \bigoplus_{i \geq 1} R_i$ is the homogeneous maximal ideal of $R$.

\begin{question}
Let $(R, \fm)$ be a Noetherian (standard graded) local ring. If $I \subsetneq \fm$ is a (homogeneous) ideal of $R$ such that $I^2 = \fm^2,$ what can be said of $R$? What can be said of $\mu(I)$?
\end{question}

Originally, we noticed that if $(R, \fm)$ is a regular local ring with an ideal $I \subseteq \fm$, then $I^2 = \fm^2$ only if $I = \fm,$ hence we were naturally led to study the invariant $$\cs(R) = \min \{\mu(I) \mid I \subseteq \fm \text{ is a (homogeneous) ideal of } R \text{ such that } I^2 = \fm^2\}.$$ Generally, it holds that $\dim R \leq \cs(R) \leq \mu(\fm)$ with $\cs(R) = \mu(\fm)$ if and only if for any ideal $I$ of $R,$ the equality $I^2 = \fm^2$ implies that $I = \fm$ (cf. the first and fifth parts of Proposition \ref{main proposition about cs and ms}, respectively). In particular, if $R$ is a regular local ring, then for any ideal $I$ of $R,$ the equality $I^2 = \fm^2$ implies that $I = \fm$; however, there exist non-regular rings for which $\cs(R) = \mu(\fm).$ For instance, if $R$ is a hypersurface, then $\cs(R) = \mu(\fm)$ (cf. Corollary \ref{cs and ms for hypersurface}). On its own, this observation already gives enough reason to study $\cs(R)$ for various kinds of rings.

Likewise, we consider the more restrictive scenario that $\fm^2 = \fm I$ implies that $I = \fm,$ i.e., the unique maximal ideal $\fm$ does not admit a proper reduction of reduction number one (cf. \cite[Definitions 1.2.1 and 8.2.3]{HS}). By Proposition the fourth part of \ref{main proposition about cs and ms}, one correct invariant to look at is $$\ms(R) = \min \{\mu(I) \mid I \subseteq \fm \text{ is a (homogeneous) ideal of } R \text{ such that } \fm^2 \subseteq I\}.$$ Ultimately, it is not obvious and takes some work to establish that $\ms(R)$ relates to the question of $\fm^2 = \fm I.$ We illustrate that this connection mainly hinges on the induction performed in the results preceding Proposition \ref{ms equivalent to m^2 = mI inductive step} with the end result of Corollary \ref{ms measures number of generators of reductions of m with reduction number 1} verifying this motivating claim.

Largely, Section \ref{Preliminaries and basic properties of the invariants} is devoted to establishing the result of Corollary \ref{ms measures number of generators of reductions of m with reduction number 1}; along the way, however, we prove that one can attain the values $\ms(R)$ and $\cs(R)$ by ideals generated by elements not in $\fm^2.$ We demonstrate subsequently that $\ms(R)$ can be obtained by an ideal generated by general linear elements (cf. Proposition \ref{ms is an open condition} for the precise statement). We conclude this section by recording two observations about the behavior of $\ms(R)$ and $\cs(R)$ with respect to familiar ring operations. Particularly, we show that $\ms(R)$ and $\cs(R)$ decrease along surjective ring homomorphisms and that they remain unchanged when passing to the $\fm$-adic completion. Even more, if $R$ is a standard graded algebra over a field $k,$ then the polynomial ring $S = R[X_{t - 1}, \dots, X_n]$ over $R$ in $n - t$ indeterminates satisfies $\ms(S) = \ms(R) + n - t$ (cf. Proposition \ref{ms of algebra obtained by adjoining indeterminates} for details).

In Section \ref{General bounds on ms(R) and cs(R)}, we present some general bounds for $\ms(R)$ and $\cs(R).$ Chiefly, Proposition \ref{main proposition about cs and ms} gives bounds on these invariants in terms of the (embedding) dimension of $R$ and examines the invariants of $\ms(R)$ and $\cs(R)$ when they are equal to $\dim(R)$ and $\mu(\fm).$ We discuss also the behavior of $\ms(R)$ and $\cs(R)$ when they are ``close to'' the two boundary values of $\dim R$ and $\mu(\fm),$ respectively. We provide in Proposition \ref{upper bound for ms in terms of upper bound for mu(m^2)} a useful bound for $\ms(R)$ when $\mu(\fm^2)$ is sufficiently small. We end this section with Proposition \ref{cs and ms of fiber product} that relates $\cs(R)$ and $\ms(R)$ of two local rings (with the same residue field) to that of their fiber product.

We turn out attention in Section \ref{The standard graded local case and the Weak Lefschetz Property} to studying the first basic properties of the invariants in the standard graded case. Particularly, we relate $\ms(R)$ to the Weak Lefschetz Property of a standard graded Artinian $k$-algebra. Even more, for a standard graded algebra $S = k[x_1, \dots, x_n]/J,$ we show in Proposition \ref{cs and ms depend only on degree two} that it is enough to study the invariants $\cs(R)$ and $\ms(R)$ in the case that $J$ is generated by quadratic polynomials. We also determine bounds on the invariants for the $n$th Veronese subring of $k[x, y]$ in Propositions \ref{cs and ms of Veronese subring} and \ref{upper bound for cs of Veronese subring} by relating them to cardinality of subsets $S$ of $\{0, 1, \dots, n\}$ such that $S + S = \{s + t \mid s, t \in S\} = \{0, 1, \dots, 2n\}.$

Using the result of Proposition \ref{cs and ms depend only on degree two}, we are motivated in Section \ref{Computing ms(R) and cs(R) for quotients by quadratic ideals} to devote specific attention to computing the invariants $\ms(R)$ and $\cs(R)$ in the case that $R$ is the quotient of a polynomial ring in $n$ indeterminates by a homogeneous ideal $I$ generated by quadratic forms. We establish a connection between the minimum number of generators of $I,$ the number of indeterminates of $R,$ and $\cs(R)$ in Proposition \ref{cs of k[x_1, ..., x_n] modulo t homogeneous quadratics} that provides a useful lower bound on $\cs(R)$; then, we use this technique to investigate $\ms(R)$ and $\cs(R)$ in several examples for which either $n$ or $\dim(R/I)$ is small. Particularly, we provide in Remark \ref{cs lower bound using u(m^2)} a general lower bound for $\cs(R).$

Last, in Section \ref{Computing ms(R) and cs(R) for the edge ring of a finite simple graph}, we consider the case that $R$ is the quotient of a polynomial ring by a quadratic squarefree monomial ideal. By the Stanley-Reisner Correspondence, this case is equivalent to the case that $R$ is the edge ring of a finite simple graph $G.$ We relate $\cs(R)$ and $\ms(R)$ to various special properties of finite simple graphs, and we estimate these invariants for many familiar classes of finite simple graphs. We prove in Proposition \ref{ms(k(G)) leq n - mcn(G) + 3} that if $\overline G$ is chordal, then $\ms(R)$ is equal to the independence number of $G.$ If $\overline G$ is not chordal, then the same proposition provides a bound on $\ms(R)$ in terms of the number of vertices and the minimum length of a cycle in $\overline G.$ Particularly, if the minimum length of a cycle of $\overline G$ is four, it turns out that $\ms(R)$ can be quite subtle. We address this case specifically for the path graph in Proposition \ref{cs(P_n) and ms(P_n)} and the cycle graph in Proposition \ref{cs(C_n) and ms(C_n)}. Even though we are not able to give a specific value for $\cs(R)$ or $\ms(R)$ for the cycle graph, we provide Macaulay2 code in Remark \ref{better bounds for cs(C_n) and ms(C_n)} that we use to conjecture better bounds for these two invariants. By the result of Proposition \ref{K_ell-connected edge cover}, it seems that under certain circumstances $\ms(R)$ is related to the number of edges of an edge cover of $G$ in which all of the edges share a common edge or overlap at a vertex. We relate $\cs(R)$ and $\ms(R)$ with the invariants of the graph join of two finite simple graphs in Proposition \ref{cs(G*H) and ms(G*H)}, and we use this to subsequently provide bounds on $\cs(R)$ for the complete ($t$-partite) graph and the wheel graph. We conclude this section with Proposition \ref{ms(K_m vee K_n)} on the wedge of complete graphs and corollaries thereof.

\section{Preliminaries and basic properties of the invariants}\label{Preliminaries and basic properties of the invariants}

We will denote by $(R, \fm, k)$ a commutative Noetherian (standard graded) local ring with unique (homogeneous) maximal ideal $\fm$ and residue field $k.$ By Nakayama's Lemma, the positive integer $\mu(I) = \dim_k(I / \fm I)$ is the unique minimum number of generators of an ideal $I$ of $R.$ Further, we denote by $\mu_1(I)= \dim_k(I / (I \cap \fm^2))$ denote the number of the generators in a minimal generating set of $I$ that do not lie in $\fm^2.$ We say that an element $x \in I$ is {\it general} if the image of $x$ in $I / \fm I$ lies in a nonempty Zariski open subset of $I / \fm I.$ We may assume that the residue field $k$ of $R$ is infinite in order to guarantee the existence of general elements.

Recall that an ideal $J \subseteq I$ is a {\it reduction} of $I$ if there exists an integer $r \gg 0$ such that $I^{r + 1} = I^r J.$ We refer to the least integer $r$ such that $I^{r + 1} = I^r J$ as the {\it reduction number} of $I$ with respect to $J.$ Even more, if $J$ is minimal with respect to inclusion among all reductions of $I,$ then $J$ is said to be a {\it minimal reduction} of $I$; the absolute reduction number of $I$ is the minimum among all reduction numbers of all minimal reductions of $I.$ By \cite[Theorem 8.3.5]{HS}, minimal reductions exist for each ideal of a Noetherian local ring.

\begin{definition}\label{cs and ms definition}
Consider a Noetherian (standard graded) local ring $(R, \fm).$ We define
\begin{align*}
\cs(R) &= \min \{\mu(I) \mid I \subseteq \fm \text{ is a (homogeneous) ideal of } R \text{ such that } I^2 = \fm^2\} \text{ and} \\
\ms(R) &= \min \{\mu(I) \mid I \subseteq \fm \text{ is a (homogeneous) ideal of } R \text{ such that } I \supseteq \fm^2\}.
\end{align*}
We say that an ideal $I$ \textit{witnesses} (or is a \textit{witness of}) $\cs(R)$ or $\ms(R)$ provided that $I^2 = \fm^2$ and $\mu(I) = \cs(R)$ or $I \supseteq \fm^2$ and $\mu(I) = \ms(R),$ respectively. We may also say in this case that $\cs(R)$ or $\ms(R)$ is \textit{witnessed by} $I.$
\end{definition}

Our immediate task is to establish that the invariants $\ms(R)$ and $\cs(R)$ can be witnessed by ideals generated by linear forms. Explicitly, if $\ms(R) = n$ or $\cs(R) = n$, we claim that there exist elements $x_1, \dots, x_n \in \fm \setminus \fm^2$ such that $(x_1, \dots, x_n) R \supseteq \fm^2$ or $(x_1, \dots, x_n)^2 = \fm^2,$ respectively.

\begin{prop}\label{cs is witnessed by linear forms}
We have that $\mu_1(I) \leq \mu (I).$ Equality holds if and only if $I \cap \fm^2 = \fm I.$ Consequently, if $I$ witnesses $\cs(R),$ then $\mu(I) = \mu_1(I),$ i.e., $\cs(R)$ is witnessed by linear forms.
\end{prop}

\begin{proof}
Consider the short exact sequence $0 \to (I \cap \fm^2) / \fm I \to I / \fm I \to I / (I \cap \fm^2) \to 0$ induced by the inclusion $\fm I \subseteq I \cap \fm^2.$ We have that $\mu_1(I) \leq \mu(I)$ by the additivity of length on short exact sequences. Equality holds if and only if $(I \cap \fm^2) / \fm I = 0$ if and only if $I \cap \fm^2 = \fm I.$ Last, if $I$ witnesses $\cs(R),$ then $I^2 = \fm^2$ so that $(I \cap \fm^2) = (I \cap I^2) = I^2 = \fm^2 = \fm I,$ where the last equality follows from $\fm^2 = I^2 \subseteq \fm I \subseteq \fm^2.$
\end{proof}

Likewise, the claim holds for $\ms(R),$ but the proof requires an induction on $\ms(R).$

\begin{prop}\label{ms is witnessed by linear forms}
There exists a (homogeneous) ideal $I$ of $R$ that witnesses $\ms(R)$ that satisfies $\mu(I) = \mu_1(I).$ Put another way, $\ms(R)$ is witnessed by (homogeneous) linear forms.
\end{prop}

We will establish Proposition \ref{ms is witnessed by linear forms} by first establishing the following lemma and corollary.

\begin{lemma}\label{ms is witnessed by linear forms base case}
Let $(R, \fm)$ be a (standard graded) local ring with $\fm^2 \neq 0.$ If there exists a (homogeneous) element $x \in \fm$ such that $\fm^2 \subseteq xR,$ then there exists a (homogeneous) element $y \in \fm \setminus \fm^2$ such that $\fm^2 \subseteq yR.$
\end{lemma} 

\begin{proof}
Consider a (homogeneous) element $x \in \fm$ that satisfies $\fm^2 \subseteq xR.$ On the contrary, suppose that for every (homogeneous) element $y \in \fm \setminus \fm^2,$ we have that $\fm^2 \not \subseteq yR.$ By hypothesis that $\fm^2 \subseteq xR,$ we must have that $x \in \fm^2,$ from which it follows that $xR \subseteq \fm^2$ and $\fm^2 = xR.$ We claim that for any (homogeneous) element $\ell \in \fm \setminus \fm^2,$ we have that $x \in \ell \fm$ if and only if $\ell \fm$ contains a minimal generator of $\fm^2 = xR$ if and only if $\ell \fm \cap (\fm^2 \setminus \fm^3) \neq \emptyset.$ We verify the first equivalence. If $x \in \ell \fm,$ then as $x$ is a minimal generator of $\fm^2$ by hypothesis that $\fm^2 \neq 0,$ it is clear that $\ell \fm$ contains a minimal generator of $\fm^2.$ Conversely, if $\ell \fm$ contains a minimal generator of $\fm^2,$ then there exists an element $z \in \ell \fm \cap (\fm^2 \setminus \fm^3)$ such that $zR = \fm^2 = xR,$ from which it follows that $x = zr \in \ell \fm.$ By taking the contrapositive of each equivalence, it follows that $\ell \fm$ does not contain a minimal generator of $\fm^2$ if and only if $\ell \fm \cap (\fm^2 \setminus \fm^3) = \emptyset$ if and only if $\ell \fm \subseteq \fm^3$ if and only if $\ell \in (\fm^3 : \fm).$ Ultimately, if $\ell \in \fm$ and $\ell \notin (\fm^3 : \fm)$ so that $\ell \notin \fm^2,$ then $\ell \fm$ must contain a minimal generator of $\fm^2 = xR.$

We claim that $\fm \setminus (\fm^3 : \fm) \neq \emptyset.$ On the contrary, if $\fm \setminus (\fm^3 : \fm) = \emptyset,$ then we would have that $\fm = (\fm^3 : \fm)$ so that $\fm^2 = (\fm^3 : \fm) \fm \subseteq \fm^3$ --- a contradiction. We conclude that there exists an element $\ell \in \fm \setminus (\fm^3 : \fm)$ so that $\ell \in \fm \setminus \fm^2$ and $\fm^2 = xR \subseteq \ell \fm$ --- a contradiction.
\end{proof}

\begin{cor}\label{ms is witnessed by linear forms inductive step}
Let $(R, \fm)$ be a (standard graded) local ring with $\fm^2 \neq 0.$ If $x_1, \dots, x_n \in \fm$ are (homogeneous) elements such that $\fm^2 \subseteq (x_1, \dots, x_n) R,$ then there are (homogeneous) elements $y_1, \dots, y_n \in \fm \setminus \fm^2$ such that $\fm^2 \subseteq (y_1, \dots, y_n) R.$
\end{cor}

\begin{proof}
Consider the quotient ring $\bar R = R / (x_2, \dots, x_n)R.$ Observe that $$\bar \fm^2 = {\left(\frac \fm {(x_2, \dots, x_n) R} \right)}^{\!2} = \frac{\fm^2 + (x_2, \dots, x_n)R}{(x_2, \dots, x_n)R} \subseteq \frac{(x_1, \dots, x_n) R}{(x_2, \dots, x_n) R} = \bar x_1 \bar R.$$ By Lemma \ref{ms is witnessed by linear forms base case}, there exists a (homogeneous) element $\bar y_1 \in \bar \fm \setminus \bar \fm^2$ such that $\bar \fm^2 \subseteq \bar y_1 \bar R$ so that $$\frac{\fm^2 + (x_2, \dots, x_n) R}{(x_2, \dots, x_n) R} = \bar \fm^2 \subseteq \bar y_1 \bar R = \frac{(y_1, x_2, \dots, x_n) R}{(x_2, \dots, x_n) R},$$ from which it follows that $\fm^2 + (x_2, \dots, x_n) R \subseteq (y_1, x_2, \dots, x_n) R$ and $\fm^2 \subseteq (y_1, x_2, \dots, x_n) R.$ Certainly, we must have that $y_1 \in \fm \setminus \fm^2$: for if $y_1 \in \fm^2,$ then $\bar y_1 \in \bar \fm^2$ --- a contradiction. We have therefore found $y_1 \in \fm \setminus \fm^2$ such that $\fm^2 \subseteq (y_1, x_2, \dots, x_n).$ By repeating this argument with $x_2, \dots, x_n,$ we obtain (homogeneous) elements $y_2, \dots, y_n \in \fm \setminus \fm^2$ such that $\fm^2 \subseteq (y_1, \dots, y_n) R.$
\end{proof}

\begin{cor}\label{ms equivalent to m^2 = mI base case}
Let $(R, \fm)$ be a (standard graded) local ring with $\fm^2 \neq 0.$ If there exists a (homogeneous) element $x \in \fm$ such that $\fm^2 \subseteq xR,$ then there exists a (homogeneous) element $y \in \fm \setminus \fm^2$ such that $\fm^2 = y \fm.$
\end{cor}

\begin{proof}
By Lemma \ref{ms is witnessed by linear forms base case}, there exists a (homogeneous) element $y \in \fm \setminus \fm^2$ such that $\fm^2 \subseteq yR.$ Consequently, for every (homogeneous) element $z \in \fm^2,$ there exists an element $r \in R$ such that $z = yr.$ We claim that $r \in \fm.$ If $(R, \fm)$ is local, the claim holds: if $r \notin \fm,$ then $r$ is a unit so that $y = zr^{-1}$ belongs to $\fm^2$ --- a contradiction. If $R = \bigoplus_{i \geq 0} R_i$ is graded, then there exist homogeneous elements $r_0, \dots, r_n$ such that $r = r_0 + \cdots + r_n$ and $z = y(r_0 + \cdots + r_n).$ If $z$ is homogeneous of degree $i,$ then $y(r_0 + \cdots + r_n) = yr_0 + \cdots + yr_n$ lies in $R_i,$ hence we have that $z = yr_j$ for some homogeneous element $r_j.$ Either way, we conclude that $r$ is in $\fm$ so that $\fm^2 = y \fm.$
\end{proof}

\begin{prop}\label{ms equivalent to m^2 = mI inductive step}
Let $(R, \fm)$ be a Noetherian (standard graded) local ring with $\fm^2 \neq 0$ and $\ms(R) = n.$ There exist (homogeneous) elements $y_1, \dots, y_{n} \in \fm \setminus \fm^2$ such that $\fm^2 = (y_1, \dots, y_n) \fm.$
\end{prop}

\begin{proof}
It suffices to show that there exist (homogeneous) elements $y_1, \dots, y_n \in \fm \setminus \fm^2$ such that $\fm^2 \subseteq (y_1, \dots, y_n) \fm.$ We proceed by induction on $\ms(R).$ Corollary \ref{ms equivalent to m^2 = mI base case} establishes the case that $\ms(R) = 1,$ so we may assume the claim holds for $1 \leq \ms(R) \leq n - 1.$ If $\ms(R) = n,$ by Corollary \ref{ms is witnessed by linear forms inductive step}, there exist (homogeneous) elements $x_1, \dots, x_n \in \fm \setminus \fm^2$ such that $\fm^2 \subseteq (x_1, \dots, x_n) R$ and $$\bar \fm^2 = {\left(\frac{\fm}{x_n R} \right)^{\!2}} = \frac{\fm^2 + x_n R}{x_n R} \subseteq \frac{(x_1, \dots, x_n) R + x_n R}{x_n R} = (\bar x_1, \dots, \bar x_{n - 1}) \bar R.$$ Let $\ms(\bar R) = m.$ By the above equation, we must have that $m + 1 \leq n$. By induction, there exist (homogeneous) elements $\bar y_1, \dots, \bar y_m \in \bar \fm \setminus \bar \fm^2$ such that $\bar \fm^2 = (\bar y_1, \dots, \bar y_m) \bar \fm$ so that
\begin{align*}
\frac{\fm^2 + x_n R}{x_n R} = {\left(\frac{\fm}{x_n R} \right)^{\!2}} = \bar \fm^2 &= (\bar y_1, \dots, \bar y_m) \bar \fm \\ \\
&= \frac{(y_1, \dots, y_m)R + x_n R}{x_n R} \cdot \frac{\fm + x_n R}{x_n R} \\ \\
&\subseteq \frac{(y_1, \dots, y_m) \fm + x_n \fm + x_n R}{x_n R} \\ \\
&= \frac{(y_1, \dots, y_m, x_n) \fm + x_n R}{x_n R},
\end{align*}
hence we have that $\fm^2 \subseteq (y_1, \dots, y_m, x_n) \fm + x_n R.$ We claim that $\fm^2 \subseteq (y_1, \dots, y_m, x_n) \fm.$ Observe that every (homogeneous) element $r \in \fm^2$ has the form $r = a + x_n b$ for some (homogeneous) elements $a \in (y_1, \dots, y_m, x_n) \fm$ and $b \in R.$ Considering that $\bar y_i \in \bar \fm \setminus \bar \fm^2,$ it follows that $y_i \in \fm \setminus \fm^2.$ On the contrary, if $b \notin \fm,$ then $b$ would be a unit by the previous proof so that $x_n = rb^{-1} - ab^{-1}$ belongs to $\fm^2$ --- a contradiction. We conclude that $\fm^2 \subseteq (y_1, \dots, y_m, x_n) \fm.$ Conversely, we have that $(y_1, \dots, y_m, x_n) \fm \subseteq \fm^2$ by hypothesis that $y_1, \dots, y_m, x_n \in \fm \setminus \fm^2.$ Considering that $\ms(R) = n,$ we must have that $n \leq m + 1$ so that $n = m + 1.$
\end{proof}

One immediate consequence of Proposition \ref{ms equivalent to m^2 = mI inductive step} is the following.

\begin{cor}\label{ms measures number of generators of reductions of m with reduction number 1}
If $(R, \fm)$ is a Noetherian (standard graded) local ring with $\fm^2 \neq 0,$ then $$\ms(R) = \min \{\mu(I) : I \text{ is a reduction of } \fm \text{ with reduction number one}\}.$$
\end{cor}

\begin{proof}
Proposition \ref{ms equivalent to m^2 = mI inductive step} illustrates that $\ms(R)$ is witnessed by a reduction of $\fm$ with reduction number one. Conversely, observe that if $I$ is a reduction of $\fm$ with reduction number one, then we have that $\fm^2 = \fm I \subseteq I,$ from which it follows that $\mu(I) \geq \ms(R),$ and the claim holds.
\end{proof}

In view of Proposition \ref{ms equivalent to m^2 = mI inductive step}, the following result demonstrates that $\ms(R)$ is an open condition.

\begin{prop}\label{ms is an open condition}
Let $(R, \fm, k)$ be a Noetherian local ring with $\ms(R) = n.$ The collection $${\left\{(x_1 + \fm^2, \dots, x_n + \fm^2) \in {\left(\frac \fm {\fm^2} \right)}^{\!\oplus n} : x_i \in \fm \setminus \fm^2 \text{ and } \fm^2 = (x_1, \dots, x_n)\fm \right\}}$$ is a nonempty Zariski open set, where we identify $\fm / \fm^2$ with $\mathbb A^r_k$ via a fixed generating set.
\end{prop}

\begin{proof}
By Proposition \ref{ms equivalent to m^2 = mI inductive step}, this collection is nonempty. Observe that $\fm^2 = (x_1, \dots, x_n) \fm$ if and only if the $R$-linear map $\fm^{\oplus n} \to \fm^2$ sending $(m_1, \dots, m_n) \mapsto x_1 m_1 + \cdots + x_n m_n$ is surjective if and only if the $k$-linear map $(\fm / \fm^2)^{\oplus n} \to \fm^2 / \fm^3$ sending $(m_1 + \fm^2, \dots, m_n + \fm^2) \mapsto x_1 m_1 + \cdots + x_n m_n + \fm^3$ is surjective. Explicitly, the implication of the second equivalence holds by the right exactness of the functor $k \otimes_R -$; the converse of the second equivalence holds by Nakayama's Lemma applied to the cokernel of $\fm^{\oplus n} \to \fm^2.$

Given that $\mu(\fm) = r,$ fix a generating set $\{y_1, \dots, y_r \}$ of $\fm.$ By setting $e_i$ to be the image $\bar y_i$ of $y_i$ in $\fm / \fm^2,$ we have that $\{e_1, \dots, e_r \}$ is an ordered basis of $\fm / \fm^2.$ Observe that $\{y_i y_j \mid 1 \leq i, j \leq r \}$ generates $\fm^2$ over $R,$ hence $\{e_i e_j \mid 1 \leq i, j \leq r \}$ generates $\fm^2 / \fm^3$ over $k.$ We may therefore obtain a $k$-vector space basis $\{e_{i_\ell} e_{j_\ell} \mid 1 \leq \ell \leq s \},$ where we denote $\dim_k (\fm^2 / \fm^3) = \mu(\fm^2) = s.$ Observe that we may write $\bar x_i = \sum_j \alpha_{ij} e_j$ for some coefficients $\alpha_{ij} \in k \cong \mathbb A_k^1.$ We claim that the entries of the matrix of the $k$-linear map $ (\fm / \fm^2)^{\oplus n} \to \fm^2 / \fm^3$ are polynomials in the coefficients $\alpha_{ij}.$ Considering that a $k$-linear map is uniquely determined by how it acts on a basis, it suffices to determine the images of the $e_i$ under $(\fm / \fm^2)^{\oplus n} \to \fm^2 / \fm^3.$ Observe that the $rn$ basis elements of $(\fm / \fm^2)^{\oplus n}$ are $n$-tuples $f_{ij}$ whose $i$th coordinate is $e_j = \bar y_j.$ Given that $x_i y_j$ is in $\fm^3,$ it follows that $x_i y_j + \fm^3 = 0 + \fm^3$ so that $f_{ij}$ is mapped to $0 + \fm^3.$ By hypothesis that the map is surjective, it follows that at least one basis element $f_{ij}$ is not mapped to $0 + \fm^3,$ hence we have that $$f_{ij} \mapsto x_i y_j + \fm^3 = (x_i + \fm^2)(y_j + \fm^2) = \bar x_i \bar y_j = {\left(\sum_t \alpha_{it} e_t \right)} e_j = \sum_t \alpha_{it} e_t e_j.$$ Crucially, any $e_t e_j$ that are not basis elements of $\fm^2 / \fm^3$ can be written in terms of $e_{i_\ell} e_{j_\ell},$ hence the entries of the matrix corresponding to the $k$-linear map are polynomials in $\alpha_{ij}.$ Last, the surjectivity of a $k$-linear map is an open condition: it can be viewed as a non-vanishing condition on some minors of the corresponding matrix.
\end{proof}

We show next that $\ms(R)$ and $\cs(R)$ behave well with respect to (graded) surjections.

\begin{prop}\label{comparison theorem for cs and ms under surjective ring homomorphism}
Let $(R, \fm)$ and $(S, \mathfrak n)$ be Noetherian (standard graded) local rings. If there exists a surjective (graded) homomorphism $\varphi : R \to S,$ then $\ms(R) \geq \ms(S)$ and $\cs(R) \geq \cs(S).$
\end{prop}

\begin{proof}
By hypothesis that $\varphi$ is a surjective (graded) ring homomorphism, we have that $\varphi(I)$ is a (homogeneous) ideal of $S$ for each (homogeneous) ideal $I$ of $R.$ By hypothesis that $(R, \fm)$ and $(S, \mathfrak n)$ are (standard graded) local rings, every non-unit of $S$ is the image under $\varphi$ of a non-unit of $R,$ i.e., $\varphi(\fm) = \mathfrak n.$ If $I$ witnesses $\ms(R),$ then $\mathfrak n^2 = \varphi(\fm)^2 = \varphi(\fm^2) \subseteq \varphi(I).$ By the Third Isomorphism Theorem, the local rings $R$ and $S$ have the same residue field $k,$ and the $k$-vector spaces $I / \fm I$ and $\varphi(I) / \varphi(\fm I)$ are isomorphic. We conclude that $\mu(\varphi(I)) = \mu(I) = \ms(R),$ from which it follows that $\ms(S) \leq \mu(I) = \ms(R).$ Likewise, if $J$ witnesses $\cs(R),$ then $\mathfrak n^2 = \varphi(\fm^2) = \varphi(J^2) = \varphi(J)^2$ yields that $\cs(S) \leq \mu(J) = \cs(R)$ by a similar rationale as before.
\end{proof}

Our next observation yields nice results on cutting down and adjoining indeterminates.

\begin{prop}\label{ms(R) leq ms(R/I) + mu(I)}
Let $(R, \fm)$ be a Noetherian (standard graded) local ring. Let $I$ be a (homogeneous) proper ideal of $R.$ We have that $\ms(R) \leq \ms(R/I) + \mu(I).$
\end{prop}

\begin{proof}
Let $\ms(R/I) = n.$ By Definition \ref{cs and ms definition}, there exist (homogeneous) elements $x_1, \dots, x_n \in \fm$ such that $(\fm / I)^2 \subseteq (x_1, \dots, x_n)(R/I).$ Consequently, it follows that $$\frac{\fm^2 + I}{I} = {\left(\frac \fm I \right)}^{\!2} \subseteq (x_1, \dots, x_n) \frac R I = \frac{(x_1, \dots, x_n) + I} I.$$ We conclude that $\fm^2 \subseteq \fm^2 + I \subseteq (x_1, \dots, x_n) + I$ so that $\ms(R) \leq n + \mu(I) = \ms(R/I) + \mu(I).$
\end{proof}

\begin{cor}\label{ms(R) with respect to cutting down by t quadratic forms}
Let $(R, \fm)$ be a Noetherian (standard graded) local ring. For any (homogeneous) elements $x_1, \dots, x_t \in \fm,$ we have that $\ms(R) \geq \ms(R / x_1 R) \geq \cdots \geq \ms(R / (x_1, \dots, x_t)) \geq \ms(R) - t.$
\end{cor}

\begin{proof}
We obtain the last inequality by Proposition \ref{ms(R) leq ms(R/I) + mu(I)} and the rest by Proposition \ref{comparison theorem for cs and ms under surjective ring homomorphism}.
\end{proof}

\begin{cor}\label{cutting down by (homogeneous) elements of degree one strictly decreases ms}
Let $(R, \fm)$ be a Noetherian (standard graded) local ring. If $x_1, \dots, x_{\ms(R)} \in \fm \setminus \fm^2$ witness $\ms(R),$ then $\ms(R / (x_1, \dots, x_t)) = \ms(R) - t$ for any integer $1 \leq t \leq \ms(R).$
\end{cor}

\begin{proof}
Let $\ms(R) = n.$ By Corollary \ref{ms(R) with respect to cutting down by t quadratic forms}, it suffices to show that $\ms(R / (x_1, \dots, x_t)) \leq \ms(R) - t.$ By hypothesis, for any integer $1 \leq t \leq n,$ we have that $\fm^2 + (x_1, \dots, x_t) \subseteq (x_{t + 1}, \dots, x_n) + (x_1, \dots, x_t)$ so that $$\bar \fm^2 = \frac{\fm^2 + (x_1, \dots, x_t)}{(x_1, \dots, x_t)} \subseteq \frac{(x_{t + 1}, \dots, x_n) + (x_1, \dots, x_t)}{(x_1, \dots, x_t)} = (\bar x_{t + 1}, \dots, \bar x_n) \frac R {(x_1, \dots, x_t)},$$ where $\bar x_i$ denotes the image of $x_i$ in $R / (x_1, \dots, x_t).$ We conclude that $\ms(R / (x_1, \dots, x_t)) \leq n - t.$
\end{proof}

\begin{cor}\label{ms(R) with respect to adjoining t indeterminates}
Let $(R, \fm)$ be a Noetherian (standard graded) local ring. For any indeterminates $X_1, \dots, X_t,$ we have that $\ms(R) \leq \ms(R[X_1]) \leq \cdots \leq \ms(R[X_1, \dots, X_t]) \leq \ms(R) + t.$
\end{cor}

\begin{proof}
We may identify $R$ and $R[X_1, \dots, X_t] / (X_1, \dots, X_t)$ by the First Isomorphism Theorem. Consequently, the last inequality holds by Proposition \ref{ms(R) leq ms(R/I) + mu(I)} and the others by Proposition \ref{comparison theorem for cs and ms under surjective ring homomorphism}.
\end{proof}

Conversely, if $R = k[X_1, \dots, X_t] / I$ for some homogeneous ideal $I$ of $k[X_1, \dots, X_t],$ then for any indeterminates $X_{t + 1}, \dots, X_n,$ the upper bound of Corollary \ref{ms(R) with respect to adjoining t indeterminates} is sharp.

\begin{prop}\label{ms of algebra obtained by adjoining indeterminates}
Let $R = k[X_1, \dots, X_t] / I$ for some homogeneous ideal $I$ of $k[X_1, \dots, X_t].$ Let $X_{t + 1}, \dots, X_n$ be indeterminates. We have that $\ms(R[X_{t + 1}, \dots, X_n]) = \ms(R) + n - t.$
\end{prop}

\begin{proof}
By Corollary \ref{ms of algebra obtained by adjoining indeterminates}, it suffices to show that $\ms(R[X_{t + 1}, \dots, X_n]) \geq \ms(R) + n - t.$ We will prove that there exists a homogeneous ideal $J$ that witnesses $\ms(R[X_{t + 1}, \dots, X_n])$ and a homogeneous ideal $J'$ of $R$ such that $J = J' + (X_{t + 1}, \dots, X_n)$ and $J' \supseteq (\bar X_1, \dots, \bar X_t)^2$; the claim follows.

Let $\ms(R[X_{t + 1}, \dots, X_n]) = r.$ Let $\fm = (X_1, \dots, X_t, X_{t + 1}, \dots, X_n)$ be the homogeneous maximal ideal of $k[X_1, \dots, X_t, X_{t + 1}, \dots, X_n].$ Let $\bar X_i$ denote the image of $X_i$ in $R[X_{t + 1}, \dots, X_n].$ By Corollary \ref{ms equivalent to m^2 = mI inductive step}, there exist homogeneous linear polynomials $g_1, \dots, g_r \in k[X_1, \dots, X_t, X_{t + 1}, \dots, X_n]$ such that $J = (\bar g_1, \dots, \bar g_r)$ and $\bar \fm^2 = (\bar g_1, \dots, \bar g_r) \bar \fm.$ Particularly, for each integer $t + 1 \leq i \leq n,$ there exist polynomials $p_1, \dots, p_r \in \fm$ such that $\bar X_i^2 = \bar p_1 \bar g_1 + \cdots + \bar p_r \bar g_r.$ Consequently, there exists a polynomial $q \in I[X_{t + 1}, \dots, X_n]$ such that $X_i^2 = p_1 g_1 + \cdots + p_r g_r + q.$ By hypothesis that $I$ belongs to $k[X_1, \dots, X_t],$ the polynomial $q$ does not admit any summands that are a scalar multiple of $X_i^2.$ Cancelling these summands, we may write $X_i^2 = (\alpha_1 g_1 + \cdots + \alpha_r g_r) X_i$ for some scalars $\alpha_1, \dots, \alpha_r.$ We conclude that $X_i = \alpha_1 g_1 + \cdots + \alpha_r g_r,$ hence $X_i$ belongs to $J$ for each integer $n + 1 \leq i \leq t.$ Consequently, for each integer $1 \leq i \leq r,$ we may write $g_i = h_i + f_i(X_{t + 1}, \dots, X_n)$ for some homogeneous polynomials $h_i \in k[X_1, \dots, X_t]$ and $f_i(X_{t + 1}, \dots, X_n) \in k[X_1, \dots, X_t, X_{t + 1}, \dots, X_n].$ By setting $J' = (\bar h_1, \dots, \bar h_r),$ we find that $J = J' + (X_{t + 1}, \dots, X_n)$ and $J'$ is a homogeneous ideal of $R.$

Last, we claim that $J' \supseteq (\bar X_1, \dots, \bar X_t)^2.$ By assumption that $J$ witnesses $\ms(R[X_{t + 1}, \dots, X_n]),$ for any integers $1 \leq i \leq j \leq t,$ the monomial $\bar X_i \bar X_j$ must belong to $J.$ If $\bar X_i \bar X_j$ does not vanish, then there exist polynomials $f_1, \dots, f_{r - n + t} \in J',$ $p_1, \dots, p_{r - n + t}, p_{t + 1}, \dots, p_n \in k[X_1, \dots, X_t, X_{t + 1}, \dots, X_n],$ and $q \in I[X_{t + 1}, \dots X_n]$ such that $X_i X_j = p_1 f_1 + \cdots + p_{r - n + t} f_{r - n + t} + p_{t + 1} X_{t + 1} + \cdots + p_n X_n + q.$ Cancelling any summands from the right-hand side that are not scalar multiples of $X_i X_j,$ we may write $X_i X_j = q_1 f_1 + \cdots + q_{r - n + t} f_{r - n + t}$ for some linear polynomials $q_1, \dots, q_{r - n + t} \in k[X_1, \dots, X_t, X_{t + 1}, \dots, X_n].$ We conclude that $X_i X_j$ belongs to $J'$ for any integers $1 \leq i \leq j \leq t,$ from which it follows that $J' \supseteq (\bar X_1, \dots, \bar X_t)^2,$ as desired.
\end{proof}

Going forward, it will sometimes be useful to treat the local and the standard graded local cases unilaterally. We achieve this as follows. Let $R = \bigoplus_{i \geq 0} R_i$ be a Noetherian standard graded local ring with homogeneous maximal ideal $\fm.$ Observe that $R \setminus \fm$ is the multiplicatively closed subset of $R$ consisting of nonzero elements of $R$ whose degree zero homogeneous component is nonzero. Consequently, if $R_0$ is a field, then the degree zero homogeneous component of any element of $R \setminus \fm$ is a unit. For simplicity, we will often assume that it is. By definition, we have that $\fm^2 R_\fm = \{x / s \mid x \in \fm^2 \text{ and } s \in R \setminus \fm\},$ from which one can verify that $\fm^2 R_\fm = (\fm R_\fm)^2.$ We begin by establishing that if $R_0$ is a field, then the square of the homogeneous maximal ideal of $R$ consists precisely of elements of $R$ whose degree zero and degree one components vanish.

\begin{lemma}\label{the square of the maximal ideal consists of elements whose degree zero and degree one homogeneous components vanish}
Let $R$ be a Noetherian standard graded local ring with homogeneous maximal ideal $\fm$ such that $R_0$ is a field. We have that $\fm^2 = \bigoplus_{i \geq 2} R_i.$
\end{lemma}

\begin{proof}
By definition, every element of $\fm^2$ is of the form $x_1 y_1 + \cdots + x_n y_n$ for some integer $n \geq 0$ and some elements $x_1, \dots, x_n, y_1, \dots, y_n \in \fm.$ By assumption that $R_0$ is a field, we have that $\fm = \bigoplus_{i \geq 1} R_i.$ Consequently, every nonzero element of $\fm^2$ is a sum of homogeneous elements of degree at least two, hence we have that $\fm^2 \subseteq \bigoplus_{i \geq 2} R_i.$ Conversely, suppose that $x$ is a homogeneous element of $R$ of degree at least two. By assumption that $R$ is standard graded, we have that $R = R_0[R_1],$ hence there exist integers $n, m_1, \dots, m_n \geq 0$ and elements $\alpha_1, \dots, \alpha_n \in R_0$ and $r_{1, 1}, \dots, r_{m_1, 1}, \dots, r_{1, n}, \dots, r_{m_n, n} \in R_1$ such that $x = \sum_{i = 1}^n \alpha_i r_{1, i} \cdots r_{m_i, i}.$ By assumption that $x$ is homogeneous of degree at least two, all summands on the right-hand side that lie in degree one must cancel; the rest of the summands lie in $\fm^2,$ hence $x$ lies in $\fm^2.$ We conclude that $\bigoplus_{i \geq 2} R_i \subseteq \fm^2.$
\end{proof}

We will now demonstrate that $\ms(R) = \ms(R_\fm)$ for a Noetherian standard graded local ring with homogeneous maximal ideal $\fm$ such that $R_0$ is a field. Our next lemma provides a crucial ingredient and illustrates that the generators of any ideal of $R_\fm$ that witnesses $\ms(R_\fm)$ can be replaced by the images of homogeneous elements of $R$ of degree one.

\begin{lemma}\label{every witness of ms(R_fm) is the image of a homogeneous ideal of R}
Let $R$ be a Noetherian standard graded local ring with homogeneous maximal ideal $\fm$ such that $R_0$ is a field. If the ideal generated by $x_1 / 1, \dots, x_n / 1$ witnesses $\ms(R_\fm),$ then there exist homogeneous elements $y_1, \dots, y_n \in \fm \setminus \fm^2$ such that $(y_1 / 1, \dots, y_n / 1) R_\fm$ witnesses $\ms(R_\fm).$
\end{lemma}

\begin{proof}
We proceed by induction on $\ms(R_\fm) = n.$ Consider first the case that $n = 1.$ By Corollary \ref{ms equivalent to m^2 = mI base case}, if $\fm^2 R_\fm \subseteq (x / 1) R_\fm$ for some element $x / 1 \in \fm R_\fm,$ then there exists an element $y / 1 \in \fm R_\fm \setminus \fm^2 R_\fm$ such that $\fm^2 R_\fm = (y / 1) \fm R_\fm.$ We note that $y \in \fm \setminus \fm^2,$ hence if $y = y^0 + \cdots + y^d$ is the unique representation of $y$ in terms of its homogeneous components, then $y^0 = 0$ and $y^1$ is nonzero by Lemma \ref{the square of the maximal ideal consists of elements whose degree zero and degree one homogeneous components vanish}. If $y = y^1$ is homogeneous, then we are done; if $y$ is not homogeneous, then we may define $m = \min \{i \geq 2 \mid y^i \neq 0\}.$ We note that $y^m$ belongs to $\fm^2,$ hence there exist nonzero elements $r \in \fm$ and $s, t \in R \setminus \fm$ such that $sty^m = rty.$ Comparing the homogeneous components of degree $m$ on the left- and right-hand sides and using the fact that $r^0 = 0$ and $y^i = 0$ for all $2 \leq i \leq m - 1,$ we find that $$s^0 t^0 y^m = r^0 t^0 y^m + r^{m - 1} t^0 y^1 + r^0 t^{m - 1} y^1 = r^{m - 1} t^0 y^1.$$ Considering that $s, t \in R \setminus \fm,$ we must have that $s^0$ and $t^0$ are units. We conclude that $y^m = r^{m - 1} t^0 u y^1$ lies in $y^1 R,$ where $u$ is the multiplicative inverse of $s^0 t^0.$ By induction on $i \geq m,$ we may write $s^0 t^0 y^i$ as a sum of $r^j t^k y^1$ such that $j + k = i - 1$ and $r^0 j^0 y^i.$ Observe that the former lies in $y^1 R,$ and the latter is zero. We conclude that all homogeneous components of $y$ of degree at least two lie in $y^1 R$ so that $(y / 1) R_\fm = (y^1 / 1) R_\fm.$

We will assume now that $\fm^2 R_\fm \subseteq (x_1 / 1, \dots, x_n / 1) R_\fm$ for some elements $x_1 / 1, \dots, x_n / 1 \in \fm R_\fm.$ By Proposition \ref{ms equivalent to m^2 = mI inductive step}, there exist elements $y_1 / 1, \dots, y_n / 1 \in \fm R_\fm \setminus \fm^2 R_\fm$ such that $\fm^2 R_\fm = (y_1 / 1, \dots, y_n / 1) \fm R_\fm.$ We note that $y_1, \dots, y_n \in \fm \setminus \fm^2.$ If $y_n = y_n^1$ is homogeneous, then we may proceed to the next paragraph. If $y_n$ is not homogeneous, then define $m = \min\{i \geq 2 \mid y_n^i \neq 0\}.$ Considering that $y_n^m$ belongs to $\fm^2,$ it follows that $$\frac{y_n^m} 1 = \sum_{i = 1}^n \frac{r_i y_i}{s_i} = \frac{p_1 r_1 y_1 + \cdots + p_n r_n y_n}{s_1 \cdots s_n}$$ for some nonzero elements $r_1, \dots, r_n \in \fm,$ $s_1, \dots, s_n \in R \setminus \fm,$ and $p_i = \prod_{j \neq i} s_j.$ Consequently, there exists an element $t \in R \setminus \fm$ such that $s_1 \cdots s_n t y_n^m = p_1 r_1 t y_1 + \cdots + p_n r_n t y_n.$ Comparing the homogeneous components of degree $m$ on the left- and right-hand sides and using the fact that $y_n^i = 0$ for all $2 \leq i \leq m - 1,$ we find that $$s_1^0 \cdots s_n^0 t^0 y_n^m = (p_1 r_1 t y_1 + \cdots + p_{n - 1} r_{n - 1} t y_{n - 1})^m + p_n^0 r_n^{m - 1} t^0 y_n^1,$$ as any term involving $r_n^0$ vanishes because $r_n^0 = 0.$ Observe that the element $a = p_1 r_1 t y_1 + \cdots + p_{n - 1} r_{n - 1} t y_{n - 1}$ lies in $(y_1, \dots, y_{n - 1}) R,$ $b = p_n^0 r_n^{m - 1} t^0$ lies in $R,$ and $y_n^m = ua^m + buy_n^1,$ where $u$ is the multiplicative inverse of $s_1^0 \cdots s_n^0 t^0.$ Consequently, we find that $y_n^m / 1 = ua^m / 1 + buy_n^1 / 1$ lies in $(y_1 / 1, \dots, y_{n - 1} / 1, y_n^1 / 1) R_\fm.$ By induction on $i \geq m,$ once again, we find that the image of all homogeneous components of $y_n$ of degree at least two lie in $(y_1 / 1, \dots, y_n^1 / 1) R_\fm,$ from which it follows that $(y_1 / 1, \dots, y_n / 1) R_\fm = (y_1 / 1, \dots, y_n^1 / 1) R_\fm.$ 

Observe that $R_\fm / (y_n^1 / 1) R_\fm \cong (R / y_n^1 R)_\fm$ is the localization of a Noetherian standard graded local ring whose degree zero graded piece is a field; moreover, it satisfies $\ms(R_\fm / (y_n^1 / 1) R_\fm) = n - 1$ by Corollary \ref{cutting down by (homogeneous) elements of degree one strictly decreases ms} because $y_n^1 / 1$ belongs to a minimal system of generators of an ideal that witnesses $\ms(R_\fm).$ By induction, there exist homogeneous elements $\bar z_1^1, \dots, \bar z_{n - 1}^1$ of degree one in $R / y_n^1 R$ such that $\bar \fm^2 \subseteq (\bar z_1^1 / 1, \dots, \bar z_{n - 1}^1 / 1)(R_\fm / (y_n^1 / 1) R_\fm).$ But this yields $\fm^2 R_\fm \subseteq (z_1^1 / 1, \dots, z_{n - 1}^1 / 1, y_n^1 / 1) R_\fm,$ hence $(z_1^1 / 1, \dots, z_{n - 1}^1 / 1, y_n^1 / 1) R_\fm$ witnesses $\ms(R_\fm).$
\end{proof}

\begin{prop}\label{ms(R) = ms(R_m) for a standard graded local ring}
Let $R$ be a Noetherian standard graded local ring with homogeneous maximal ideal $\fm$ such that $R_0$ is a field. We have that $\ms(R) = \ms(R_\fm).$
\end{prop}

\begin{proof}
If $I$ witnesses $\ms(R),$ then $\fm^2 \subseteq I$ and $\ms(R) = \mu(I).$ Consequently, we have that $\fm^2 R_\fm \subseteq I R_\fm,$ from which it follows that $\ms(R_\fm) \leq \mu(I R_\fm) \leq \mu(I) = \ms(R).$ Conversely, if $\ms(R_\fm) = \ell,$ then by Lemma \ref{every witness of ms(R_fm) is the image of a homogeneous ideal of R}, there exist homogeneous elements $x_1, \dots, x_\ell \in \fm \setminus \fm^2$ such that $\fm^2 R_\fm \subseteq (x_1, \dots, x_\ell) R_\fm.$ We claim that $\fm^2 \subseteq (x_1, \dots, x_\ell) R.$ It suffices to prove that the homogeneous elements of $\fm^2$ belong to $(x_1, \dots, x_\ell) R.$ If $a \in \fm^2$ is homogeneous, then as in the proof of Lemma \ref{every witness of ms(R_fm) is the image of a homogeneous ideal of R}, there exist elements $r_1, \dots, r_\ell \in R,$ $s_1, \dots, s_\ell, t \in R \setminus \fm,$ and $p_i = \prod_{j \neq i} s_j$ such that $s_1 \cdots s_\ell t a = p_1 r_1 t x_1 + \cdots + p_\ell r_\ell t x_\ell.$ Observe that $s_1^0 \cdots s_\ell^0 t^0 a$ belongs to $(x_1, \dots, x_\ell) R$ and $s_1^0 \cdots s_\ell^0 t^0$ is a unit by assumption that $R_0$ is a field, hence $a$ belongs to $(x_1, \dots, x_\ell) R.$ We conclude that the homogeneous elements of $\fm^2$ belong to $(x_1, \dots, x_\ell) R$ so that $\fm^2 \subseteq (x_1, \dots, x_\ell) R$ and $\ms(R) \leq \ell = \ms(R_\fm).$
\end{proof}

We wrap up this section by establishing that the invariants $\ms(R)$ and $\cs(R)$ do not change with respect to $\fm$-adic completion $\widehat R$ of a Noetherian (standard graded) local ring $(R, \fm).$ Recall that $\widehat M \cong M \otimes_R \widehat R$ for any finitely generated $R$-module $M.$ Because $\widehat R$ is faithfully flat over $R,$ we have that $M = 0$ if and only if $\widehat M = 0.$ Likewise, for any ideals $I_1 \subseteq I_2,$ we have that $I_1 = I_2$ if and only if $\widehat{I_1} = \widehat{I_2}.$

\begin{prop}\label{cs and ms are invariant under completion}
Let $(R, \fm)$ be a Noetherian (standard graded) local ring with $\fm$-adic completion $\widehat R.$ We have that $\cs(R) = \cs(\widehat R)$ and $\ms(R) = \ms(\widehat R).$
\end{prop}

\begin{proof}
For any (homogeneous) ideal $I$ of $R,$ we have that $\mu(I) = \mu(I \widehat R)$ and $I^2 \widehat R = (I \widehat R)^2.$ Consequently, if $I^2 = \fm^2,$ then the exposition preceding the statement of proposition implies that $(I \widehat R)^2 = I^2 \widehat R = \fm^2 \widehat R = (\fm \widehat R)^2.$ Likewise, if $\fm^2 \subseteq I,$ then $(\fm \widehat R)^2 = \fm^2 \widehat R \subseteq I \widehat R.$ We conclude that $\cs(\widehat R) \leq \cs(R)$ and $\ms(\widehat R) \leq \ms(R).$

We will establish now that $\cs(R) \leq \cs(\widehat R)$ and $\ms(R) \leq \ms(\widehat R).$ Observe that the canonical injection $R \to \widehat R$ induces an isomorphism $R / \fm^2 \cong \widehat R / (\fm \widehat R)^2.$ By the Fourth Isomorphism Theorem, for any ideal $J$ of $\widehat R$ with $J^2 = (\fm \widehat R)^2 = \fm^2 \widehat R,$ there exists an ideal $I$ of $R$ such that $I\supseteq \fm^2$ and $J = I \widehat R.$ Considering that $I^2 \widehat R = (I \widehat R)^2 = J^2 = \fm^2 \widehat R$ and $I^2 \subseteq \fm^2,$ we conclude that $I^2 = \fm^2$ so that $\cs(R) \leq \cs(\widehat R).$ Likewise, for any ideal $L$ of $\widehat R$ with $L^2 \supseteq (\fm \widehat R)^2 = \fm^2 \widehat R,$ there exists an ideal $K$ of $R$ such that $K \supseteq \fm^2$ and $L = K \widehat R$ so that $\ms(R) \leq \ms(\widehat R).$
\end{proof}

\section{General bounds on \texorpdfstring{$\ms(R)$}{ms(R)} and \texorpdfstring{$\cs(R)$}{cs(R)}}\label{General bounds on ms(R) and cs(R)}

Primarily, we devote this section to providing bounds for $\ms(R)$ and $\cs(R).$ By the proof of Proposition \ref{ms is an open condition}, we obtain an immediate lower bound for $\ms(R).$

\begin{cor}\label{initial lower bound for ms and comparison theorem for mu(m^n) for ms(R) = 1}
Let $(R, \fm)$ be a Noetherian local ring. We have that $$\ms(R) \geq {\left \lceil \dfrac{\mu(\fm^2)}{ \mu(\fm)} \right \rceil}.$$ Even more, if $\ms(R) = 1,$ then we have that $\mu(\fm^{n+1}) \leq \mu(\fm^n)$ for all integers $n \geq 1.$
\end{cor}

\begin{proof}
Let $\ms(R) = r.$ By Proposition \ref{ms equivalent to m^2 = mI inductive step}, there exist elements $y_1, \dots, y_r \in \fm \setminus \fm^2$ such that $\fm^2 = (y_1, \dots, y_r) \fm.$ Consequently, we have that $\mu(\fm^2) = \mu((y_1, \dots, y_r) \fm) \leq r \mu(\fm)$ or $$\ms(R) = r \geq {\left \lceil \dfrac{\mu(\fm^2)}{\mu(\fm)} \right \rceil}.$$ If $\ms(R) = 1,$ we have that $\fm^2 = \ell \fm$ for some $\ell \in \fm \setminus \fm^2$ so that $\ell \fm^2 = (\ell \fm) \fm = \fm^3.$ Continuing in this manner shows that $\ell \fm^n = \fm^{n + 1}$ and $\mu(\fm^{n + 1}) = \mu(\ell \fm^n) \leq \mu(\fm^n)$ for all integers $n \geq 1.$ 
\end{proof}

Our next proposition establishes more general bounds for $\ms(R)$ and $\cs(R).$ We illustrate moreover that if $R$ is a Cohen-Macaulay local ring, then $\ms(R)$ and $\cs(R)$ are as small as possible if and only if $R$ exhibits ``nice'' properties. Before this, we need the following lemma.

\begin{lemma}\label{I^2 = mI implies R is regular}
Let $(R, \fm)$ be a Cohen-Macaulay local ring of positive dimension $d.$ If there exist an integer $n \geq 0$ and an $\fm$-primary ideal $(x_1, \dots, x_d)$ of $R$ such that $(x_1, \dots, x_d)^{n + 1} = \fm (x_1, \dots, x_d)^n,$ then $R$ is regular. 
\end{lemma}

\begin{proof}
If $n = 0,$ the assertion is trivial, so we may assume that $n \geq 1$ and $I = (x_1, \dots, x_d).$ By hypothesis, the elements $x_1, \dots, x_d$ form a system of parameters that is an $R$-regular sequence, as $R$ is Cohen-Macaulay. Consequently, the map $(R/I)[X_1, \dots, X_d] \to \bigoplus_{k \geq 0} I^k / I^{k + 1}$ that sends $X_i$ to the image of $x_i$ in $I / I^2$ is an isomorphism. Particularly, $I^n / I^{n+1}$ is isomorphic to the degree $n$ graded piece of $(R/I)[X_1, \dots, X_d],$ which is isomorphic to $(R/I)^{\oplus \binom{n + d - 1} n}$. By assumption that $I^{n + 1} = \fm I^n,$ we have that $I^n / \fm I^n \cong (R/I)^{\oplus \binom{n + d - 1} n}$ so that $\mu(I^n) = \binom{n + d - 1}n \ell_R(R/I)$ by taking length on both sides. Considering that $\mu(I^n) \leq \binom{n + \mu(I) - 1} n \leq \binom{n + d - 1} n,$ we find that $\binom{n + d - 1} n \ell_R(R/I) \leq \binom{n + d - 1} n$ and $\ell_R(R/I) \leq 1 = \ell_R(R / \fm).$ On the other hand, the inclusion $I \subseteq \fm$ induces a surjection $R/I \to R / \fm,$ hence the additivity of length on short exact sequences yields $\ell_R(R/I) \geq \ell_R(R / \fm).$ We conclude that $\ell_R(R/I) = \ell_R(R / \fm)$ so that $\ell_R(\fm / I) = 0$ and $I = \fm,$ i.e., $R$ is regular. 
\end{proof}

\begin{prop}\label{main proposition about cs and ms}
Let $(R, \fm)$ be a Noetherian (standard graded) local ring.
\begin{enumerate}[\rm (1.)] 

\item We have that $\dim R \leq \ms(R) \leq \cs(R) \leq \mu(\fm).$ Particularly, if $R$ is regular, these invariants are all equal.

\vspace{0.25cm}

\item If $R$ is a Cohen-Macaulay local ring with infinite residue field, then $\ms(R) = \dim R$ if and only if $R$ has minimal multiplicity.

\vspace{0.25cm}

\item If $R$ is local Cohen-Macaulay of $\dim R > 0,$ then $\cs(R) = \dim R$ if and only if $R$ is regular.

\vspace{0.25cm} 

\item We have that $\ms(R) = \mu(\fm)$ if and only if $\fm I = \fm^2$ implies $I = \fm$ for any ideal $I$ of $R.$

\vspace{0.25cm}

\item We have that $\cs(R) = \mu(\fm)$ if and only if $I^2 = \fm^2$ implies $I = \fm$ for any ideal $I$ of $R.$

\end{enumerate}
\end{prop}

\begin{proof}
(1.) By definition of $\cs(R),$ we have that $\cs(R) \leq \mu(\fm).$ If $I$ witnesses $\cs(R),$ then $I^2 = \fm^2$ so that $\fm^2 \subseteq I$ and $\ms(R) \leq \cs(R).$ If $I$ witnesses $\ms(R),$ then $\fm^2 \subseteq I \subseteq \fm$ so that $\sqrt I = \fm$ and $\height(I) = \height(\fm).$ Further, we have that $\height(I) \leq \mu(I) = \ms(R)$ by Krull's Height Theorem. Combining these two observations gives that $\dim R = \height(\fm) = \height(I) \leq \ms(R).$ If $R$ is regular, then $\dim R = \mu(\fm)$ and the invariants are all equal.

(2.) We will assume that the residue field $k$ of $R$ is infinite. By \cite[Exercise 4.6.14]{BH}, $R$ has minimal multiplicity if and only if $\fm^2 = (x_1, \dots, x_n) \fm$ for some $R$-regular sequence $(x_1, \dots, x_n).$ If $R$ has minimal multiplicity, therefore, there exists an $R$-regular sequence $(x_1, \dots, x_n)$ such that $\fm^2 \subseteq (x_1, \dots, x_n)$ and $n \leq \dim R.$ By definition of $\ms(R)$ and (1.), we find that $n \leq \dim R \leq \ms(R) \leq n,$ from which it follows that $\ms(R) = \dim R.$ Conversely, we will assume that $\ms(R) = \dim R.$ By Proposition \ref{ms is witnessed by linear forms}, there exists an ideal $I$ such that $\fm^2 \subseteq I$ and $\mu(I) = \mu_1(I) = \ms(R) = \dim R.$ By Proposition \ref{cs is witnessed by linear forms}, we have that $I \cap \fm^2 = \fm I$ so that $\fm^2 = \fm I.$ Considering that $\mu(I) = \ms(R) = \dim R$ and $\fm^2 \subseteq I \subseteq \fm,$ we have that $I$ is a parameter ideal, hence $I$ is generated by a regular sequence since $R$ is Cohen-Macaulay. Consequently, $R$ has minimal multiplicity.

(3.) If $R$ is regular, the claim holds. Conversely, we will assume that $\cs(R) = \dim R = d > 0.$ By definition of $\cs(R),$ there exists an ideal $I$ such that $I^2 = \fm^2$ and $\mu(I) = \dim R.$ Considering that $\fm^2 = I^2 \subseteq \fm I \subseteq \fm^2,$ we find that $I^2 = \fm I,$ hence $I$ is $\fm$-primary. By Lemma \ref{I^2 = mI implies R is regular} with $n = 1,$ we conclude that $R$ is regular.

(4.) We will assume first that $\fm I = \fm^2$ implies that $I = \fm$ for any (homogeneous) ideal $I.$ If $\ms(R) = n,$ there exists a (homogeneous) ideal $J$ with $\mu(J) = n$ and $\fm^2 = \fm J$ by Proposition \ref{ms equivalent to m^2 = mI inductive step}. By hypothesis, we conclude that $J = \fm$ so that $\ms(R) = \mu(J) =\mu(\fm).$ Conversely, assume that $\ms(R) = \mu(\fm).$ Given a (homogeneous) ideal $I$ such that $\fm^2 = \fm I,$ we have an inclusion of $k$-vector spaces $I / \fm I = I / \fm^2 \subseteq \fm / \fm^2$ such that $\dim_k(I / \fm I) = \ms(R) = \mu(\fm) = \dim_k(\fm / \fm^2).$ Consequently, we must have that $I / \fm^2 = \fm / \fm^2$ so that $I = \fm.$

(5.) If $I^2 = \fm^2$ implies that $I = \fm$ for any (homogeneous) ideal $I,$ then the set of (homogeneous) ideals of $R$ whose square is $\fm^2$ is $S = \{\fm\}.$ We conclude that $\cs(R) = \min \{\mu(I) \mid I \in S\} = \mu(\fm).$ We will assume therefore that $\cs(R) = \mu(\fm).$ Given an ideal $I$ such that $I^2 = \fm^2,$ we have that $$\dim_k(I/\fm^2) = \dim_k(I/\fm I) = \mu(I) \geq \cs(R) = \mu(\fm) = \dim_k(\fm/\fm^2).$$ Considering that $I/\fm^2$ is a $k$-subspace of $\fm/\fm^2,$ we conclude that $I/\fm^2 = \fm/\fm^2$ and $I = \fm.$
\end{proof}

\begin{remark}\label{necessity of certain assumptions}
Based on Proposition \ref{main proposition about cs and ms}(2.), if $R$ is Cohen-Macaulay with infinite residue field and $R$ does not have minimal multiplicity, then $\ms(R) > \dim R.$ One such family of Cohen-Macaulay local rings is given by $\CC \fps[x, y, z] / (x^n + y^n + z^n)$ for any integer $n \geq 3.$ Generally, for a regular local ring $(S, \mathfrak n)$ and any nonzero element $s \in \mathfrak n,$ we have that $e(S/(s)) = \min \{m \mid s \in \fm^m \setminus \fm^{m + 1}\}$ so that $\ms(S/(s)) = \dim(S/(s))$ if and only if $e(S/(s)) \leq 2$ if and only if $s \notin \mathfrak n^3$ by Corollary \ref{cs and ms for hypersurface} and \cite[Example 11.2.8]{HS}. Considering that $R$ is a complete intersection with unique maximal ideal $\fm = (\bar x, \bar y, \bar z)$ and $x^n + y^n + z^n \in (x,y,z)^3$ by assumption that $n \geq 3,$ we have that $\ms(R) > \dim R.$


Even more, it is possible for $\dim R < \ms(R) < \cs(R) < \mu(\fm)$ to hold simultaneously. Consider the Artinian complete intersection ring $R = k[x,y,z] / (x^2, y^2, z^2)$ with $\fm = (\bar x, \bar y, \bar z)$ over a field $k$ with $\chara(k) \neq 2.$ Observe that $I = (\bar x + \bar y + \bar z)$ satisfies $\fm^2 \subseteq I,$ from which it follows that $\ms(R) = 1.$ We demonstrate in Proposition \ref{cs of k[x_1, ..., x_n]/(x_1^2, ..., x_n^2)} that $\cs(R) = 2.$ Evidently, we have that $\mu(\fm) = 3.$

Last, the Cohen-Macaulay assumption in the third part of Proposition \ref{main proposition about cs and ms} is necessary. Consider the ring $R = k[x, y] / (x^2, xy).$ Observe that $I = (\bar y)$ witnesses both $\ms(R)$ and $\cs(R)$ since we have that $\bar \fm^2 = (\bar x, \bar y)^2 = (\bar y^2),$ from which it follows that $\dim R = \ms(R) = \cs(R) = 1 < 2 = \mu(\fm)$ so that $R$ is not regular; however, both $\bar x$ and $\bar y$ are zero divisors in $R,$ hence we have that $\depth(R) = 0,$ i.e., $R$ is not Cohen-Macaulay.
\end{remark}

Our next corollary improves upon Corollary \ref{initial lower bound for ms and comparison theorem for mu(m^n) for ms(R) = 1} in the case that $R$ is a one-dimensional Cohen-Macaulay local ring with infinite residue field such that $\ms(R) = 1.$

\begin{cor}\label{1-dimensional Cohen-Macaulay local ring with ms(R) = 1}
Let $(R, \fm)$ be a one-dimensional Cohen-Macaulay local ring with infinite residue field. If $\ms(R) = 1,$ we have that $\mu(\fm) = \mu(\fm^{n + 1})$ for all integers $n \geq 1.$
\end{cor} 

\begin{proof}
By Proposition \ref{main proposition about cs and ms}, it follows that $R$ has minimal multiplicity. Consequently, we have that $\fm^2 = x \fm$ for some $R$-regular element $x$ and $\fm^{n+1} = x^n \fm \cong \fm$ for all integers $n \geq 1.$
\end{proof} 

By Proposition \ref{main proposition about cs and ms}, we can explicitly compute $\ms(R)$ and $\cs(R)$ when $(R, \fm)$ is a hypersurface.

\begin{cor}\label{cs and ms for hypersurface}
If $(R, \fm)$ is a hypersurface and $\fm^2 \neq 0,$ then $\cs(R) = \mu(\fm).$ Further, if $e(R) \leq 2,$ then we have that $\ms(R) = \dim R;$ otherwise, we have that $\ms(R) = \mu(\fm).$
\end{cor}

\begin{proof}
If $R$ is regular, then we have that $\dim R = \ms(R) = \cs(R) = \mu(\fm),$ and the claim holds. We may assume therefore that $R$ is not regular. By hypothesis that $R$ is a hypersurface, we have that $R$ is Cohen-Macaulay and $\dim R \geq \mu(\fm) - 1 .$ If $\dim R = 0,$ it follows that $\mu(\fm) = 1$ by assumption that $\fm^2 \neq 0$ so that $\cs(R) \leq \mu(\fm) = 1.$ On the other hand, we have that $\cs(R) \geq 1$ since $\fm^2 \neq 0,$ hence $\cs(R) = 1.$ We will assume therefore that $\dim R \geq 1.$ By Proposition \ref{main proposition about cs and ms}, we have that $\dim R \leq \cs(R) \leq \mu(\fm)$ so that $\cs(R) = \mu(\fm)$ or $\cs(R) = \mu(\fm) - 1$; however, the latter cannot happen, as it would imply that $R$ is regular by the third part of Proposition \ref{main proposition about cs and ms} --- a contradiction. We conclude that $\cs(R) = \mu(\fm).$

By hypothesis that $R$ is a hypersurface, we have that $\ms(R) = \mu(\fm) - 1$ or $\ms(R) = \mu(\fm).$ Considering that $R$ is Cohen-Macaulay, it follows that $\ms(R) = \dim R$ if and only if $R$ has minimal multiplicity, i.e., if and only if $e(R) = \mu(\fm) - \dim R + 1 \leq 2.$
\end{proof}

\begin{cor}\label{ms for a hypersurface of positive dimension}
Let $(R, \fm)$ be a Noetherian (standard graded) local ring. If $R$ is a hypersurface and $\dim(R) > 0,$ then $\ms(R) \in \{\mu(\fm) - 1, \mu(\fm)\}$; the latter happens if and only if $\fm I = \fm^2$ implies that $I = \fm$ for any ideal $I.$
\end{cor}

\begin{proof}
If $R$ is a hypersurface of positive dimension, then Corollary \ref{cs and ms for hypersurface} implies that $\cs(R) = \mu(\fm).$ By Proposition \ref{main proposition about cs and ms}, we have that $\mu(\fm) - 1 \leq \dim R \leq \ms(R) \leq \mu(\fm)$ so that $\ms(R) = \mu(\fm) - 1$ or $\ms(R) = \mu(\fm)$; the latter happens if and only if $\fm I = \fm^2$ implies that $I = \fm.$
\end{proof}

One crucial point in Corollary \ref{cs and ms for hypersurface} is the assumption that $R$ is Cohen-Macaulay. If $\mu(\fm) = \dim R + 1,$ both $\ms(R)$ and $\cs(R)$ take values in $\{\dim R, \dim R + 1 \};$ however, it might not be easy to conclude which value each invariant takes if $R$ is not Cohen-Macaulay. Bearing this in mind, it is desirable to seek some additional properties of $R$ along with $\mu(\fm) = \dim R + 1$ that guarantee $R$ is Cohen-Macaulay; the following proposition addresses this concern. We note that this is known, but we provide the statement and proof for reference.

\begin{prop}\label{mu(m) = dim R + 1 implies Cohen-Macaulay in some cases}
Let $(R, \fm)$ be a Noetherian local ring with $\mu(\fm) = \dim R + 1.$ If the $\fm$-adic completion $\widehat R$ is an integral domain, then $R$ is Cohen-Macaulay. Particularly, if the associated graded ring $\gr_\fm(R)$ is an integral domain, then $R$ is Cohen-Macaulay.
\end{prop}

\begin{proof}
Observe that $\mu(\widehat{\fm}) - 1 = \mu(\fm) - 1 = \dim R = \dim \widehat R.$ By the Cohen Structure Theorem, we can write $\widehat R \cong S/J$ for some ideal $J$ in a regular local ring $S$ such that $\dim S = \mu(\fm).$ By hypothesis that $\widehat R$ is an integral domain, we must have that $J$ is a prime ideal of $S$ with $\height(J) = \dim S - \dim(S/J) = \mu(\fm) - \dim R = 1.$ Considering that $S$ is a UFD, we have that $J$ is principal, hence $\widehat R$ is Cohen-Macaulay so that $R$ is Cohen-Macaulay. Particularly, if $\gr_\fm(R)$ is an integral domain, then $\widehat R$ is also a domain, hence the last claim follows.
\end{proof} 

\begin{remark}
It cannot be concluded (even in the equicharacteristic case) that $R$ is Cohen-Macaulay if we only assume that $(R, \fm)$ is a domain with $\mu(\fm) = \dim R + 1.$ Counterexamples exist even in dimension two. We will construct an example of such by employing \cite[Theorem 1]{L}, a discussion of which can be found in \cite{dJ}.

Given a field $k,$ consider the complete Noetherian local ring $S = k \fps[x, y, z] / (xy, xz)$ with unique maximal ideal $\mathfrak n = (\bar x, \bar y, \bar z).$ Observe that $S$ is reduced but not equidimensional, hence we have that $\depth(S) = 1.$ Considering that $S$ contains the field $k,$ the prime ring of $S$ is either $\ZZ$ or $\ZZ / p \ZZ,$ and its action on $S$ is torsion-free. By \cite[Theorem 1]{L}, it follows that $S$ is the completion of a local domain $(R, \fm).$ We have conclude that $\dim R = \dim S = 2$ and $\mu(\fm) = \mu(\mathfrak n) = 3 = \dim R + 1.$ Because $\widehat R$ is not Cohen-Macaulay, $R$ is not Cohen-Macaulay. We must therefore assume that $R$ is a hypersurface in Corollary \ref{cs and ms for hypersurface}, i.e., $R$ must be Cohen-Macaulay.
\end{remark}

If the square of the maximal ideal of a local ring $(R, \fm)$ is minimally generated by ``few enough'' elements, then we obtain an upper bound for $\ms(R)$ as follows.

\begin{prop}\label{upper bound for ms in terms of upper bound for mu(m^2)}
Let $(R, \fm, k)$ be a Noetherian (standard graded) local ring with infinite residue field $k$ with $\fm^2 \neq 0.$ If $\mu(\fm^2) < \binom{r + 2} r,$ then $\ms(R) \leq r.$ Particularly, if $\mu(\fm^2) \leq 2,$ then $\ms(R) = 1.$ 
\end{prop}

\begin{proof}
We obtain this result in the local case as a corollary to the main theorem of \cite{ES} by taking $I = \fm$ and $n = 2$ (cf. \cite[Theorem 8.6.8]{HS}). If $\mu(\fm^2) < \binom{r+2}{r},$ then by the aforementioned theorem, there exist linear forms $x_1, \dots, x_r$ such that $\fm^2 = (x_1, \dots, x_r) \fm \subseteq (x_1, \dots, x_r)$ so that $\ms(R) \leq r.$ Given that $\mu(\fm^2) \leq 2 < 3 = \binom{1+2}{1},$ we may set $r = 1$ to obtain $\ms(R) \leq 1.$ We conclude that $\ms(R) = 1$ by assumption that $\fm^2 \neq 0.$

If $R$ is standard graded local and $R_0$ is a field, the result holds by Proposition \ref{ms(R) = ms(R_m) for a standard graded local ring}.
\end{proof}

\begin{cor}\label{upper bound for ms in terms of upper bound for mu(m^2) 2}
If $(R, \fm, k)$ is a Noetherian (standard graded) local ring with infinite residue field $k$ such that $\mu(\fm^2) < \binom{\mu(\fm) + 1} 2,$ then $\ms(R) \leq \mu(\fm) - 1.$ If $\dim(R) = \mu(\fm) - 1,$ equality holds.
\end{cor}

\begin{proof}
Using $r = \mu(\fm) - 1$ and the fact that $\binom{\mu(\fm) + 1} {\mu(\fm) - 1} = \binom{\mu(\fm) + 1} 2,$ we obtain the first claim from Proposition \ref{upper bound for ms in terms of upper bound for mu(m^2)}; the second claims follows from first part of Proposition \ref{main proposition about cs and ms}.
\end{proof}

We note that it is always true that $\mu(\fm^2) \leq \binom{\mu(\fm) + 1} 2.$ In fact, for most of the rings that we will consider in this paper, strict inequality holds because some element of $\fm^2$ vanishes.

Observe that when $\dim R$ is small --- especially when $R$ is Artinian --- it is less restrictive to assume that $\mu(\fm^2)$ is small than to assume that $\mu(\fm)$ is small, as $\mu(\fm^2)$ can be small when $\mu(\fm)$ is arbitrarily large. Our next two propositions deal with cases when $\mu(\fm^2)$ is small. We remark that Proposition \ref{upper bound for ms in terms of upper bound for mu(m^2)} guaranteed that $\ms(R) = 1$ whenever $\mu(\fm^2) \leq 2$ without any further assumptions on $R.$ Even though we will mainly focus on $\cs(R)$ in the following two propositions, we do make mention of $\ms(R)$ without resorting to Proposition \ref{upper bound for ms in terms of upper bound for mu(m^2)}.

\begin{prop}\label{cs and ms when u(m^2) is small}
Let $(R, \fm, k)$ be a Noetherian local ring such that $\fm^2 \neq 0.$
\begin{enumerate}[\rm (1.)] 

\item If $\cs(R) = 1,$ then $\mu (\fm^2) = 1.$ If $\dim R = 1,$ then $e(R) = 1.$ If $R$ is Cohen-Macaulay, then $R$ is regular.

\vspace{0.25cm}

\item Conversely, if $\mu(\fm^2) = 1,$ then $\ms(R) = 1$ and $\dim R \in \{0, 1\}.$ If $\dim R = 1$ and $R$ is Cohen-Macaulay, then $R$ is regular and $\cs(R) = 1.$ On the other hand, if $\dim R = 0$ and 2 is a unit in $R,$ then $\cs(R) = 1.$

\end{enumerate}
\end{prop}

\begin{proof}
(1.) Consider a (homogeneous) ideal $I$ that witnesses $\cs(R).$ We have that $\mu(I) = 1$ and $\mu(I^2) = 1.$ By definition of $\cs(R),$ we have that $I^2 = \fm^2$ so that $\mu(\fm^2) = \mu(I^2) = 1,$ as desired.

Consequently, it follows that $\mu(\fm^{2n}) = 1$ for all integers $n \geq 1.$ By Krull's Height Theorem, we have that $\dim R = \height(\fm) = \height(\fm^2) \leq \mu(\fm^2) = 1.$ If $\dim R = 1,$ we have that $e(R) = \mu(\fm^n)$ for all $n \gg 0,$ from which it follows that $e(R) = 1.$ Last, if $R$ is Cohen-Macaulay, then $R$ is regular by Abhyankar's Inequality since it is a Cohen-Macaulay local ring of multiplicity $1.$

(2.) Considering that $\ms(R) \leq \mu(\fm^2) = 1,$ we find that $\ms(R) = 1$ by assumption that $\fm^2 \neq 0.$ By Krull's Height Theorem, we have that $\dim R \leq \mu(\fm^2) = 1,$ hence $\dim R = 0$ or $\dim R = 1.$

If $\dim R = 1,$ then $e(R) = \mu(\fm^n)$ for all $n \gg 0.$ Considering that $\mu(\fm^{2n}) = 1$ for all integers $n \geq 1,$ we have that $e(R) = 1.$ By assumption that $R$ is Cohen-Macaulay, we have that $R$ is regular. By the first part of Proposition \ref{main proposition about cs and ms}, we have that $\cs(R) = \mu(\fm) = \dim R = 1.$

Last, suppose that $\dim R = 0,$ i.e., $R$ is Artinian. By assumption that $\mu(\fm^2) = 1,$ i.e., that $\fm^2$ is principal, we have that $R$ is stretched in the sense of \cite{Sal80} and \cite{Sal79}. We will adopt the notation of the former so that $\ell_R(R) = e,$ $\mu(\fm) = e - h,$ and $\dim_k(0 : \fm) = r.$ By definition of stretched, we have that $\fm^{h + 1} = 0.$ By assumption that $\fm^2 \neq 0,$ we have that $h \geq 2.$ Consider the following cases. 
\begin{enumerate}[\rm (a.)]

\item If $h > 2,$ then following the discussion preceding \cite[Theorem 1]{Sal80}, we have that $\fm^2 = z^2 R$ for some element $z \in \fm \setminus \fm^2.$ Consequently, we have that $\cs(R) = 1.$

\vspace{0.25cm}

\item On the other hand, if $h = 2,$ then once again by the exposition preceding \cite[Theorem 1]{Sal80}, for all indices $i, j \in \{1, \dots, e - h - r + 1\},$ there exist elements $z_i, z_j \in \fm \setminus \fm^2$ such that either $z_i z_j = 0$ or $\fm^2 = z_i z_j R.$ If $z_i^2 = z_i z_i \neq 0$ for some index $i,$ then we must have that $\fm^2 = z_i^2 R.$ Otherwise, we have that $z_i^2 = 0$ for all indices $i.$ By the aforementioned discussion in \cite{Sal80}, for each index $i,$ there exists an index $j$ such that $z_i z_j \neq 0,$ from which it follows that $\fm^2 = z_i z_j R$ so that $(z_i + z_j)^2 = 2z_i z_j R = z_i z_j R = \fm^2.$
\end{enumerate}
Either way, we have that $\cs(R) = 1.$
\end{proof}

\begin{prop}\label{cs and ms when u(m^2) = 2}
Let $(R, \fm, k)$ be a Noetherian local ring such that $\mu(\fm^2) = 2.$ We have that $\dim R = 0$ or $\dim R = 1.$ Even more, the following properties hold.
\begin{enumerate}[\rm (1.)] 

\item If $\dim R = 0$ and $\chara(k) = 0,$ then $\cs(R) = 2.$

\vspace{0.25cm}

\item If $\dim R = 1,$ then $e(R) \leq 2.$ Further, if $R$ is Cohen-Macaulay, then $R$ has minimal multiplicity $e(R) = \mu(\fm) = 2,$ $\ms(R) = 1,$ and $\cs(R) = 2.$

\end{enumerate}
\end{prop}

\begin{proof}
By Krull's Height Theorem, we have that $\dim R = \height(\fm) = \height(\fm^2) \leq \mu(\fm^2) = 2.$ On the contrary, suppose that $\dim R = 2.$ By the first corollary of \cite{Sal75}, we have that $\mu(\fm^{n + 2}) \leq 2$ for all integers $n \geq 0.$ Evidently, then, we have that $\mu(\fm^t) \leq 2$ for all $t \gg 0.$ We have therefore that $e(R) = \lim_{n \to \infty} \mu(\fm^n)/n = 0$ --- a contradiction. We conclude that $\dim R = 0$ or $\dim R = 1.$

(1.) If $\dim R = 0,$ then $R$ is Artinian and almost stretched in the sense of \cite{EV}. By \cite[Proposition 2.3]{EV}, we have that $\fm^2 = (x^2, xy)$ for some elements $x,y \in \fm \setminus \fm^2.$ Consider the ideal $I = (x,y) \subseteq \fm.$ Observe that $\fm^2 = (x^2, xy) \subseteq (x^2, xy, y^2) = I^2$ so that $I^2 = \fm^2$ and $\cs(R) \in \{1, 2\}.$ If $\cs(R) = 1,$ then by Proposition \ref{cs and ms when u(m^2) is small}, we would have that $\mu(\fm^2) = 1$ --- a contradiction --- so $\cs(R) = 2.$

(2.) If $\dim R = 1,$ then we have that $e(R) = \mu(\fm^n)$ for all $n \gg 0.$ Considering that $\mu(\fm^n) \leq 2$ for all $n \gg 0$ by our exposition in the first paragraph, we have that $e(R) \leq 2.$

Further, if $R$ is Cohen-Macaulay, then we must have that $e(R) = 2.$ For if $e(R) = 1,$ then it would follow that $R$ is regular so that $\mu(\fm) = \dim R = 1$ and thus $\mu(\fm^2) = 1$ --- a contradiction. Consequently, we have that $\mu(\fm) = \mu(\fm) - \dim R + 1 \leq e(R) = 2.$ Considering that $\mu(\fm) \neq 1,$ we conclude that $\mu(\fm) = 2 = e(R)$ so that $R$ has minimal multiplicity. By the second part of Proposition \ref{main proposition about cs and ms}, it follows that $\ms(R) = \dim R = 1.$ By the first part of Proposition \ref{main proposition about cs and ms}, we have that $1 = \dim R \leq \cs(R) \leq \mu(\fm) = 2.$ If $\cs(R) = 1,$ then once again, by Proposition \ref{cs and ms when u(m^2) is small}, we would have that $\mu(\fm^2) = 1$ --- a contradiction --- so we conclude that $\cs(R) = 2.$
\end{proof}

We conclude this section with a proposition concerning the fiber product.

\begin{prop}\label{cs and ms of fiber product}
Consider Noetherian local rings $(S, \fm_S, k)$ and $(T, \fm_T, k)$ with the same residue field $k.$ We have that $\max \{\cs(S), \cs(T)\} \leq \cs(S \times_k T) \leq \cs(S) + \cs(T)$ and $\ms(S \times_k T) = \max\{\ms(S), \ms(T)\}.$
\end{prop} 

\begin{proof}
Consider the canonical surjections $\pi_S : S \to k$ and $\pi_T : T \to k.$ We define the fiber product of $S$ and $T$ as the subset of $S \times T$ consisting of all pairs $(a, b)$ of $S \times T$ whose images are equal under the respective canonical surjections, i.e., we have that $$S \times_k T \stackrel{\text{def}} = \{(a,b)\in S \times T \mid \pi_S(a) = \pi_T(b)\}.$$ Observe that $S \times_k T$ is a local subring of $S \times T$ with unique maximal ideal $\fm_S \oplus \fm_T.$ We will denote $R = S \times_k T$ when convenient. By definition of $R,$ the maps $p_S : S \times_k T \to S$ and $p_T : S \times_k T \to T$ defined by $p_S(a, b) = a$ and $p_T(a, b) = b$ give rise to the following commutative diagram.
\begin{center}
\begin{tikzcd}
S\times_k T \arrow[r, "p_S"] \arrow[d, "p_T"] & S \arrow[d, "\pi_S"] \\
T \arrow[r, "\pi_T"'] & k
\end{tikzcd}
\end{center}
Considering that $\pi_S$ and $\pi_T$ are surjections and $\pi_S \circ p_S$ and $\pi_T \circ p_T$ are surjective by definition of $S \times_k T,$ we have that $p_S$ and $p_T$ are surjective. Consequently, it follows that $\ms(S \times_k T) \geq \ms(S)$ and $\ms(S \times_k T) \geq \ms(T)$ by Proposition \ref{comparison theorem for cs and ms under surjective ring homomorphism} so that $\max\{\ms(S), \ms(T)\} \leq \ms(S \times_k T).$ Likewise, the same holds for $\cs(S \times_k T).$

Given any ideals $I \subseteq \fm_S$ and $J \subseteq \fm_T$ of $S$ and $T,$ respectively, observe that $I \oplus J$ is an ideal of $R,$ as we have that $\pi_S(I) = 0 = \pi_T(J).$ We claim that $\mu(I \oplus J) \leq \mu(I) + \mu(J).$ If $I = (s_1, \dots, s_m)S$ and $J = (t_1, \dots, t_n)T,$ then for any element $(a,b) \in I \oplus J,$ there exist elements $a_1, \dots, a_m \in S$ and $b_1, \dots, b_n \in T$ such that $$(a,b) = {\left(\sum_{i = 1}^m a_i s_i, \sum_{j = 1}^n b_j t_j \right)} = \sum_{i = 1}^m (a_i s_i, 0) + \sum_{j = 1}^n (0, b_j t_j).$$ Certainly, we have that $(a_i s_i, 0) = (a_i, 1)(s_i, 0) = (a_i, 0)(s_i, 0)$ as elements of $S \times T$; however, this may not hold in $S \times_k T$ because there is no guarantee that $\pi_S(a_i) \in \{0, 1\}.$ Luckily, there is no issue, as we may employ the following trick: for each integer $1 \leq i \leq m,$ we have that $\pi_S(a_i)$ belongs to the residue field $k = T / \fm_T,$ hence we may find an element $a_i' \in T$ such that $\pi_S(a_i) = \pi_T(a'_i).$ Likewise, for each integer $1 \leq j \leq n,$ we may find an element $b_j' \in S$ such that $\pi_S(b_j') = \pi_T(b_j),$ $\pi_T(b_j)$ belongs to $k = S / \fm_S.$ We can therefore write $$(a,b) = \sum_{i = 1}^m (a_i s_i, 0) + \sum_{j = 1}^n (0, b_j t_j) = \sum_{i = 1}^m (a_i, a_i')(s_i, 0) + \sum_{j = 1}^n (b_j', b_j)(0, t_j),$$ hence $(a, b)$ belongs to the ideal $K$ of $R$ generated by $\{(s_1, 0), \dots, (s_m, 0), (0, t_1), \dots, (0, t_n)\}.$ We conclude that $I \oplus J \subseteq K.$ Conversely, the generators of $K$ belong to $I \oplus J.$ We have therefore shown that $I \oplus J = K$ so that $\mu(I \oplus J) = \mu(K) \leq m + n = \mu(I) + \mu(J).$ 

Now, if $I$ witnesses $\cs(S)$ and $J$ witnesses $\cs(T),$ then we have that $\fm_S^2 = I^2$ and $\fm_T^2 = J^2$ so that $(\fm_S \oplus \fm_T)^2 = \fm_S^2 \oplus \fm_T^2 = I^2 \oplus J^2 = (I \oplus J)^2$ and $\cs(S \times_k T) \leq \cs(S) + \cs(T).$

On the other hand, assume $I$ witnesses $\ms(S),$ $J$ witnesses $\ms(T),$ and $n = \max\{\ms(S), \ms(T)\}.$ Denote by $I = (s_1, \dots, s_n)S$ and $J = (t_1, \dots, t_n)T.$ Considering that $I \subseteq \fm_S$ and $J \subseteq \fm_T,$ we have that $s_i \in \fm_S$ and $t_j \in \fm_T$ so that $(s_i, t_j) \in R$ for all integers $1 \leq i,j \leq n.$ We claim that $K = ((s_1, t_1), (s_2, t_2), \dots, (s_n, t_n))R$ contains $\fm_R^2,$ hence $\ms(R) \leq n = \max\{\ms(S), \ms(T)\}.$ Observe that $\fm_R^2 = \fm^2_S \oplus \fm^2_T,$ so we may consider $(x,y) \in \fm_S^2 \oplus \fm_T^2.$ Considering that $\fm_S^2 = \fm_S I$ and $\fm_T^2 = \fm_T J$ by Proposition \ref{ms equivalent to m^2 = mI inductive step}, we have that $x \in \fm_S I$ and $y \in \fm_T J.$ By definition, there exist some elements $x_1, \dots, x_n \in \fm_S$ and $y_1, \dots, y_n \in \fm_T$ such that $x = \sum_{i = 1}^n s_i x_i$ and $y = \sum_{i = 1}^n t_i y_i.$ Consequently, the pairs $(x_i, 0)$ and $(0, y_i)$ lie in $R$ and $$(x, y) = \sum_{i = 1}^n (x_i, 0)(s_i, t_i) + \sum_{i = 1}^n (0, y_i) (s_i, t_i)$$ lies in $K.$ We conclude therefore that $\fm_R^2 \subseteq K,$ as desired.
\end{proof}

\section{The standard graded local case and the Weak Lefschetz Property}\label{The standard graded local case and the Weak Lefschetz Property}

We turn our attention to the case that $(R, \fm)$ is a standard graded local ring with unique homogeneous maximal ideal $\fm.$ Ultimately, we will consider the case that $R$ is the quotient of a polynomial ring over a field by a homogeneous ideal, e.g., a quadratic (monomial) ideal.

Given an ideal $I$ of a Noetherian local ring $(R, \fm),$ recall that the Rees algebra of $I$ in $R$ is defined by $R[It] = \oplus_{n = 0}^\infty I^n t^n \subseteq R[t].$ We note that $R[It]$ is likewise Noetherian. We define the special fiber ring of $I$ in $R$ as $\mathfrak F_I(R) = R[It] / \fm R[It].$ Observe that $\mathfrak F_\fm(R) \cong \gr_\fm(R).$

One naturally wonders how the invariants $\ms(R)$ and $\cs(R)$ behave with respect to the associated graded ring of $R.$ Unfortunately, as our next proposition illustrates, it is difficult to say.

\begin{prop}\label{ms(gr_m(R)) leq ms(R) when gr_m(R) is a domain}
Let $(R, \fm)$ be a Noetherian local ring. If the associated graded ring $\gr_\fm(R)$ is an integral domain and $\ms(R) = 1,$ then we have that $\ms(\gr_\fm(R)) \leq \ms(R).$
\end{prop}

\begin{proof}
We will henceforth denote $\ms(R) = n$ and the unique homogeneous maximal ideal $$\widetilde{\fm} = \bigoplus_{i \geq 1} \frac{\fm^i}{\fm^{i + 1}}$$ of $\gr_\fm(R).$ By the multiplication defined on $\gr_\fm(R),$ we have that $$\widetilde{\fm}^2 = \bigoplus_{i \geq 2} \frac{\fm^i}{\fm^{i + 1}.}$$ We denote by $r^*$ the initial form of $r$ in $\gr_\fm(R),$ i.e., $r^* = \{r + \fm^n / \fm^{n + 1} \mid r \in \fm^n \setminus \fm^{n + 1} \}.$

We will assume that $\ms(R) = 1.$ By Proposition \ref{ms is witnessed by linear forms}, there exists a linear form $\ell$ such that $\fm^2 \subseteq \ell R \subseteq \fm.$ We claim that $\widetilde \fm^2 \subseteq \ell^* \gr_\fm(R).$ Certainly, it is enough to show the inclusion for all homogeneous elements of $\widetilde \fm^2,$ hence we may consider some element $r^* \in \widetilde \fm^2$ with $r \in \fm^2.$ By hypothesis that $\fm^2 \subseteq \ell R,$ we have that $r = s \ell$ for some element $s \in R.$ By taking initial forms on both sides, we have that $r^* = (s \ell)^* = s^* \ell^*$ is in $\ell^* \gr_\fm(R)$ by hypothesis that $\gr_\fm(R)$ is a domain. We conclude that $\widetilde \fm^2 \subseteq \ell^* \gr_\fm(R)$ so that $\ms(\gr_\fm(R)) \leq \ms(R).$
\end{proof}

Observe that $\ms(R) \leq \mu(\fm^2)$ holds in general. On the other hand, if $\gr_\fm(R)$ has positive depth, then we have that $\cs(R) \leq \mu(\fm^2).$ Before we establish this, we record the following lemmas.

\begin{lemma}\label{u(I^k) leq u(I^k + 1) when depth F_I(R) geq 1}
If $\depth \mathfrak F_I(R) \geq 1,$ then $\mu(I^k) \leq \mu(I^{k+1})$ for each positive integer $k.$ Further, if equality holds for some positive integer $k,$ it must be the case that $\height(I) \leq 1.$
\end{lemma}

\begin{proof}
If $\depth \mathfrak F_I(R) \geq 1,$ there is a linear form $x \in \mathfrak F_I(R)$ that is not a zero divisor on $\mathfrak F_I(R).$ Consequently, multiplication by $x$ induces an injective map $I^k/\fm I^k \to I^{k + 1}/\fm I^{k + 1},$ hence we have that $\mu(I^k) \leq \mu(I^{k + 1}).$ If $\mu(I^k) = \mu(I^{k + 1}),$ then $I^k/\fm I^k \to I^{k + 1}/\fm I^{k + 1}$ is surjective, from which it follows that $I^{k + 1} = x I^k$ and $I^{k + 1} \subseteq xR.$ We conclude that $\height(I) = \height(I^{k + 1}) \leq \height(xR) = 1.$
\end{proof}

\begin{lemma}\label{regular element of gr_m(R) is regular element of R}
Let $(R, \fm)$ be a Noetherian local ring. Given any element $x \in R,$ denote by $\bar x$ the image of $x$ in $\gr_\fm(R).$ If $\bar x$ is not a zero divisor on $\gr_\fm(R),$ then $x$ is not a zero divisor on $R.$
\end{lemma}

\begin{proof}
We will establish the contrapositive. Consider some nonzero element $y \in R$ such that $xy = 0.$ By Krull's Intersection Theorem, it follows that $\bar y$ is nonzero in $\gr_\fm(R);$ however, we have that $\bar x \bar y = \overline{xy} = \bar 0$ in $\gr_\fm(R),$ hence $\bar x$ is a zero divisor on $\gr_\fm(R).$
\end{proof}

\begin{prop}\label{cs(R) leq u(m^2) when depth gr_m(R) geq 1}
Let $(R, \fm)$ be a Noetherian local ring. We have that $\ms(R) \leq \min \{\mu(\fm), \mu(\fm^2)\}.$ If we have that $\depth \gr_\fm(R) \geq 1,$ then $\cs(R) \leq \mu (\fm^2).$ If equality holds here, then both $R$ and $\gr_\fm(R)$ are one-dimensional Cohen-Macaulay local rings, and $R$ has minimal multiplicity. Even more, if the stronger equality $\ms(R) = \mu(\fm^2)$ holds, then $R$ is regular of dimension one.
\end{prop}

\begin{proof}
By definition, we have that $\ms(R) = \min \{\mu (I) \mid I \subseteq \fm \text{ is an ideal of } R \text { and } \fm^2 \subseteq I \subseteq \fm \},$ from which it follows that $\ms(R) \leq \mu(\fm)$ and $\ms(R) \leq \mu(\fm^2)$ or $\ms(R) \leq \min\{\mu(\fm), \mu(\fm^2)\}.$

We will assume throughout the rest of the proof that $\depth \gr_\fm(R) \geq 1.$ By Lemma \ref{u(I^k) leq u(I^k + 1) when depth F_I(R) geq 1}, it follows that $\mu(\fm) \leq \mu(\fm^2)$ so that $\cs(R) \leq \mu(\fm^2)$ by the first part of Proposition \ref{main proposition about cs and ms}.

If $\cs(R) = \mu(\fm^2),$ then $\mu(\fm^2) = \cs(R) \leq \mu(\fm)$ by the first part of Proposition \ref{main proposition about cs and ms}, from which it follows that $\mu(\fm) = \mu(\fm^2)$ by the previous paragraph. By Lemma \ref{u(I^k) leq u(I^k + 1) when depth F_I(R) geq 1}, we conclude that $\dim R = \height(\fm) \leq 1.$ Considering that $\dim R = \dim \gr_\fm(R) \geq \depth \gr_\fm(R) \geq 1$ by assumption, we conclude that $\dim R = 1$ so that $\dim \gr_\fm(R) = 1$ and $\depth \gr_\fm(R) = 1,$ i.e., $\gr_\fm(R)$ is Cohen-Macaulay of dimension one; it remains to be seen that $\depth R = 1.$

By the proof of Lemma \ref{u(I^k) leq u(I^k + 1) when depth F_I(R) geq 1}, we have that $\fm^2 = x \fm$ for some element $x \in R$ whose image in $\gr_\fm(R)$ is a linear form and hence is neither a unit in $\gr_\fm(R)$ nor a zero divisor on $\gr_\fm(R).$ By Lemma \ref{regular element of gr_m(R) is regular element of R}, we have that $x$ is not a zero divisor on $R,$ and since $x$ is not a unit in $R,$ we have that $\depth R \geq 1.$ We conclude that $\depth R = 1$ so that $R$ is Cohen-Macaulay of dimension one. Observe that this shows that $R$ has minimal multiplicity since $\fm^2 = x \fm$ for some $R$-regular element $x.$ By the second part of Proposition \ref{main proposition about cs and ms}, we conclude that $\ms(R) = \dim R.$

Last, if $\ms(R) = \mu(\fm^2),$ then $\mu(\fm^2) = \ms(R) \leq \cs(R) \leq \mu(\fm^2)$ so that $\cs(R) = \mu(\fm^2).$ By the previous paragraphs, $R$ is Cohen-Macaulay, $\dim R = 1,$ and $\mu(\fm) = \mu(\fm^2)$ so that $\dim R = \ms(R) = \mu(\fm^2) = \mu(\fm).$
\end{proof}

We will assume for the rest of this section that $R = k[x_1, \dots, x_n] / J$ is a standard graded Artinian $k$-algebra with unique maximal ideal $\fm.$ We will express $J$ explicitly when necessary. Out of desire for notational convenience, we will simply write an element of $R$ as $f$ as opposed to $\bar f.$ Further, we will denote by $[R]_i$ the $i$th graded piece of $R,$ i.e., the $k$-vector subspace generated by the monomials of $k[x_1, \dots, x_n] / J$ of degree $i.$

\begin{definition}\label{Weak Lefschetz Property definition}
We say that $R$ enjoys the {\it Weak Lefschetz Property} (henceforth abbreviated WLP) if for any general linear form $\ell,$ the multiplication map $$\cdot \ell : [R]_i \to [R]_{i + 1}$$ has maximal rank, i.e., it is either injective or surjective.
\end{definition}

If $R$ has Hilbert function $(1, h_1, h_2, \dots, h_e),$ then \cite[Lemma 2.9]{MN} guarantees that $R$ enjoys the WLP if and only if for any general linear form $\ell$ and all indices $i,$ we have that $$h_i(R / (\ell)) = \max \{h_i - h_{i-1}, 0 \},$$ where $h_i = \dim_k [R]_i = \mu(\fm^i)$ and $h_i(R / (\ell))$ is the $i$th entry of the Hilbert function of $R / (\ell).$

Consider a homogeneous ideal $I$ of $R$ that witnesses $\ms(R).$ By Proposition \ref{ms is witnessed by linear forms}, $I$ can be generated by linear forms, i.e., there exist $\ell_1, \dots, \ell_{\ms(R)} \in I \setminus \fm^2$ such that $I = (\ell_1, \dots, \ell_{\ms(R)}).$ We will denote by $R_i = R / (\ell_1, \dots, \ell_i)$ and by $h_j(i)$ the $j$th Hilbert coefficient of $R_i.$ We say that $R$ {\it enjoys the Weak Lefschetz Property at the $i$th step} if $R_i$ enjoys the WLP, hence $R$ enjoys the WLP at each step $1 \leq i \leq \ms(R)$ whenever $R_i$ enjoys the WLP for each $1 \leq i \leq \ms(R).$

\begin{prop}\label{ms(R) if R enjoys the WLP}
If $R$ enjoys the WLP at each step, then $\ms(R) = \min \{i \mid h_2(i) = 0 \}.$
\end{prop}

\begin{proof}
If $h_2(i) = 0,$ then $\fm^2 \subseteq (\ell_1, \dots, \ell_i)$ so that $\ms(R) \leq \min \{i \mid h_2(i) = 0 \}.$ On the other hand, by Proposition \ref{ms is witnessed by linear forms}, $\ms(R)$ is the least positive integer $i$ with $\fm^2 \subseteq (\ell_1, \dots, \ell_i),$ i.e., $h_2(i) = 0.$
\end{proof}

\begin{prop}\label{ms(R) if R enjoys the WLP and u(m^2) leq u(m)}
If $R$ enjoys the WLP and $\mu(\fm^2) \leq \mu(\fm),$ then $\ms(R) \leq 1.$ Conversely, if $\ms(R) \leq 1,$ then $R$ enjoys the WLP and $\mu(\fm^{i + 1}) \leq \mu(\fm^i)$ for all integers $i \geq 1.$
\end{prop}

\begin{proof}
We will assume first that $R$ enjoys the WLP and $\mu(\fm^2) \leq \mu(\fm).$ Certainly, if $\fm^2 = 0,$ then it follows that $\ms(R) = 0,$ as the zero ideal witnesses $\ms(R).$ We may assume therefore that $\fm^2 \neq 0.$ By hypothesis that $R$ enjoys the WLP, for any general linear form $\ell,$ the Hilbert function of $R / (\ell)$ is given by $(1, \mu(\fm) - 1, \max \{\mu(\fm^2) - \mu(\fm), 0\}, \dots).$ Considering that $\mu(\fm^2) \leq \mu(\fm),$ we have that $\mu(\fm^2) - \mu(\fm) \leq 0$ so that the Hilbert function of $R / (\ell)$ is in fact $(1, \mu(\fm) - 1, 0, \dots).$ We conclude by Proposition \ref{ms(R) if R enjoys the WLP} that $\ms(R) = 1.$

Conversely, suppose that $\ms(R) \leq 1.$ If $\ms(R) = 0,$ then $\fm^i = 0$ for all integers $i \geq 2,$ so assume that $\ms(R) = 1.$ By Proposition \ref{ms is witnessed by linear forms}, we have that $\fm^2 \subseteq \ell R \subseteq \fm$ for some general linear form $\ell.$ Given any minimal generator $f$ of $\fm^2,$ there exists an element $g \in \fm$ such that $f = g \ell.$ Consequently, the multiplication map $\cdot \ell : [R]_1 \to [R]_2$ is surjective. By \cite[Proposition 2.6(a)]{MN}, it follows that the multiplication map $\cdot \ell : [R]_i \to [R]_{i + 1}$ is surjective for all integers $i \geq 1,$ hence $R$ enjoys the WLP. Further, we have that $\mu(\fm^{i + 1}) = \dim_k [R]_{i+1} \leq \dim_k [R]_i = \mu(\fm^i)$ for all $i \geq 1.$
\end{proof}

\begin{prop}\label{ms(R) if R enjoys the WLP and u(m^2) > u(m)}
Let $R$ enjoy the WLP, and suppose that $\mu(\fm^2) > \mu(\fm).$
\begin{enumerate}[\rm (1.)]

\item If $\mu(\fm) = 2,$ we have that $\mu(\fm^2) = 3$ and $\ms(R) = 2.$

\vspace{0.25cm}

\item If $\mu(\fm) \geq 3$ and $\mu(\fm^2) - \mu(\fm) \in \{1,2\},$ we have that $\ms(R) = 2.$

\end{enumerate}
\end{prop}

\begin{proof}
(1.) By our hypotheses that $R$ enjoys the WLP and $\mu(\fm) = 2,$ given any general linear form $\ell,$ the Hilbert function of $R / (\ell)$ is given by $(1, 1, \mu(\fm^2) - 2, \dots).$ Consequently, the image $\bar \fm$ of $\fm$ in $R / (\ell)$ is principal, i.e., we have that $\bar \fm = (\bar f)$ for some element $\bar f \in R / (\ell).$ But this implies that $\bar \fm^2 = (\bar f^2)$ so that $\bar \fm^2$ is either principal or the zero ideal, i.e., $\mu(\fm^2) - 2 = 1$ or $\mu(\fm^2) - 2 = 0.$ By hypothesis that $\mu(\fm^2) > \mu(\fm) = 2,$ the latter cannot happen, hence we conclude that $\mu(\fm^2) = 3,$ and the Hilbert function of $R / (\ell)$ is given by $(1, 1, 1, \dots).$ Observe that $\min \{i \mid h_i(1) \leq i \} = 1,$ hence \cite[Theorem 5]{MZ07} implies that $R / (\ell)$ enjoys the WLP. Given any linear form $\ell'$ in $R$ such that $\bar{\ell'}$ is a general linear form in $R / (\ell),$ then, the Hilbert function of $R / (\ell, \ell')$ is given by $(1, 0, 0, \dots).$ We conclude by Proposition \ref{ms(R) if R enjoys the WLP} that $\ms(R) = 2.$

(2.) By our hypothesis that $R$ enjoys the WLP and $\mu(\fm^2) - \mu(\fm) \in \{1,2\},$ given any general linear form $\ell,$ the Hilbert function of $R / (\ell)$ is given by either $(1, \mu(\fm) - 1, 1, \dots)$ or $(1, \mu(\fm) - 1, 2, \dots).$ Consequently, we have that $h_2(1) \leq 2.$ By hypothesis that $\mu(\fm) \geq 3,$ we have that $\mu(\fm) - 1 \geq 2,$ and we conclude that $\min \{i \mid h_i(1) \leq i\} = 2.$ By \cite[Theorem 5]{MZ07}, $R / (\ell)$ enjoys the WLP if and only if $h_0(1) = 1 = ((h_1(1))_{(1)})_{-1}^{-1}$ (cf. \cite[Definition-Remark 1]{MZ07}). By definition, the unique 1-binomial expansion of $h_1(1)$ is $\binom{h_1(1)}{1},$ hence we have $$((h_1(1)_{(1)})_{-1}^{-1} = {\!\left(\binom{h_1(1)}{1}\!\right)}_{\!-1}^{\!-1} = \binom{h_1(1) - 1}{0} = 1 = h_0(1),$$ and $R / (\ell)$ enjoys the WLP. Given any linear form $\ell'$ in $R$ such that $\bar{\ell'}$ is a general linear form in $R / (\ell),$ the Hilbert function of $R / (\ell, \ell')$ is either $(1, \mu(\fm) - 2, 0, \dots)$ or $(1, \mu(\fm) - 3, 0, \dots)$ by assumption that $\mu(\fm) \geq 3.$ Either way, Proposition \ref{ms(R) if R enjoys the WLP} gives that $\ms(R) = 2.$
\end{proof}

We turn our attention to the Artinian complete intersection $R = k[x_1, \dots, x_n] / (x_1^2, \dots, x_n^2)$ over a field $k$ of characteristic zero. By \cite[Theorem 1.1]{MN}, $R$ enjoys the Weak Lefschetz Property.

\begin{conjecture}\label{k[x_1, ..., x_n]/(x_1^2, ..., x_n^2) enjoys the WLP}
If $R = k[x_1, \dots, x_n] / (x_1^2, \dots, x_n^2)$ and $\chara(k) = 0,$ then $R_i = R / (\ell_1, \dots, \ell_i)$ enjoys the Weak Lefschetz Property for any linear forms $\ell_1, \dots, \ell_i$ in $R.$
\end{conjecture}

\begin{prop}\label{ms if k[x_1, ..., x_n]/(x_1^2, ..., x_n^2) enjoys WLP}
If Conjecture \ref{k[x_1, ..., x_n]/(x_1^2, ..., x_n^2) enjoys the WLP} holds, then $$\ms(R) = {\left \lceil \frac{1}{2} (2n + 1 - \sqrt{8n + 1}) \right \rceil}.$$
\end{prop}

\begin{proof}
By \cite[Lemma 2.9]{MZ}, we may focus our attention on the coefficients $h_2(i).$ If Conjecture \ref{k[x_1, ..., x_n]/(x_1^2, ..., x_n^2) enjoys the WLP} holds, then $R_i$ enjoys the Weak Lefschetz Property for each integer $1 \leq i \leq \ms(R).$ Consequently, we have that $h_0(i) = 1,$ $h_1(i) = n - i,$ and $h_2(i) = h_2(i-1) - h_1(i-1)$ with $h_2(0) = h_2 = \binom{n}{2}$ and $h_1(0) = h_1 = n.$ One can check that $h_2(i) = \binom{n}{2} - in + \binom{i}{2} = \frac{1}{2}[i^2- (2n+1)i + n(n-1)]$ and $$\min \{i \mid h_2(i) = 0 \} = {\left \lceil \frac{1}{2} (2n + 1 - \sqrt{8n + 1}) \right \rceil}.$$ Our proof is complete by Proposition \ref{ms(R) if R enjoys the WLP}.
\end{proof}

\begin{remark}
We note that the integer in Proposition \ref{ms if k[x_1, ..., x_n]/(x_1^2, ..., x_n^2) enjoys WLP} is precisely the number of non-triangular numbers that do not exceed $n,$ according to the \cite[OEIS]{OEIS}. 
\end{remark}

\begin{comment}
By Wolfram Alpha, the only integer solutions of this are $i = \frac{m^2 - m}{2}$ and $n = \frac{m^2 + m}{2}.$
\end{comment}

Observe that $\gr_\fm(R)$ is a standard graded $k$-algebra, hence we may write $\gr_\fm(R)$ as the quotient of the polynomial ring $S = k[x_1, \dots, x_{\mu(\fm)}]$ by a homogeneous ideal $I.$ Let $\bar x_i$ denote the image of $x_i$ modulo $I.$ Our next proposition reduces our study to a polynomial ring modulo an ideal generated by quadratic forms.

\begin{prop}\label{cs and ms depend only on degree two}
Let $S$ and $I$ be defined as above. We have that $\cs(S/I) = \cs(S/I_2)$ and $\ms(S/I) = \ms(S/I_2),$ where $I_2$ is the ideal generated by the elements of $I$ of degree two.
\end{prop}

\begin{proof}
Considering that $I_2 \subseteq I,$ it follows that the canonical projection $S/I_2 \to S/I$ is a surjective graded ring homomorphism. Consequently, Proposition \ref{comparison theorem for cs and ms under surjective ring homomorphism} guarantees that $\cs(S/I) \leq \cs(S/I_2)$ and $\ms(S/I) \leq \ms(S/I_2).$ Conversely, we will assume that $\bar J$ witnesses $\ms(S/I)$ with $\mu(\bar J) = \ms(S/I) = k.$ By Proposition \ref{ms is witnessed by linear forms}, there exist linear forms $\bar \ell_1, \dots, \bar \ell_k$ in $S/I$ such that $\bar J = (\bar \ell_1, \dots, \bar \ell_k).$ Given that $\bar \fm = (\bar x_1, \dots, \bar x_m),$ we have that $$\frac{(x_1, \dots, x_m)^2 + I} I = \frac{\fm^2 + I} I = \bar \fm^2 \subseteq \bar J = \frac{J + I}{I} = \frac{(\ell_1, \dots, \ell_k) + I} I$$ so that $(x_1, \dots, x_m)^2 + I \subseteq (\ell_1, \dots, \ell_k) + I.$ Given any generator $x_i x_j$ of $\fm^2,$ write $x_i x_j = a_1 \ell_1 + \cdots + a_k \ell_k + s$ for some elements $a_e$ in $S$ and $s$ in $I.$ Express $s$ in terms of its homogeneous components $s = s_0 + s_1 + \cdots + s_d,$ where each element $s_f$ is homogeneous of degree $f.$ Comparing degrees shows that $a_1 \ell_1 + \cdots + a_k \ell_k + s = b_1 \ell_1 + \cdots + b_k \ell_k + s'$ for some elements $b_e$ in $S$ and $s'$ in $I$ such that the $b_e \ell_e$ and $s'$ are homogeneous of degree two. Consequently, $x_i x_j$ is an element of $(\ell_1, \dots, \ell_k) + I_2$ so that $(x_1, \dots, x_m)^2 \subseteq (\ell_1, \dots, \ell_k) + I_2$ and $$\frac{\fm^2 + I_2}{I_2} = \frac{(x_1, \dots, x_m)^2 + I_2}{I_2} \subseteq \frac{(\ell_1, \dots, \ell_k) + I_2}{I_2} = \frac{J + I_2}{I_2}.$$ We conclude that $\ms(S/I) \geq \ms(S/I_2),$ from which it follows that $\ms(S / I) = \ms(S / I_2).$ Likewise, if $\bar J$ witnesses $\cs(S/I)$ with $\mu(\bar J) = \cs(S / I) = k,$ by Proposition \ref{cs is witnessed by linear forms}, there exist linear forms $\bar \ell_1, \dots, \bar \ell_k$ such that $\bar J = (\bar \ell_1, \dots, \bar \ell_k).$ Given that $\bar \fm = (\bar x_1, \dots, \bar x_m),$ we have $$\frac{(x_1, \dots, x_m)^2 + I}{I} = \frac{\fm^2 + I}{I} = \bar \fm^2 = \bar J^2 = \frac{J^2 + I}{I} = \frac{(\ell_1, \dots, \ell_k)^2 + I}{I}$$ so that $(x_1, \dots, x_m)^2 + I = (\ell_1, \dots, \ell_k)^2 + I.$ Given any generator $x_i x_j$ of $\fm^2,$ write $x_i x_j = s + \sum_{e, f} a_{ef} \ell_e \ell_f$ for some elements $a_{ef}$ in $S$ and $s$ in $J.$ Once again, comparing the degrees on the left- and right-hand sides gives that $s + \sum_{e, f} a_{ef} \ell_e \ell_f = s' + \sum_{e, f} b_{ef} \ell_e \ell_f$ for some elements $b_{ef}$ in $S$ and $s'$ in $I$ such that $b_{ef} \ell_e \ell_f$ and $s'$ are homogeneous of degree two. Like before, we conclude that $\cs(S/I) \geq \cs(S / I_2)$ so that $\cs(S/I) = \cs(S/I_2).$
\end{proof}

One natural curiosity that arises when studying $\ms(R)$ and $\cs(R)$ for a standard graded $k$-algebra $R$ is whether these invariants depend on the field $k.$ Our next proposition provides a partial answer and describes the behavior of $\ms(R)$ and $\cs(R)$ with respect to field extensions.

\begin{prop}\label{comparison theorem for cs and ms under field extension}
Consider an injective ring homomorphism $\iota : K \to L$ of fields. If $I$ is a monomial ideal of $R = K[x_1, \dots, x_n]$ and $S = L[x_1, \dots, x_n],$ then $\cs(S/I) \leq \cs(R/I)$ and $\ms(S/I) \leq \ms(R/I).$
\end{prop}

\begin{proof}
We may view any monomial ideal $I$ of $R$ as the monomial ideal of $S$ generated by the same monomials as in $R.$ Every element $a$ of $K$ may be identified with the element $\iota(a)$ of $L,$ hence we may simply write $a$ in place of the element $\iota(a)$ of $L.$ We will denote by $\bar \fm$ the image of the homogeneous maximal ideal $\fm = (x_1, \dots, x_n)$ of $R$ (or $S$) in the quotient ring $R/I$ (or $S/I$).

By Proposition \ref{cs is witnessed by linear forms}, there exist elements $\bar \ell_1, \dots, \bar \ell_n$ in $\bar \fm \setminus \bar \fm^2$ such that $\bar \fm^2 = J^2$ for the ideal $J = (\bar \ell_1, \dots, \bar \ell_n)R$ and $\cs(R/I) = n.$ Consequently, for every pair of integers $1 \leq i \leq j \leq n,$ there exist polynomials $f_1, \dots, f_n$ of $R/I$ such that $\bar x_i \bar x_j = f_1 \bar \ell_1 + \cdots + f_n \bar \ell_n.$ Considering this as an identity in $S/I,$ we conclude that $J = (\bar \ell_1, \dots, \ell_n) S$ satisfies $\bar \fm^2 = J^2$ so that $\cs(S/I) \leq \cs(R/I).$

By Proposition \ref{ms is witnessed by linear forms}, the invariant $\ms(R/I)$ is likewise be witnessed by linear forms, hence by a similar argument as above, we conclude that $\ms(S/I) \leq \ms(R/I).$
\end{proof}

\begin{remark}
For any two algebraically closed fields $K$ and $L$ of the same characteristic, either $K$ embeds into $L$ or vice-versa; in this case, the key hypothesis of Proposition \ref{comparison theorem for cs and ms under field extension} is satisfied.
\end{remark}

We conclude this section with a discussion of $\ms(R)$ and $\cs(R)$ for two-dimensional Veronese subrings. Let $R = k[x, y].$ Recall that for a positive integer $n,$ the monomial subring $R^{(n)} = k[x^i y^{n - i} \mid 0 \leq i \leq n]$ is called the $n$th {\it Veronese} subring of $R.$ Observe that $R$ is integral over $R^{(n)},$ hence we have that $\dim R^{(n)} = \dim R = 2.$ Further, $R^{(n)}$ is a standard graded local ring with homogeneous maximal ideal $\fm = (x^i y^{n - i} \mid 0 \leq i \leq n).$

\begin{prop}\label{cs and ms of Veronese subring}
Let $k$ be a field, and let $n \geq 1$ be an integer. Let $R^{(n)}$ denote the $n$th Veronese subring of $R = k[x, y],$ and let $[n] \cup \{0\} = \{0, 1, \dots, n\}.$ We have that $\ms(R^{(n)}) = 2$ and $$\cs(R^{(n)}) \leq \min \{\lvert S \rvert : S \subseteq [n] \cup \{0\} \text{ and } S + S = [2n] \cup \{0\} \}.$$
\end{prop}

\begin{proof}
Observe that $\fm^2 = (x^{i + j} y^{2n - i - j} \mid 1 \leq i \leq j \leq n) = (x^\ell y^{2n - \ell} \mid 0 \leq \ell \leq 2n).$ Clearly, we have that $([n] \cup \{0\}) + \{0, n\} = \{x + y \mid x \in [n] \cup \{0\} \text{ and } y \in \{0, n\}\} = [2n] \cup \{0\}$ so that $$\fm^2 = (x^\ell y^{2n - \ell} \mid 0 \leq \ell \leq 2n) = (x^{i + j} y^{2n - i - j} \mid i \in [n] \text{ and } j \in \{0, n\}) = \fm (x^n, y^n) \subseteq (x^n, y^n).$$ We conclude that $\ms(R^{(n)}) \leq 2.$ Considering that $R$ is integral over $R^{(n)},$ we have that $\ms(R^{(n)}) \geq 2$ by the first part of Proposition \ref{main proposition about cs and ms}, hence equality holds.

On the other hand, if $S \subseteq [n] \cup \{0\}$ satisfies $S + S = [2n] \cup \{0\},$ then $I = (x^i y^{n - i} \mid i \in S)$ satisfies $$I^2 = (x^{i + j} y^{2n - i - j} \mid i, j \in S) = (x^\ell y^{2n - \ell} \mid \ell \in S + S) = (x^\ell y^{2n - \ell} \mid 0 \leq \ell \leq 2n) = \fm^2,$$ from which it follows that $\cs(R^{(n)}) \leq \mu(I) \leq \lvert S \rvert.$ Consequently, we conclude that
\begin{align*}
\cs(R^{(n)}) &\leq \min \{\lvert S \rvert : S \subseteq [n] \cup \{0\} \text{ and } S + S = [2n] \cup \{0\}\}. \qedhere
\end{align*}
\end{proof}

Before we establish our main result on $\cs(R^{(n)}),$ we establish two technical lemmas. We gratefully acknowledge Gerry Myerson for his original suggestion in the comments of \cite{D-MO} to consider the set $S(n, d)$ of Lemma \ref{complete double of [n]}.

\begin{lemma}\label{largest integer multiple d leq (n - d)/d}
Given any integers $1 \leq d \leq n - 1,$ we have that $${\left \lfloor \frac{n - d} d \right \rfloor} d =
n - r - d,$$ where $r$ is the least non-negative residue of $n$ modulo $d.$
\end{lemma}

\begin{proof}
If $n - d < d,$ then ${\left \lfloor \frac{n - d} d \right \rfloor} = 0$ and $r = n - d,$ so the claim holds. If $n - d \geq d,$ then by the Division Algorithm, there exist integers $q \geq 1$ and $0 \leq r \leq d - 1$ such that $n = qd + r.$ Consequently, we find that
\begin{align*}
{\left \lfloor \frac{n - d} d \right \rfloor} d &= {\left \lfloor \frac{(q - 1)d + r} d \right \rfloor} d = {\left \lfloor q - 1 + \frac r d \right \rfloor} d = (q - 1)d = qd - d = n - r - d. \qedhere
\end{align*}
\end{proof}

\begin{lemma}\label{complete double of [n]}
Given any integers $1 \leq d \leq n,$ the set $$S(n, d) = \{0, 1, \dots, d, n - d, n - d + 1, \dots, n\} \cup \{kd \mid k \geq 1 \text{ is an integer and } d \leq kd \leq n - d\}$$ is contained in $[n] \cup \{0\}$ and satisfies $S(n, d) + S(n, d) = [2n] \cup \{0\}.$
\end{lemma}

\begin{proof}
Let $r$ denote the least non-negative residue of $n$ modulo $d.$ By Lemma \ref{largest integer multiple d leq (n - d)/d}, we have that $$\max \{kd \mid k \geq 1 \text{ is an integer and } d \leq k \leq n - d \} = n - r - d.$$ Considering that $kd + i$ belongs to $S(n, d) + S(n, d)$ each pair of integers $0 \leq i \leq d$ and $k \geq 0$ such that $0 \leq kd \leq n - d,$ the integers $0, 1, \dots, n - r$ belong to $S(n, d) + S(n, d).$ By hypothesis that $S(n, d)$ contains $n - d, n - d + 1, \dots, n,$ we conclude that $0, 1, \dots, n$ belong to $S(n, d) + S(n, d).$ Further, $kd + (n - d + i)$ belongs to $S(n, d) + S(n, d)$ for each pair of integers $0 \leq i \leq d$ and $k \geq 1$ such that $d \leq kd \leq n - d,$ hence $n + 1, n + 2, \dots, 2n - d - r$ belong to $S(n, d) + S(n, d).$ Clearly, the integers $2n - 2d, 2n - 2d + 1, \dots, 2n$ belong to $S(n, d) + S(n, d),$ hence $S(n, d) + S(n, d) = [2n] \cup \{0\}.$
\end{proof}

\begin{comment}
Clearly, the integers $0, 1, \dots, n$ belong to $S(n, d) + S(n, d).$ By the Division Algorithm, every integer greater than or equal to $d + 1$ can be written as $kd + r$ for some integers $k \geq 1$ and $0 \leq r \leq d - 1.$ Consequently, the integers $d + 1, d + 2, \dots, n - r - 1$ belong to $S(n, d) + S(n, d)$ by Lemma \ref{largest integer multiple d leq (n - d)/d}. Considering that $r \leq d - 1,$ it follows that $n - r - 1 \geq n - d.$ Combined with our assumption that the integers $n - d, n - d + 1, \dots, n$ belong to $S(n, d),$ this implies that the integers $d + 1, d + 2, \dots, n - d, n - d + 1, \dots, n$ belong to $S(n, d) + S(n, d).$ Furthermore, the integers $n, n + 1, \dots, 2n - r - d$ belong to $S(n, d) + S(n, d),$ as they can be written as $n + i = (n - d + i) + kd$ for some integers $0 \leq i \leq d$ and $k \geq 1$ such that $d \leq kd \leq n - d.$ Considering that $r \leq d - 1,$ it follows that $2n - r - d \geq 2n - 2d + 1.$ By construction, the integers $n - d, n - d + 1, \dots, n$ belong to $S(n, d),$ hence the integers $2n - 2d, 2n - 2d + 1, \dots, 2n$ belong to $S(n, d) + S(n, d).$ We conclude therefore that $S(n, d) + S(n, d) = [2n] \cup \{0\},$ as desired.
\end{comment}

\begin{prop}\label{upper bound for cs of Veronese subring}
Let $n$ be a positive integer. We have that $$\min \{\lvert S \rvert : S \subseteq [n] \cup \{0\} \text{ and } S + S = [2n] \cup \{0\}\} \leq 2 \sqrt{2n} + 1.$$ Consequently, for the $n$th Veronese subring $R^{(n)}$ of $R = k[x, y],$ we have that $\cs(R^{(n)}) \leq 2 \sqrt{2n} + 1.$
\end{prop}

\begin{proof}
By Lemma \ref{complete double of [n]}, we have that $S(n, d) \subseteq [n] \cup \{0\}$ and $S(n, d) + S(n, d) = [2n] \cup \{0\}.$ We claim that $\lvert S(n, d) \rvert \leq 2 \sqrt{2n} + 1.$ One can readily verify that $\lvert S(n, d) \rvert = (d + 1) + (d + 1) + {\left \lfloor \frac{n - d} d \right \rfloor} = 2d + {\left \lfloor \frac n d \right \rfloor} + 1.$ Consequently, we have that $\lvert S(n, d) \rvert \leq 2d + \frac n d + 1 = f_n(d)$ for all integers $d \geq 1.$ Considering that $f_n(x)$ attains its minimum $2 \sqrt{2n} + 1$ at $x = \sqrt{\frac n 2},$ we conclude that $\lvert S(n, d) \rvert \leq 2 \sqrt{2n} + 1.$
\end{proof}

\begin{cor}\label{upper bound for size of complete double}
For any integer $n \geq 2,$ we have that $$\min \{\lvert S \rvert : S \subseteq [n] \cup \{0\} \text{ and } S + S = [2n] \cup \{0\}\} \geq 2 \sqrt{n + \frac 9 {16}} - \frac 1 2.$$
\end{cor}

\begin{proof}
Let $R^{(n)}$ denote the $n$th Veronese subring of $R = k[x, y]$ with unique homogeneous maximal ideal $\fm = (x^i y^{n - i} \mid 0 \leq i \leq n).$ Observe that $\mu(\fm^2) = 2n + 1.$ By Remark \ref{cs lower bound using u(m^2)}, we find that $$\cs(R) \geq \sqrt{2(2n + 1) + \frac 1 4} - \frac 1 2 = \sqrt{4n + 2 + \frac 1 4} - \frac 1 2 = 2 \sqrt{n + \frac 1 2 + \frac 1 {16}} - \frac 1 2 = 2 \sqrt{n + \frac 9 {16}} - \frac 1 2,$$ hence we conclude the desired result by Proposition \ref{cs and ms of Veronese subring}.
\end{proof}

\section{Computing \texorpdfstring{$\ms(R)$}{ms(R)} and \texorpdfstring{$\cs(R)$}{cs(R)} for quotients by quadratic ideals}\label{Computing ms(R) and cs(R) for quotients by quadratic ideals}

Let $k$ be a field. We will assume throughout this section that $R$ is a standard graded $k$-algebra, i.e., $R$ is the quotient of the polynomial ring $S = k[x_1, \dots, x_n]$ by a homogeneous ideal $I.$ We denote by $\fm = (x_1, \dots, x_n)$ the homogeneous maximal ideal of $S$ and by $\bar \fm$ the image of $\fm$ in $R.$ By Proposition \ref{cs and ms depend only on degree two}, $\ms(R)$ and $\cs(R)$ depend only on the degree two part of $I,$ hence we may assume that $I$ is a homogeneous quadratic ideal.

Using this as our motivation, we seek to compute $\ms(R)$ and $\cs(R)$ in the case that $I$ possesses certain quadratic generators. We begin by establishing the following lower bound for $\cs(R).$

\begin{prop}\label{cs of k[x_1, ..., x_n] modulo t homogeneous quadratics}
Let $R = k[x_1, \dots, x_n]/I,$ where $I$ is minimally generated over $k[x_1, \dots, x_n]$ by $t > 0$ homogeneous polynomials of degree two. We have that $$\binom{n + 1} 2 - t \leq \binom{\cs(R) + 1} 2.$$ Particularly, if $t \leq \dfrac{(s + 1)(2n - s)} 2$ for some integer $0 \leq s \leq n - 1,$ then $\cs(R) \geq n - s - 1.$ Even more, if $t \leq n - 1,$ then $\cs(R) = n = \mu(\bar \fm),$ where $\bar \fm$ is the image of $(x_1, \dots, x_n)$ in $R.$
\end{prop}

\begin{proof}
We note that it is straightforward to verify that $\mu(\bar \fm^2) = \mu(\fm^2) - \mu(I) = \binom{n + 1}{2} - t.$ Consider an ideal $J$ of $R$ that witnesses $\cs(R).$ We have that $\mu(J^2) \leq \binom{\mu(J) + 1}{2}$ so that $$\binom{n + 1}{2} - t = \mu(\bar \fm^2) = \mu(J^2) \leq \binom{\mu(J) + 1} 2 = \binom{\cs(R) + 1} 2,$$ where the last equality holds because $J$ witnesses $\cs(R).$ If $t \leq \dfrac{(s + 1)(2n - s)} 2,$ then
\begin{align*}
\binom{n + 1}{2} - t \geq \frac{n(n + 1) - (s + 1)(2n - s)}{2}
&= \frac{n^2 + n - 2ns + s^2 - 2n + s}{2} \\
&= \frac{(n - s)^2 - (n - s)}{2} = \binom{n - s}{2}
\end{align*}
so that $\binom{\cs(R) + 1}{2} \geq \binom{n - s}{2}$ and hence $\cs(R) \geq n - s - 1,$ as desired.

For the last claim, we will show that if $\cs(R) \neq n,$ then $t \geq n.$ By the first part of Proposition \ref{main proposition about cs and ms}, we have that $\cs(R) \leq \mu(\bar \fm) = n.$ Consequently, if $\cs(R) \neq n,$ we must have that $\cs(R) \leq n-1.$ But then, we have that $\binom{n + 1}{2} - t \leq \binom{\cs(R) + 1}{2} \leq \binom{n}{2}$ so that $t \geq \binom{n + 1}{2} - \binom{n}{2} = n.$
\end{proof}

\begin{remark}\label{cs for rings other than regular local and hypersurface}
For any integer $n \geq 3,$ Proposition \ref{cs of k[x_1, ..., x_n] modulo t homogeneous quadratics} provides a class of (standard graded) local rings for which $\cs(R) = \mu(\fm)$ other than regular local rings and hypersurface rings.
\end{remark}

\begin{remark}\label{cs lower bound using u(m^2)}
Using a similar idea as in the proof of Proposition \ref{cs of k[x_1, ..., x_n] modulo t homogeneous quadratics}, one can show that $$\cs(R) \geq {\left \lceil \frac{\sqrt{8 \mu(\fm^2) + 1} - 1}{2} \right \rceil}$$ holds for any Noetherian local ring $(R, \fm).$ By the third sentence of the proof of Proposition \ref{cs of k[x_1, ..., x_n] modulo t homogeneous quadratics}, we have that $\mu(\fm^2) \leq \binom{\cs(R) + 1}{2}$ so that $2\mu(\fm)^2 \leq \cs(R)[\cs(R) + 1].$ Completing the square yields $${\left(\cs(R) + \frac{1}{2} \right)}^{\!2} \geq 2 \mu(\fm^2) + \frac{1}{4} = \frac{8 \mu(\fm)^2 + 1}{4}.$$ From here, one can easily verify the original displayed inequality.
\end{remark}

We turn our attention to the following proposition that we used in Remark \ref{necessity of certain assumptions}.

\begin{prop}\label{cs of k[x_1, ..., x_n]/(x_1^2, ..., x_n^2)}
If $R = k[x_1, \dots, x_n] / (x_1^2, \dots, x_n^2)$ and $\chara(k) \neq 2,$ then $\cs(R) = n - 1.$
\end{prop}

\begin{proof}
By applying Proposition \ref{cs of k[x_1, ..., x_n] modulo t homogeneous quadratics} with $s = 0,$ we have that $\cs(R) \geq n-1,$ so it suffices to exhibit an ideal $J$ of $R$ such that $\mu(J) = n - 1$ and $J^2 = \bar \fm^2.$

Consider the ideal $J = (\bar x_i + \bar x_{i + 1} \mid 1 \leq i \leq n-1).$ Evidently, we have that $\mu(J) = n - 1.$ We claim that $\bar x_i \bar x_j \in J^2$ for any pair of integers $1 \leq i < j \leq n.$ Observe that $\bar x_i \bar x_{i+1} = (\bar x_i + \bar x_{i+1})^2$ belongs to $J^2$ for each integer $1 \leq i \leq n - 1,$ hence $\bar x_i \bar x_{i+2} = (\bar x_i + \bar x_{i+1})(\bar x_{i+1} + \bar x_{i+2}) - x_i \bar x_{i+1} - \bar x_{i+1} \bar x_{i+2}$ belongs to $J^2.$ We obtain $\bar x_i \bar x_j$ for any pair of integers $1 \leq i < j \leq n$ as follows.
\begin{enumerate}

\item[(i.)] Compute first the squares $(\bar x_i + \bar x_{i+1})^2$ of the generators of $J$ for each integer $i \leq j-1.$ From this, we obtain generators of $J^2$ of the form $\bar x_i \bar x_{i+1}.$

\vspace{0.25cm}

\item[(ii.)] Compute next the products $(\bar x_i + \bar x_{i+1})(\bar x_{i+1} + \bar x_{i+2})$ for each integer $i \leq j-2.$ Using the previous step, we obtain generators of $J^2$ of the form $\bar x_i \bar x_{i+2}.$

\vspace{0.25cm}

\item[(iii.)] Compute the products $(\bar x_i + \bar x_{i+1})(\bar x_k + \bar x_{k+1})$ for each integer $i+2 \leq k \leq j-1.$ Use the previous steps to cancel any quadratic forms that have already appeared.

\end{enumerate}
Ultimately, we find that $\bar x_i \bar x_j \in J^2$ for all integers $1 \leq i < j \leq n$ so that $\bar \fm^2 \subseteq J^2.$
\end{proof}

Our next proposition illustrates that when $t = n,$ it is possible that $\cs(R) = n - 1$ or $\cs(R) = n,$ hence the lower bound for $\cs(R)$ provided in Proposition \ref{cs of k[x_1, ..., x_n] modulo t homogeneous quadratics} is sharp in this case.

\begin{prop}\label{cs for k[x_1, x_2, x_3]/(x_i^2, x_j^2, x_i x_j)}
Let $R = k[x_1, x_2, x_3] / I,$ where $I$ is minimally generated by three monomials of degree two and $\chara(k) \neq 2.$ If $I = (x_i^2, x_j^2, x_i x_j)$ for some integers $1 \leq i < j \leq 3,$ then $\cs(R) = 3;$ otherwise, $\cs(R) = 2.$
\end{prop}

\begin{proof}
By Proposition \ref{cs of k[x_1, ..., x_n] modulo t homogeneous quadratics}, we have that $\cs(R)[\cs(R) + 1] \geq 2 {\left(\binom{3 + 1}{2} - 3 \right)} = 6,$ hence we have that $\cs(R) \geq 2.$ If $I = (x_i^2, x_j^2, x_i x_j),$ let us assume to the contrary that $\cs(R) = 2,$ i.e., there exists an ideal $J$ such that $\mu(J) = 2$ and $J^2 = \bar \fm^2.$ By Proposition \ref{cs is witnessed by linear forms}, we may assume that $J = (a \bar x_i + b \bar x_j + c \bar x_k, d \bar x_i + e \bar x_j + f \bar x_k)$ for some elements $a,b,c,d,e,f \in k.$ We claim that both $c$ and $f$ must be nonzero. For if not, then without loss of generality, we have that $J = (a \bar x_i + b \bar x_j, d \bar x_i + e \bar x_j + f \bar x_k)$ so that $J^2 = (af \bar x_i \bar x_k + bf \bar x_j \bar x_k, 2df \bar x_i \bar x_k + 2ef \bar x_j \bar x_k + \bar x_k^2).$ But then, it would be the case that $3 = \mu(\bar \fm^2) = \mu(J^2) \leq 2$ --- a contradiction. Consequently, both $c$ and $f$ are nonzero, hence we may assume without loss of generality that $c = f = 1$ and $J = (a \bar x_i + b \bar x_j + \bar x_k, d \bar x_i + e \bar x_j + \bar x_k).$ But then, $(a-d) \bar x_i + (b-e) \bar x_j$ is a linear combination of the generators of $J,$ hence we may write $J = (a \bar x_i + b \bar x_j + \bar x_k, (a-d) \bar x_i + (b-e) \bar x_j).$ But this contradicts our previous observation that both generators of $J$ must possess a nonzero multiple of $\bar x_k.$ We conclude that $\cs(R) = 3 = \mu(\fm).$

We will assume henceforth that $I$ is not generated by $x_i^2,$ $x_j^2,$ and $x_i x_j$ for any integers $1 \leq i < j \leq 3.$ Consequently, there are five possibilities for $I.$ Let $i,j,$ and $k$ be distinct indices.
\begin{enumerate}

\item[(i.)] If $I$ contains three squarefree monomials, we have that $I = (x_1 x_2, x_1 x_3, x_2 x_3).$ Observe that $J = (\bar x_1 + \bar x_2, \bar x_1 + \bar x_3)$ witnesses $\cs(R)$ since $I^2 = (\bar x_1^2 + \bar x_2^2, \bar x_1^2, \bar x_1^2 + \bar x_3^2),$ from which it follows that $J^2 = (\bar x_1^2, \bar x_2^2, \bar x_3^2) = \bar \fm^2$ and $\cs(R) = \mu(J) = 2.$

\vspace{0.25cm}

\item[(ii.)] If $I$ contains two squarefree monomials, then there are two possibilities for $J.$ If $I = (x_i^2, x_i x_j, x_i x_k),$ it suffices to take $J = (\bar x_i + \bar x_j, \bar x_i + \bar x_k)$ since it follows that $J^2 = (\bar x_j^2, \bar x_i \bar x_j, \bar x_k^2) = \bar \fm^2$ so that $\cs(R) = \mu(J) = 2.$ On the other hand, it is possible that $I = (x_i^2, x_i x_j, x_j x_k).$ Even in this case, we may take $J = (\bar x_i + \bar x_j, \bar x_i + \bar x_k)$ since it holds that $J^2 = (\bar x_j^2, \bar x_i \bar x_k, 2 \bar x_i \bar x_k + \bar x_k^2) = (\bar x_j^2, \bar x_i \bar x_k, \bar x_k^2) = \bar \fm^2$ and $\cs(R) = \mu(J) = 2.$

\vspace{0.25cm}

\item[(iii.)] If $I$ contains one squarefree monomial, then by our assumption at the beginning of the above paragraph, we must have that $I = (x_i^2, x_j^2, x_i x_k).$ Observe that $J = (\bar x_i + \bar x_j, \bar x_k)$ witnesses $\cs(R)$ since $J^2 = (\bar x_i \bar x_j, \bar x_j \bar x_k, \bar x_k^2) = \bar \fm^2$ and $\cs(R) = \mu(J) = 2.$

\vspace{0.25cm}

\item[(iv.)] If $I$ contains no squarefree monomials, we have that $I = (x_1^2, x_2^2, x_3^2).$ By Proposition \ref{cs of k[x_1, ..., x_n]/(x_1^2, ..., x_n^2)}, we have that $\cs(R) = 2.$ (We constructed $J$ in the proof of Proposition \ref{cs of k[x_1, ..., x_n]/(x_1^2, ..., x_n^2)}.)

\end{enumerate}
One can readily verify that these are all of the possibilities for $I,$ so we are done.
\end{proof}

For the case of $n = 3,$ if $s = 1,$ then Proposition \ref{cs of k[x_1, ..., x_n] modulo t homogeneous quadratics} implies that for $t \leq 5,$ we have that $1 \leq \cs(R) \leq 3.$ Our next proposition illustrates that $\cs(R)$ can lie strictly between these bounds.

\begin{prop}\label{cs of k[x_1, x_2, x_3] modulo 4 quadratics} 
Let $R = k[x_1, x_2, x_3] / I,$ where $I$ is minimally generated over $k[x_1, x_2, x_3]$ by $t = 4$ monomials of degree two and $\chara(k) \neq 2.$ We have that $\cs(R) = 2.$
\end{prop}

\begin{proof}
By Proposition \ref{cs of k[x_1, ..., x_n] modulo t homogeneous quadratics}, we have that $\cs(R)[\cs(R) + 1] \geq 2 {\left(\binom{3+1}{2} - 4 \right)} = 4,$ from which it follows that $\cs(R) \geq 2.$ Consider the following cases.
\begin{enumerate}

\item[(i.)] If $I = (x_1^2, x_2^2, x_3^2, x_i x_j)$ for some integers $1 \leq i < j \leq 3,$ then $J = (\bar x_i + \bar x_k, \bar x_j + \bar x_k)$ witnesses $\cs(R),$ where $k$ is the index remaining in the set $\{1,2,3\} \setminus \{i,j\}.$ We note that $\bar \fm^2 = (\bar x_i \bar x_k, \bar x_j \bar x_k).$ Further, we have that $J^2$ is generated by $(\bar x_i + \bar x_k)^2 = 2 \bar x_i \bar x_k,$ $(\bar x_i + \bar x_k)(\bar x_j + \bar x_k) = \bar x_i \bar x_k + \bar x_j \bar x_k,$ and $(\bar x_j + \bar x_k)^2 = 2 \bar x_j \bar x_k.$ By assumption that $\chara(k) \neq 2,$ it follows that $J^2 = (\bar x_i \bar x_k, \bar x_j \bar x_k) = \bar \fm^2$ so that $\cs(R) = 2.$

\vspace{0.25cm}

\item[(ii.)] If $I = (x_i^2, x_j^2, x_i x_j, x_i x_k)$ for some distinct indices $i,j,k,$ then $J = (\bar x_j, \bar x_k)$ witnesses $\cs(R).$ We note that $\bar \fm^2 = (\bar x_k^2, \bar x_j \bar x_k) = J^2$ so that $\cs(R) = 2.$

\vspace{0.25cm}

\item[(iii.)] If $I = (x_i^2, x_j^2, x_i x_k, x_j x_k)$ for some distinct indices $i,j,k,$ then $J = (\bar x_i + \bar x_j, \bar x_k)$ witnesses $\cs(R).$ We note that $\bar \fm^2 = (\bar x_k^2, \bar x_i \bar x_j).$ Further, we have that $J^2$ is generated by $(\bar x_i + \bar x_j)^2 = \bar x_i \bar x_j$ and $\bar x_k^2$ so that $J^2 = (\bar x_k^2, \bar x_i \bar x_j) = \bar \fm^2$ and $\cs(R) = 2.$ 

\vspace{0.25cm}

\item[(iv.)] If $I = (x_i^2, x_i x_j, x_i x_k, x_j x_k)$ for some distinct indices $i,j,k,$ then $J = (\bar x_j, \bar x_k)$ witnesses $\cs(R).$ We note that $\bar \fm^2 = (\bar x_j^2, \bar x_k^2) = J^2$ so that $\cs(R) = 2.$

\end{enumerate}
One can readily verify that these are all of the possibilities for $I,$ so we are done.
\end{proof}

\begin{prop}\label{cs of k[x_1, ..., x_n] modulo (n + 1 choose 2) - 1 quadratics}
Let $R = k[x_1, \dots, x_n] / I$ with $\chara(k) \neq 2,$ where $I$ is minimally generated by $t = \binom{n+1}{2} - 1$ monomials of degree two. We have that $\cs(R) = 1.$ 
\end{prop} 

\begin{proof}
We note that $\fm^2$ is generated by $\binom{n+1}{2}$ distinct monomials of degree two, so $I$ consists of all but one quadratic monomial. Consequently, we have that $\bar \fm^2 = \bar q R,$ where $q$ is the quadratic monomial excluded from $I.$ Observe that $J = {\left(\sum_{i=1}^n \bar x_i \right)}$ satisfies $J^2 = \bar f R = \fm^2.$ We have conclude that $\cs(R) = 1,$ as desired.
\end{proof}

Our next aim is to establish similar bounds for $\ms(R).$ We continue to restrict our attention to the case that $R = k[x_1, \dots, x_n] / I,$ where $I$ is minimally generated by $t > 0$ quadratic monomials.

\begin{prop}\label{ms for univariate and bivariate polynomial ring modulo homogeneous ideal, hypersurface in three variables}
Let $R = k[x_1, \dots, x_n] / I,$ where $I$ is minimally generated by $t > 0$ quadratic monomials.
\begin{enumerate}[\rm (1.)]

\item If $n = 1,$ then $\ms(R) = 0.$

\vspace{0.25cm}

\item If $n = 2,$ $\bar \fm^2 \neq \bar 0,$ and $\chara(k) \neq 2,$ then $\ms(R) = 1.$

\end{enumerate}
\end{prop}

\begin{proof}
(1.) Clearly, if $n = 1,$ then $I = \fm^2.$ Consequently, we have that $\bar \fm^2 = \bar 0$ so that $\ms(R) = 0.$

(2.) If $n = 2,$ then $I$ must contain (at least) one of the quadratic monomials $x_1^2, x_1 x_2,$ or $x_2^2.$ We claim that $J = (\bar x_1 + \bar x_2)$ satisfies $\bar \fm^2 \subseteq \bar \fm I.$ Observe that $\bar x_i (\bar x_1 + \bar x_2) = \bar x_i \bar x_1 + \bar x_i \bar x_2$ and $(\bar x_1 + \bar x_2)^2 = \bar x_1^2 + 2 \bar x_1 \bar x_2 + \bar x_2^2.$ If $I$ contains either of the pure squares $\bar x_i^2,$ then $J$ must contain the other square $\bar x_j^2$ and the mixed term $\bar x_1 \bar x_2$ by hypothesis that $\chara(k) \neq 2,$ and the aforementioned equations show that $\bar \fm^2 \subseteq J.$ If $I$ does not contain either of the pure squares, then it must contain $\bar x_1 \bar x_2,$ and again, we find that $\bar \fm^2 \subseteq J.$
\end{proof}

\begin{prop}\label{ms of k[x_1, ..., x_n] modulo ideal contained in (x_1, ..., x_n)^2}
Let $R = k[x_1, \dots, x_n] / I,$ where $I$ is any non-maximal homogeneous ideal of $k[x_1, \dots, x_n]$ that contains all quadratic squarefree monomials, i.e., $(x_i x_j \mid 1 \leq i < j \leq n) \subseteq I \subsetneq \fm.$ We have that $\ms(R) = 1.$
\end{prop}

\begin{proof}
Consider the ideal $J = (\bar x_1 + \cdots + \bar x_n).$ By assumption that $x_i x_j \in I$ for all $1 \leq i < j \leq n,$ we have that $$\bar \fm J = {\left(\bar x_i \sum_{j = 1}^n \bar x_j \mid 1 \leq i \leq n \right)} = {\left(\sum_{j=1}^n \bar x_i \bar x_j \mid 1 \leq i \le n \right)} = (\bar x_i^2 \mid 1 \leq i \leq n) = \bar \fm^2$$ so that $\bar \fm^2 = \bar \fm J \subseteq J.$ Consequently, it follows that $0 \leq \ms(R) \leq 1.$ Even more, by hypothesis that $I$ is not the homogeneous maximal ideal, we have that $\bar \fm^2 = (\fm^2 + I)/I \neq \bar 0,$ and we conclude that $\ms(R) = 1.$
\end{proof}

Our previous proposition suggests that $\ms(R)$ is controlled primarily by the quadratic squarefree monomials of $k[x_1, \dots, x_n].$ Consequently, we devote the last section to this case.

\section{Computing \texorpdfstring{$\ms(R)$}{ms(R)} and \texorpdfstring{$\cs(R)$}{cs(R)} for the edge ring of a finite simple graph}\label{Computing ms(R) and cs(R) for the edge ring of a finite simple graph}

Last, we turn our attention to the case that $R = k[x_1, \dots, x_n] / I$ for some quadratic squarefree monomial ideal $I.$ By the Stanley-Reisner Correspondence, this is equivalent to studying the combinatorial structures of finite simple graphs (cf. \cite[Chapter 4]{MRSW}). We denote by $G$ a simple graph on the vertex set $[n] = \{1, 2, \dots, n \}$ with edges denoted by unordered pairs $\{i, j\}$ for some integers $1 \leq i < j \leq n.$ We say that a vertex $i$ is {\it isolated} if $\{i, j\}$ is not an edge of $G$ for any integer $j.$ If $G$ has no isolated vertices, then $G$ is {\it connected}.

We say that a subgraph $H$ of $G$ is {\it induced} if the edge $\{i, j\}$ in $G$ is an edge of $H$ whenever $H$ contains the vertices $i$ and $j.$ Given any nonempty set $V \subseteq [n],$ we will write $G[V]$ for the induced subgraph of $G$ on the vertex set $V.$ We define also the {\it complement graph} $\overline G$ on the vertex set $[n]$ such that $\{i, j\}$ is an edge of $\overline G$ if and only if it is not an edge of $G.$ Observe that for any nonempty set $V \subseteq [n],$ the complement of an induced subgraph of $G$ is the induced subgraph of $\overline G$ on the same underlying vertex set, i.e., we have that $\overline{G[V]} = \overline G[V].$ One can visualize the aforementioned constructions in the following example.
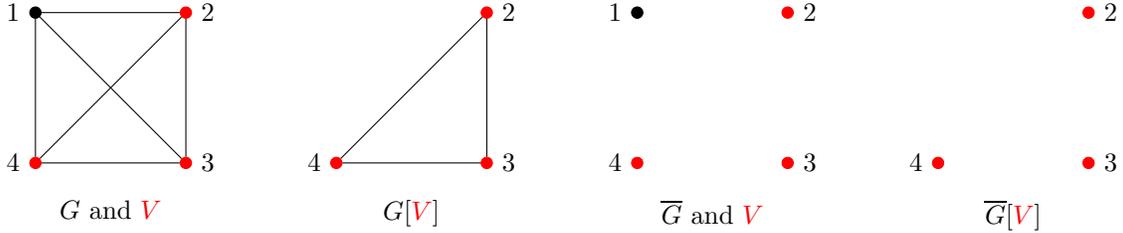
\begin{center}
\begin{tikzpicture}
\tikzset{enclosed/.style={draw, circle, inner sep=0pt, minimum size=.15cm, fill=black}}
\node[enclosed, label={left: 1}] (1) at (-7,1) {};
\node[enclosed, color=red, label={right: 2}] (2) at (-5,1) {};
\node[enclosed, color=red, label={right: 3}] (3) at (-5,-1) {};
\node[enclosed, color=red, label={left: 4}] (4) at (-7,-1) {};
\draw (1) -- (2);
\draw (2) -- (3);
\draw (3) -- (4) node[midway, label={below, yshift=-.25cm: $\textcolor{black}{G \text{ and }} \textcolor{red}{V}$}] (edge4) {};
\draw(1) -- (4);
\draw (1) -- (3);
\draw (2) -- (4);
\node[enclosed, color=red, label={right: 2}] (2) at (-1,1) {};
\node[enclosed, color=red, label={right: 3}] (3) at (-1,-1) {};
\node[enclosed, color=red, label={left: 4}] (4) at (-3,-1) {};
\draw (2) -- (3);
\draw (3) -- (4) node[midway, label={below, yshift=-.25cm: $\textcolor{black}{G[\textcolor{red}{V}]}$}] (edge3) {};
\draw (2) -- (4);
\node[enclosed, label={left: 1}] (1) at (1,1) {};
\node[enclosed, color=red, label={right: 2}] (2) at (3,1) {};
\node[enclosed, color=red, label={right: 3}] (3) at (3,-1) {};
\node[enclosed, color=red, label={left: 4}] (4) at (1,-1) {};
\draw[color=white] (3) -- (4) node[midway, label={below, yshift=-.25cm: $\textcolor{black}{\overline G} \textcolor{black}{\text{ and }} \textcolor{red}{V}$}] (edge0) {};
\node[enclosed, color=red, label={right: 2}] (2) at (7,1) {};
\node[enclosed, color=red, label={right: 3}] (3) at (7,-1) {};
\node[enclosed, color=red, label={left: 4}] (4) at (5,-1) {};
\draw[color=white] (3) -- (4) node[midway, label={below, yshift=-.25cm: $\textcolor{black}{\overline G[\textcolor{red}{V}]}$}] (edge0) {};
\end{tikzpicture}
\end{center}

We denote by $K_n$ the complete graph on $[n]$ with edges $\{i, j\}$ for all integers $1 \leq i < j \leq n,$ hence $K_n$ has $\binom{n}{2}$ edges. Observe that $K_m$ is an induced subgraph of $K_n$ for each integer $1 \leq m \leq n.$

Our main object of study throughout this section is defined as follows.

\begin{definition}\label{edge ring definition}
Given a finite simple graph $G$ on the vertex set $[n]$ and any field $k,$ we refer to the quadratic squarefree monomial ideal $I(G) = (x_i x_j \mid \{i, j\} \text{ is an edge of } G)$ of $k[x_1, \dots, x_n]$ as the \textit{edge ideal} of $G,$ and we call the quotient ring $k(G) = k[x_1, \dots, x_n] / I(G)$ the \textit{edge ring} of $G.$
\end{definition}

We will continue to study the homogeneous maximal ideal $\fm = (x_1, \dots, x_n)$ of $k[x_1, \dots, x_n]$ and its image $\bar \fm$ in $k(G).$ Our motivation to consider $\ms(R)$ and $\cs(R)$ from a graphical perspective is rooted in the results of Proposition \ref{cs and ms depend only on degree two} and Section \ref{Computing ms(R) and cs(R) for quotients by quadratic ideals}. We begin with the following.

\begin{prop}\label{ms(K_n)}
We have that $\ms(k(K_n)) = 1$ and $\cs(k(K_n)) \geq {\left \lceil \sqrt{2n + \frac 1 4} - \frac 1 2 \right \rceil}$ for all $n \geq 1.$
\end{prop}

\begin{proof}
Observe that $K_1$ consists of one vertex, hence $k(K_1) = k[x]$ is regular, and the result holds by the first part of Proposition \ref{main proposition about cs and ms}. Otherwise, we have that $n \geq 2$ and $I(K_n) = (x_i x_j \mid 1 \leq i < j \leq n).$ By Proposition \ref{ms of k[x_1, ..., x_n] modulo ideal contained in (x_1, ..., x_n)^2}, it follows that $\ms(k(K_n)) = 1.$ Observe that $\bar \fm^2 = (\bar x_i^2 \mid 1 \leq i \leq n),$ hence Remark \ref{cs lower bound using u(m^2)} implies that
\begin{align*}
\cs(k(K_n)) &\geq {\left \lceil \frac{\sqrt{8 \mu(\bar \fm^2) + 1} - 1} 2 \right \rceil} = {\left \lceil \sqrt{2n + \frac 1 4} - \frac 1 2 \right \rceil}. \qedhere
\end{align*}
\end{proof}

On the other hand, the lower bound for $\cs(k(K_t))$ is sharp whenever $t = \binom{r + 1} 2$ for some integer $r \geq 1,$ i.e., $t$ is a triangular number; the idea for the proof is due to Mark Denker.

\begin{prop}\label{cs(K_n) for triangular numbers}
Let $t = \binom{r + 1} 2$ for some integer $r \geq 1.$ We have that $\cs(k(K_t)) \leq \sqrt{2t + \frac 1 4} - \frac 1 2 = r.$
\end{prop}

\begin{proof}
By the Quadratic Formula, the upper bound holds because $$t = \binom{r + 1} 2 \text{ if and only if } r^2 + r - 2t = 0 \text{ if and only if } r = \frac{\sqrt{8t + 1} - 1} 2 = \sqrt{2t + \frac 1 4} - \frac 1 2.$$

Given any integer $r \geq 1,$ we will construct an ideal $J$ with $\mu(J) = r$ such that $J^2 = (\bar x_1^2, \dots, \bar x_t^2),$ where $\bar x_i$ denotes the image of $x_i$ modulo $I(K_t).$ We accomplish this via the following steps.
\begin{enumerate}[\rm (i.)]

\item For each pair of distinct indices $1 \leq i < j \leq r,$ choose a distinct monomial generator $m_{i, j}$ of $(\bar x_1, \dots, \bar x_t).$ Once all $\binom r 2$ monomials have been chosen as such, there will remain $r$ unused monomial generators.

\vspace{0.25cm}

\item Define homogeneous linear polynomials $f_1, \dots, f_r$ as $f_i = \sum_{j \neq i} m_{i, j}.$ Observe that each of the polynomials $f_1, \dots, f_r$ is by definition the sum of $r - 1$ distinct monomials. Even more, for each pair of distinct indices $1 \leq i < j \leq r,$ the polynomials $f_i$ and $f_j$ have only one summand in common --- the monomial $m_{i, j}.$

\vspace{0.25cm}

\item For each integer $1 \leq i \leq r,$ choose a monomial generator $m_i$ that does not appear as a summand of any of the polynomials $f_1, \dots, f_r$; then, define the polynomials $g_1, \dots, g_r$ as $g_i = f_i + m_i.$

\vspace{0.25cm}

\item Ultimately, we obtain $r$ homogeneous linear polynomials $g_1, \dots, g_r,$ each of which is the sum of $r$ distinct linear monomials. Observe that for each integer $1 \leq i \leq t,$ we have that $g_i^2$ is the sum of the squares of all of its monomial summands. By construction, for each monomial summand $m$ of $f_i,$ there exists a distinct polynomial $f_j$ such that $m$ is a summand of $f_j$ and $f_i f_j = m^2.$ Consequently, the ideal $J^2$ contains all of the pure squares $\bar x_1, \dots, \bar x_t^2$: they appear either as $m_{i, j}^2 = g_i g_j$ or $m_i^2 = g_i^2 - \sum_{j \neq i} g_i g_j.$

\end{enumerate}
We conclude that the ideal $J = (g_1, \dots, g_r)$ of $k(K_t)$ satisfies $J^2 = (\bar x_1^2, \dots, \bar x_t^2)$ and $\mu(J) \leq r.$
\end{proof}

We illustrate the idea of proof of Proposition \ref{cs(K_n) for triangular numbers} in the following example.

\begin{example}\label{proof of cs(K_n) for triangular numbers example}
Let $r = 4.$ Observe that $\binom{r + 1} 2 = 10,$ so it suffices to exhibit an ideal $J$ of $\cs(K_{10})$ with $\mu(J) = 4$ and $J^2 = (\bar x_1^2, \cdots, \bar x_{10}^2).$ By the proof of Proposition \ref{cs(K_n) for triangular numbers}, this can be accomplished as follows.
\begin{enumerate}[\rm (i.)]

\item Choose the monomial generators $m_{i, j}$ for each pair of distinct integers $1 \leq i < j \leq 4.$ We will simply take $m_{1, 2} = \bar x_1,$ $m_{1, 3} = \bar x_2,$ $m_{1, 4} = \bar x_3,$ $m_{2, 3} = \bar x_4,$ $m_{2, 4} = \bar x_5,$ and $m_{3, 4} = \bar x_6.$

\vspace{0.25cm}

\item Define the polynomials $f_i = \sum_{j \neq i} m_{i, j}$ for each integer $1 \leq i \leq r.$ By our above choices, they are given by $f_1 = \bar x_1 + \bar x_2 + \bar x_3,$ $f_2 = \bar x_1 + \bar x_4 + \bar x_5,$ $f_3 = \bar x_2 + \bar x_4 + \bar x_6,$ and $f_4 = \bar x_3 + \bar x_5 + \bar x_6.$

\vspace{0.25cm}

\item Choose a monomial generator $m_i$ that does not appear in $f_i$ for each integer $1 \leq i \leq r$; then, add it to $f_i,$ and call the resulting polynomial $g_i.$ We will set $m_i = \bar x_{6 + i}$ for each integer $1 \leq i \leq r$ so that $g_1 = \bar x_1 + \bar x_2 + \bar x_3 + \bar x_7,$ $g_2 = \bar x_1 + \bar x_4 + \bar x_5 + \bar x_8,$ $g_3 = \bar x_2 + \bar x_4 + \bar x_6 + \bar x_9,$ and $g_4 = \bar x_3 + \bar x_5 + \bar x_6 + \bar x_{10}.$

\vspace{0.25cm}

\item Observe that $\bar x_1^2 = g_1 g_2,$ $\bar x_2^2 = g_1 g_3,$ $\bar x_3^2 = g_1 g_4,$ $\bar x_4^2 = g_2 g_3,$ $\bar x_5^2 = g_2 g_4,$ and $\bar x_6^2 = g_3 g_4.$ Once we have these, it follows that $\bar x_7^2 = g_1^2 - \sum_{j \neq 1} g_1 g_j,$ $\bar x_8^2 = g_2^2 - \sum_{j \neq 2} g_2 g_j,$ $\bar x_9^2 = g_3^2 - \sum_{j \neq 3} g_3 g_j,$ and $\bar x_{10}^2 = g_4^2 - \sum_{j \neq 4} g_4 g_j.$

\end{enumerate}
By taking $J = (g_1, g_2, g_3, g_4),$ we find that $J^2 = (\bar x_1^2, \cdots, \bar x_{10}^2)$ and $\mu(J) \leq 4.$
\end{example}

We note that the invariants behave nicely with respect to induced subgraphs.

\begin{prop}\label{cs and ms of induced subgraphs}
Given a finite simple graph $G$ on the vertex set $[n]$ with an induced subgraph $H$ on $m$ vertices, we have that $\ms(k(H)) \leq \ms(k(G))$ and $\cs(k(H)) \leq \cs(k(G)).$
\end{prop}

\begin{proof}
We may assume that $H = G[V]$ is the induced subgraph on the vertex set $V = [m].$ Observe that for any edge $\{i, j\}$ of $G$ such that $m + 1 \leq i \leq n,$ the monomial $x_i x_j$ of $I(G)$ belongs to the ideal $(x_{m + 1}, \dots, x_n)$ of $k[x_1, \dots, x_n].$ Consequently, we have that $(x_{m + 1}, \dots, x_n) + I(G) = (x_{m + 1}, \dots, x_n) + I(H)$ so that $$k(H) = \frac{k[x_1, \dots, x_m]}{I(H)} \cong \frac{k[x_1, \dots, x_n]}{(x_{m + 1}, \dots, x_n) + I(H)} = \frac{k[x_1, \dots, x_n]}{(x_{m + 1}, \dots, x_n) + I(G)} \cong \frac{k(G)}{(\bar x_{m + 1}, \dots, \bar x_n)}.$$ By Proposition \ref{comparison theorem for cs and ms under surjective ring homomorphism}, we conclude that $\ms(k(G)) \geq \ms(k(H))$ and $\cs(k(G)) \geq \cs(k(H)).$
\end{proof}

Using a short technical lemma regarding the ceiling function in conjunction with the propositions established thus far, the lower bound for $\cs(k(K_n))$ is sharp for all integers $n \geq 1$ as follows.

\begin{lemma}\label{ceiling function lemma}
Let $n, r, t \geq 1$ be integers such that $\binom r 2 < n \leq \binom{r + 1} 2 = t.$ If $f(n) = \sqrt{2n + \frac 1 4} - \frac 1 2,$ then ${\left \lceil f(n) \right \rceil} = r.$
\end{lemma}

\begin{proof}
Observe that $f$ is an increasing function with $f(t) = r$ and $f{\left(\binom r 2 \right)} = r - 1.$ By hypothesis that $\binom r 2 < n \leq \binom{r + 1} 2,$ it follows that $r - 1 = f{\left(\binom r 2 \right)} < f(n) \leq f{\left(\binom{r + 1} 2 \right)} = f(t) = r.$ Consequently, the ceiling function yields $r - 1 = {\left \lceil r - 1 \right \rceil} < {\left \lceil f(n) \right \rceil} \leq {\left \lceil r \right \rceil} = r$ so that ${\left \lceil f(n) \right \rceil} = r.$
\end{proof}

Observe that $t$ in the hypotheses of Lemma \ref{ceiling function lemma} is the smallest triangular number with $n \leq t.$

\begin{cor}\label{cs(K_n)}
If $t = \binom{r + 1} 2$ is the smallest triangular number such that $n \leq t,$ then $\cs(k(K_n)) = r.$
\end{cor}

\begin{proof}
Observe that $K_n$ is an induced subgraph of $K_t$ for every integer $t \geq n.$ By Propositions \ref{cs(K_n) for triangular numbers} and \ref{cs and ms of induced subgraphs}, we have that $\cs(k(K_n)) \leq \cs(k(K_t)) = r.$ Consequently, it suffices to show that $\cs(k(K_n)) \geq r.$ But this is precisely what Proposition \ref{ms(K_n)} and Lemma \ref{ceiling function lemma} together imply.
\end{proof}

Even more, $\ms(k(G))$ is monotone decreasing with respect to adding edges between existing vertices of $G.$

\begin{prop}\label{ms(k(G)) is monotone decreasing with respect to adding edges}
Let $G$ be a finite simple graph that does not contain an edge $\{i, j\}.$ Let $G'$ be the finite simple graph obtained from $G$ by adjoining the edge $\{i, j\}.$ We have that $\ms(k(G)) - 1 \leq \ms(k(G')) \leq \ms(k(G)).$
\end{prop}

\begin{proof}
Observe that $k(G') = k[x_1, \dots, x_n] / [I(G) + (x_i x_j)] \cong k(G) / (\bar x_i \bar x_j),$ so this follows by Corollary \ref{ms(R) with respect to cutting down by t quadratic forms}.
\end{proof}

Conversely, $\ms(k(G))$ is monotone increasing with respect to adding vertices to $G.$

\begin{prop}\label{ms(k(G)) is monotone increasing with respect to adding vertices}
Let $H$ be a simple graph on the vertex set $[t].$ Let $G$ be the simple graph obtained from $H$ by adding some additional vertices $t + 1, \dots, n.$ We have that $\ms(k(G)) = \ms(k(H)) + n - t.$
\end{prop}

\begin{proof}
Observe that $k(G) \cong k(H)[X_{t + 1}, \dots, X_n],$ so this follows by Proposition \ref{ms of algebra obtained by adjoining indeterminates}.
\end{proof}

\begin{cor}\label{ms in terms of number of isolated vertices and nontrivial connected components}
Let $G$ be a finite simple graph on the vertex set $[n]$ with $d$ isolated vertices. Let $H$ be the induced subgraph of $G$ on the non-isolated vertices. We have that $\ms(k(G)) = \ms(k(H)) + d.$
\end{cor}

Our immediate goal is to find bounds for $\ms(k(G))$ and $\cs(k(G))$ for several well-known families of graphs. Before we proceed with this agenda, we recall the following definitions.

\begin{definition}\label{vertex cover definition}
Given a finite simple graph $G$ on the vertex set $[n],$ we say that a set $C \subseteq [n]$ forms a \textit{vertex cover} of $G$ if for every edge $\{i, j\}$ of $G,$ we have that $i \in C$ or $j \in C.$ Further, we say that a vertex cover $C$ is a \textit{minimal vertex cover} if $C \setminus \{i\}$ is not a vertex cover of $G$ for any vertex $i \in C.$ We denote $\tau(G) = \min\{\lvert C \rvert : C \text{ is a vertex cover of } G\}.$
\end{definition}

\begin{definition}\label{independent set definition}
Given a finite simple graph $G$ on the vertex set $[n],$ we say that a set $I \subseteq [n]$ forms an \textit{independent vertex set} of $G$ if $\{i, j\}$ is not an edge of $G$ for any vertices $i, j \in I.$ Further, we say that an independent vertex set $I$ is a \textit{maximal independent vertex set} if $I \cup \{i\}$ is not an independent vertex set for any vertex $i \notin I.$ We denote $\alpha(G) = \max\{\lvert I \rvert : I \text{ is an independent vertex set of } G\}.$
\end{definition}

\begin{remark}\label{dim k(G)}
By \cite[Theorem 4.3.6]{MRSW}, there exists an irredundant primary decomposition of $I(G)$ in which each ideal is generated by the variables corresponding to the vertices of a distinct minimal vertex cover $C.$ Considering that the height of an ideal in this primary decomposition is $\lvert C \rvert,$ by Proposition \ref{main proposition about cs and ms}, we have that $\ms(G) \geq \dim k(G) = n - \height(I(G)) = n - \tau(G) = \alpha(G),$ where the last equality holds by \cite[Lemma 3.1.21]{West}.
\end{remark}

\begin{cor}\label{ms(G) = 1 if and only if G is complete}
We have that $\ms(k(G)) = 1$ if and only if $G$ is a complete graph.
\end{cor}

\begin{proof}
One direction holds by Proposition \ref{ms(K_n)}. Conversely, if $G$ is not complete, then $\{v, w\}$ is not an edge of $G$ for some vertices $v$ and $w$ of $G,$ i.e., it is an independent vertex set. We conclude that $\ms(k(G)) \geq \alpha(G) \geq 2.$
\end{proof}

If $k$ is infinite, we may also obtain bounds for $\ms(k(G))$ based on the number of edges of $G.$

\begin{prop}\label{ms(k(G)) leq n - 1}
Let $G$ be a simple graph with $n$ vertices with at least one edge $\{i, j\}.$ If the field $k$ is infinite, then $\ms(k(G)) \leq n - 1.$ Even more, if $G$ has at least $n + 1$ edges, then $\ms(k(G)) \leq n - 2.$
\end{prop}

\begin{proof}
Considering that $\{i, j\}$ is an edge, the monomial $\bar x_i \bar x_j$ vanishes in $k(G).$ Consequently, we have that $\mu(\bar \fm^2) < \binom{\mu(\bar \fm) + 1} 2.$ If $k$ is infinite, then $\ms(k(G)) \leq \mu(\bar \fm) - 1 = n - 1$ by Corollary \ref{upper bound for ms in terms of upper bound for mu(m^2) 2}. Further, if $G$ has at least $n + 1$ edges, then $\overline G$ has at most $\binom n 2 - n - 1$ edges, hence we have that $\mu(\bar \fm^2) \leq \binom n 2 - 1 < \binom n 2.$
\end{proof}

We will now demonstrate that if the complement graph $\overline G$ satisfies a certain property on its induced subgraphs, then $\ms(k(G))$ is precisely the independence number of $G.$ Before we do so, we make the following definitions.

\begin{definition}\label{chordal definition}
We say that a finite simple graph $G$ is {\it chordal} if it has no induced subgraph that is isomorphic to a cycle graph $C_i$ for any integer $i \geq 4.$
\end{definition}

We refer to an induced subgraph of $G$ that is isomorphic to a cycle graph $C_i$ with $i \geq 3$ as an {\it induced cycle of length} $i.$ Consequently, if $G$ has no induced cycles of length $i \geq 4,$ it is chordal.

\begin{definition}\label{mcn(G) definition}
Given a finite simple graph $G$ such that $\overline G$ is not chordal, we define $$\mcn(G) = \min \{i \geq 4 \mid C_i \text{ is an induced subgraph of } \overline G\}.$$
\end{definition}

\begin{prop}\label{alpha(G) leq n - mcn(G) + 3}
Let $G$ be a simple graph on $n$ vertices. If $\overline G$ is not chordal, then $\alpha(G) \leq n - \mcn(G) + 3.$
\end{prop}

\begin{proof}
By the exposition preceding Definition \ref{clique-adjacent definition}, we have that $\alpha(G) = \omega(\overline G),$ i.e., the maximum size of a clique of $\overline G.$ Let $Q$ be a clique of $\overline G$ of size $\omega(\overline G).$ Let $C$ be a cycle of $\overline G$ of size $\mcn(G) \geq 4.$ We claim that no more than three vertices of $G$ lie in both $Q$ and $C.$ On the contrary, if $i \geq 4$ vertices of $G$ lie in both $Q$ and $C,$ then the induced subgraph $H$ of $\overline G$ on these $i$ vertices must be $K_i$ because $Q$ is a clique. But this is a contradiction: $H$ is an induced subgraph of the $\mcn(G)$-cycle $C,$ which does not admit $K_i$ as an induced subgraph for $i \geq 4.$ We conclude that the number of vertices that $Q$ and $C$ have in common is no more than three, hence we find that $\alpha(G) + \mcn(G) = \omega(\overline G) + \mcn(G) = \lvert V(Q) \rvert + \lvert V(C) \rvert = \lvert V(Q) \cup V(C) \rvert + \lvert V(Q) \cap V(C) \rvert \leq n + 3.$
\end{proof}

\begin{prop}\label{ms(k(G)) leq n - mcn(G) + 3}
Let $G$ be a simple graph on the vertex set $[n].$ Let $k$ be an infinite field. If $\overline G$ is chordal, then $\ms(k(G)) = \alpha(G).$ If $\overline G$ is not chordal, then $\ms(k(G)) \leq n - \mcn(G) + 3.$
\end{prop} 

\begin{proof}
We will assume first that $\overline G$ is chordal. By \cite[Theorem 1]{Fr90}, the edge ideal $I(G)$ of $k[x_1, \dots, x_n]$ has a linear resolution. Put another way, the minimal free resolution of $I(G)$ is linear for $k - 1$ steps for each integer $k \geq 1.$ Particularly, the minimal free resolution $I(G)$ is linear for $\height(I(G)) - 1$ steps. By \cite[Corollary 5.2]{EHU}, we conclude that $\fm^2 \subseteq I(G) + L$ for any ideal $L$ generated by $n - \height(I(G))$ linearly independent general linear forms, where $\fm$ denotes the maximal irrelevant ideal of $k[x_1, \dots, x_n].$ Consequently, we have that $\ms(k(G)) \leq \mu(L) = n - \height(I(G)) = \alpha(G).$ By Proposition \ref{main proposition about cs and ms}(1.) and Remark \ref{dim k(G)}, we conclude that $\ms(k(G)) = \alpha(G).$

We will assume now that $\overline G$ is not chordal, i.e., $\overline G$ has an induced cycle of length $i \geq 4.$ Crucially, observe that $\overline G$ has no induced cycles of length $4 \leq i \leq \mcn(G) - 1 = (\mcn(G) - 3) + 2.$ By \cite[Theorem 2.7]{DHS}, we have that $I(G)$ is $(\mcn(G) - 4) = [(\mcn(G) - 3) - 1]$-steps linear. By Proposition \ref{alpha(G) leq n - mcn(G) + 3}, we have that $\mcn(G) - 3 \leq n - \alpha(G) = n - \dim k(G) = \height(I(G)),$ and \cite[Corollary 5.2]{EHU} yields $\ms(k(G)) \leq n - \mcn(G) + 3.$
\end{proof}

\begin{cor}\label{ms for the complement of a tree}
Let $G$ be a finite simple graph with at least two vertices. Let $k$ be an infinite field. If $\overline G$ has no induced cycles (i.e., if $\overline G$ is a tree), then $\ms(k(G)) = 2.$
\end{cor}

\begin{proof}
Observe that $\overline G$ is chordal with $\alpha(G) = \omega(\overline G) = 2$ (cf. the exposition preceding Definition \ref{clique-adjacent definition}).
\end{proof}

\begin{remark}\label{ms(k(G)) < n - mcn(G) + 3}
For a finite simple graph $G$ whose complement graph $\overline G$ is not chordal, it is possible for the upper bound produced in Proposition \ref{ms(k(G)) leq n - mcn(G) + 3} to be strict, as the following illustrates.
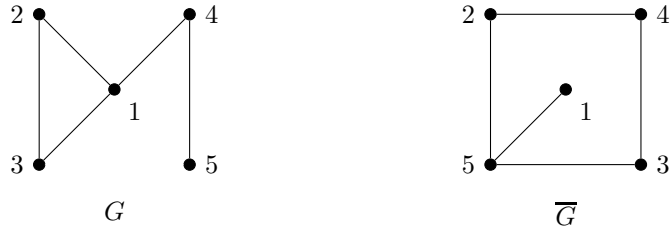
\begin{center}
\begin{tikzpicture}
\tikzset{enclosed/.style={draw, circle, inner sep=0pt, minimum size=.15cm, fill=black}}
\node[enclosed, label={below right: 1}] (1) at (-2,0) {};
\node[enclosed, label={left: 2}] (2) at (-3,1) {};
\node[enclosed, label={left: 3}] (3) at (-3,-1) {};
\node[enclosed, label={right: 4}] (4) at (-1,1) {};
\node[enclosed, label={right: 5}] (5) at (-1,-1) {};
\draw (1) -- (2);
\draw[color=white] (3) -- (5) node[midway, label={below, yshift=-.25cm: $\textcolor{black}{G}$}] (edge4) {};
\draw (1) -- (3);
\draw (1) -- (4);
\draw (2) -- (3);
\draw (4) -- (5);
\node[enclosed, label={below right: 1}] (1) at (4,0) {};
\node[enclosed, label={left: 2}] (2) at (3,1) {};
\node[enclosed, label={right: 3}] (3) at (5,-1) {};
\node[enclosed, label={right: 4}] (4) at (5,1) {};
\node[enclosed, label={left: 5}] (5) at (3,-1) {};
\draw (2) -- (4);
\draw (4) -- (3);
\draw (3) -- (5) node[midway, label={below, yshift=-.25cm: $\textcolor{black}{\overline G}$}] (edge0) {};
\draw (5) -- (2);
\draw (5) -- (1);
\end{tikzpicture}
\end{center}
Observe that $\mcn(G) = 4$ so that $\ms(k(G)) \leq 4 = n - \mcn(G) + 3$ by Proposition \ref{ms(k(G)) leq n - mcn(G) + 3}; however, the ideal $J = (x_1 + x_2, x_1 + x_3, x_4 + x_5)$ satisfies $\fm^2 \subseteq I(G) + J$ so that $\ms(k(G)) \leq 3.$
\end{remark}

\begin{remark}\label{ms(k(G)) = dim k(G) does not imply chordal}
The converse of the first part of Proposition \ref{ms(k(G)) leq n - mcn(G) + 3} does not hold: the finite simple graph $G$ whose complement is pictured below satisfies $\overline G$ is not chordal and $\ms(k(G)) = \alpha(G).$
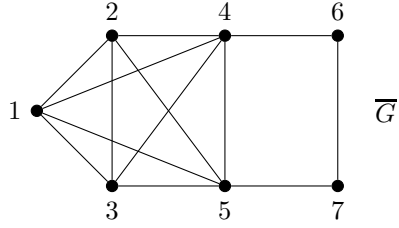
\begin{center}
\begin{tikzpicture}
\tikzset{enclosed/.style={draw, circle, inner sep=0pt, minimum size=.15cm, fill=black}}
\node[enclosed, label={left: 1}] (1) at (2,0) {};
\node[enclosed, label={above: 2}] (2) at (3,1) {};
\node[enclosed, label={below: 3}] (3) at (3,-1) {};
\node[enclosed, label={above: 4}] (4) at (4.5,1) {};
\node[enclosed, label={below: 5}] (5) at (4.5,-1) {};
\node[enclosed, label={above: 6}] (6) at (6,1) {};
\node[enclosed, label={below: 7}] (7) at (6,-1) {};
\draw (1) -- (2);
\draw (1) -- (3);
\draw (1) -- (4);
\draw (1) -- (5);
\draw (2) -- (3);
\draw (2) -- (4);
\draw (2) -- (5);
\draw (3) -- (4);
\draw (3) -- (5);
\draw (4) -- (5);
\draw (4) -- (6);
\draw (5) -- (7);
\draw (6) -- (7) node[midway, label={right, xshift=.25cm: $\textcolor{black}{\overline G}$}] (edge0) {};
\end{tikzpicture}
\end{center}
By Proposition \ref{upper bound for ms in terms of upper bound for mu(m^2)}, if $k$ is infinite and $n + \#E(\overline G) < \binom{r + 2}{r},$ then $\ms(k(G)) \leq r,$ where $\#E(\overline G)$ denotes the number of edges of $\overline G.$ Observe that $\#E(\overline G) = \binom 5 2 + 3$ so that $n + \#E(\overline G) = 20 < \binom{5 + 2} 2$ and $\ms(k(G)) \leq 5.$ On the other hand, the vertex set $I = \{1, 2, 3, 4, 5\}$ of $G$ is independent, hence we have that $5 \leq \alpha(G) \leq \ms(k(G)).$ We conclude that $\ms(k(G)) = \alpha(G),$ but $\overline G$ is not chordal, as $\overline G[\{4, 5, 6, 7\}]$ is isomorphic to the four-cycle $C_4.$

Generally, this construction produces a family of graphs $G_n$ on $n \geq 7$ vertices with $\ms(k(G_n)) = \alpha(G_n) = n - 2$ whose complement graphs $\overline{G_n}$ are not chordal. Explicitly, define $\overline{G_n}$ to be the complete graph $K_{n - 2}$ on the vertices $1, 2, \dots, n - 2$ adjoined with the edges $\{n - 3, n - 1\},$ $\{n - 2, n\},$ and $\{n - 1, n\}.$ Observe that $$n + \#E(\overline{G_n}) = n + \binom{n - 2} 2 + 3 = \frac{2n + (n - 2)(n - 3) + 6} 2 = \frac{n^2 - 3n + 12} 2,$$ from which it follows that $$\binom n 2 - n - \#E(\overline{G_n}) = \frac{n^2 - n} 2 - \frac{n^2 - 3n + 12} 2 = \frac{2n - 12} 2 = n - 6 > 0$$ for all integers $n \geq 7.$ By Proposition \ref{upper bound for ms in terms of upper bound for mu(m^2)}, we conclude that $\ms(k(G_n)) \leq n - 2.$ Conversely, the vertex set $I = \{1, 2, \dots, n - 2\}$ of $G$ is independent so that $n - 2 \leq \alpha(G_n) \leq \ms(k(G_n)).$

We note that this construction also provides a simple graph $G$ on $n$ vertices such that $\ms(k(G)) = \alpha(G)$ and $\alpha(G) < n - \mcn(G) + 3$ whenever $k$ is infinite. By the second part of Proposition \ref{ms(k(G)) leq n - mcn(G) + 3}, if $\alpha(G) = n - \mcn(G) + 3,$ then $\ms(k(G)) = \alpha(G)$; however, for each integer $n \geq 7,$ the complement graph $\overline{G_n}$ is not chordal but it holds that $\ms(k(G_n)) = \alpha(G_n),$ and we have that $\alpha(G_n) = n - 2 < n - \mcn(G) + 3.$
\end{remark}

Our next proposition partially answers the question posed prior to Definition \ref{chordal definition}.

\begin{prop}\label{ms(K_n) minus non-adjacent edges}
Let $G$ be the graph obtained from the complete graph $K_n$ on $n \geq 3$ vertices by removing $1 \leq \ell \leq {\left \lfloor \frac n 2 \right \rfloor}$ pairwise non-adjacent edges. If $k$ is infinite, then we have that $\ms(k(G)) = 2.$
\end{prop}

\begin{proof}
Observe that $\overline G$ consists of $n - 2 \ell$ isolated vertices and $\ell$ pairwise non-adjacent edges. Consequently, $\overline G$ is chordal, and we have that $\ms(k(G)) = \alpha(G) = 2$ by Proposition \ref{ms(k(G)) leq n - mcn(G) + 3}.
\end{proof}

If $k$ is an infinite field, then Proposition \ref{ms(k(G)) leq n - 1} demonstrates that $\ms(k(G)) \leq n - 1.$ Our next proposition investigates a class of graphs with $\alpha(G) = n - 1.$

\begin{prop}\label{cs(S_n) and ms(S_n)}
Let $S_n$ be the star graph on $n \geq 3$ vertices, i.e., the graph with edges $\{1, i\}$ for each integer $2 \leq i \leq n.$ We have that $\cs(k(S_n)) = n$ and $\ms(k(S_n)) = n - 1.$
\end{prop}

\begin{proof}
Observe that $I(S_n) = (x_1 x_i \mid 2 \leq i \leq n)$ is generated by $t = n - 1$ homogeneous polynomials of degree two. By Proposition \ref{cs of k[x_1, ..., x_n] modulo t homogeneous quadratics}, we conclude that $\cs(k(S_n)) = n.$

Consider the ideal $J = (\bar x_1 + \bar x_i \mid 2 \leq i \leq n).$ We find that $\bar x_1^2 = \bar x_1(\bar x_1 + \bar x_i)$ and $\bar x_i^2 = \bar x_i(\bar x_1 + \bar x_i)$ are in $\bar \fm J$ for all integers $2 \leq i \leq n.$ Likewise, we have that $\bar x_i \bar x_j = \bar x_j(\bar x_1 + \bar x_i)$ is in $\bar \fm J$ for all integers $2 \leq i < j \leq n.$ We conclude that $\bar \fm^2 \subseteq \bar \fm J$ so that $\ms(k(S_n)) \leq n - 1.$

Conversely, the vertices $2, 3, \dots, n$ are independent, hence we have that $\ms(k(S_n)) \geq n - 1.$
\end{proof}

\begin{remark}\label{ms(S_n) = n - 1 if k is infinite}
If $k$ is infinite, then we have that $\ms(k(S_n)) = n - 1$ by Proposition \ref{ms(k(G)) leq n - mcn(G) + 3}: the complement graph $\overline{S_n}$ is the graph union of $K_1$ and the complete graph $K_{n - 1},$ hence it is chordal with $\alpha(S_n) = n - 1.$ For a detailed explanation of this argument, see the discussion preceding Proposition \ref{cs(G*H) and ms(G*H)}.
\end{remark}

If $\mcn(G) = 4$ and $k$ is infinite, then $\ms(k(G)) \leq n - 4 + 3 = n - 1$; however, this upper bound follows already from Proposition \ref{ms(k(G)) leq n - 1}. Consequently, Proposition \ref{ms(k(G)) leq n - mcn(G) + 3} does not provide any new information. Going forward, our aim is to understand graphs with $\mcn(G) = 4.$ We make use of the following terminology.

\begin{definition}\label{gap-free definition}
We say that a finite simple graph $G$ is {\it gap-free} if $\overline G$ has no induced cycle of length $4.$ Put another way, $G$ is gap-free if and only if either $\overline G$ is chordal or $\mcn(G) \geq 5.$
\end{definition}

Considering that an induced cycle of length 4 in $\overline G$ arises from a pair of non-adjacent edges in $G$ that are not connected by a third edge, it follows that many familiar graphs are not gap-free and therefore satisfy $\mcn(G) = 4.$ We present non-trivial bounds on $\ms(k(G))$ and $\cs(k(G))$ for these graphs whenever possible.

\begin{prop}\label{cs(P_n) and ms(P_n)}
Let $P_n$ be the path graph on $n \geq 4$ vertices with edges $\{i, i + 1\}$ for each integer $1 \leq i \leq n - 1.$ We have that $\cs(k(P_n)) = n$ and ${\left \lceil \frac n 2 \right \rceil} \leq \ms(k(P_n)) \leq n - 1.$
\end{prop}

\begin{proof}
We note that the edge ideal $I(P_n) = (x_i x_{i+1} \mid 1 \leq i \leq n - 1)$ is generated by $t = n-1$ homogeneous polynomials of degree two. By Proposition \ref{cs of k[x_1, ..., x_n] modulo t homogeneous quadratics}, we conclude that $\cs(k(P_n)) = n.$

Consider the ideals $J = (\bar x_i + \bar x_{i+1} \mid 1 \leq i \leq n - 1)$ and $\bar \fm^2 = (\bar x_i^2 \mid 1 \leq i \leq n) + (\bar x_i \bar x_j \mid 3 \leq i + 2 \leq j \leq n)$ of $k(P_n).$ We obtain the pure squares $\bar x_i^2$ of $\bar \fm^2$ in $\bar \fm J$ by taking the products $\bar x_i(\bar x_i + \bar x_{i+1}).$ Further, the mixed terms $\bar x_i \bar x_j$ with $3 \leq i + 2 \leq j \leq n$ can be obtained as follows.
\begin{enumerate}

\item[(i.)] Using the fact that $\bar x_i \bar x_{i+1}$ vanishes in $k(P_n),$ we have that $\bar x_i(\bar x_{i+1} + \bar x_{i+2}) = \bar x_i \bar x_{i+2}.$

\vspace{0.25cm}

\item[(ii.)] Using the previous step, it follows that $\bar x_i \bar x_{i+3} = \bar x_i(\bar x_{i+2} + \bar x_{i+3}) - \bar x_i \bar x_{i+2}$ is in $\bar \fm J.$

\vspace{0.25cm}

\item[(iii.)] Continue in this manner to obtain $\bar x_i \bar x_k$ for all integers $i + 2 \leq k \leq j.$

\end{enumerate}
We conclude that $\bar \fm^2 \subseteq \bar \fm J,$ from which it follows that $\ms(k(P_n)) \leq n - 1.$

On the other hand, we have that $\alpha(P_n) = {\left \lceil \frac n 2 \right \rceil}$ so that $\ms(k(P_n)) \geq {\left \lceil \frac n 2 \right \rceil}$ by Remark \ref{dim k(G)}. Explicitly, the collection of ${\left \lceil \frac n 2 \right \rceil}$ odd vertices of $P_n$ is a maximum independent vertex set.
\end{proof}

\begin{remark}\label{ms(k(P_4)) = 2}
We have that $\ms(k(P_4)) = 2,$ as the ideal $J = (\bar x_1 + \bar x_2, \bar x_3 + \bar x_4)$ of $k(P_4)$ contains $\bar \fm^2.$ On the other hand, the path graph $P_n$ is not gap-free for any integer $n \geq 5.$ Given that $k$ is infinite, we may conclude that $\ms(k(P_n)) \leq n - 1$ by Proposition \ref{ms(k(G)) leq n - 1}.
\end{remark}

\begin{prop}\label{cs(C_n) and ms(C_n)}
Let $C_n$ be the cycle graph on $n \geq 3$ vertices, i.e., the path graph $P_n$ together with the edge $\{1, n\}.$ We have that $n - 1 \leq \cs(k(C_n)) \leq n$ and $\ms(k(C_n)) \geq {\left \lfloor \frac n 2 \right \rfloor}.$ If $n \geq 3$ is odd, then $\cs(k(C_n)) = n - 1.$ If $n = 3,$ then $\ms(k(C_n)) = {\left \lfloor \frac n 2 \right \rfloor}.$ For $n \leq 7,$ we have that $\ms(k(C_n)) \leq {\left \lceil \frac{n}{2} \right \rceil}.$ Equality holds for $n = 4$ and $n = 6.$
\end{prop}

\begin{proof}
Observe that $I(C_n) = (x_i x_{i+1} \mid 1 \leq i \leq n-1) + (x_1 x_n)$ is generated by $t = n$ quadratic monomials. By Proposition \ref{cs of k[x_1, ..., x_n] modulo t homogeneous quadratics} with $s = 0$ and Proposition \ref{main proposition about cs and ms}, we have $n - 1 \leq \cs(k(C_n)) \leq n.$

Clearly, the cycle graph $C_3$ is isomorphic to the complete graph $K_3,$ hence we have that $\ms(k(C_3)) = 1 = {\left \lfloor \frac 3 2 \right \rfloor}$ by Proposition \ref{ms(K_n)}. Generally, we find that $\ms(k(C_n)) \geq \alpha(C_n) = {\left \lfloor \frac n 2 \right \rfloor}$ by Remark \ref{dim k(G)}, as the ${\left \lfloor \frac n 2 \right \rfloor}$ even vertices of $C_n$ form a maximum independent vertex set.

Consequently, it remains to prove the last two statements. We will exhibit for each integer $4 \leq n \leq 7$ an ideal $J$ of $k(C_n)$ such that $\mu(J) = {\left \lceil \frac{n}{2} \right \rceil}$ and $\bar \fm^2$ is contained in $\bar \fm J.$

For $n = 4,$ we have $I(C_n) = (x_1 x_2, x_2 x_3, x_3 x_4, x_1 x_4)$ and $\bar \fm^2 = (\bar x_1^2, \bar x_2^2, \bar x_3^2, \bar x_4^2) + (\bar x_1 \bar x_3, \bar x_2 \bar x_4).$ Consider the ideal $J = (\bar x_1 + \bar x_2, \bar x_3 + \bar x_4)$ of $k(C_n).$ We obtain the pure squares $\bar x_i^2$ of $\bar \fm^2$ in $\bar \fm J$ by taking the products $\bar x_i(\bar x_i + \bar x_{i+1})$ and $\bar x_{i+1}(\bar x_i + \bar x_{i+1})$ for $i = 1$ and $i = 3.$ We obtain the terms $\bar x_1 \bar x_3$ and $\bar x_2 \bar x_4$ by taking $\bar x_1(\bar x_3 + \bar x_4)$ and $\bar x_2(\bar x_3 + \bar x_4).$ We conclude that $\ms(k(C_4)) = 2 = {\left \lceil \frac 4 2 \right \rceil}.$

For $n = 5,$ we have $I(C_n) = (x_1 x_2, x_2 x_3, x_3 x_4, x_4 x_5, x_1 x_5)$ and $$\bar \fm^2 = (\bar x_1^2, \bar x_2^2, \bar x_3^2, \bar x_4^2, \bar x_5^2) + (\bar x_1 \bar x_3, \bar x_1 \bar x_4, \bar x_2 \bar x_4, \bar x_2 \bar x_5, \bar x_3 \bar x_5).$$ Consider the ideal $J = (\bar x_1 + \bar x_2, \bar x_3 + \bar x_4, \bar x_4 + \bar x_5)$ of $k(C_n).$ We obtain the pure squares $\bar x_i^2$ of $\bar \fm^2$ in the same fashion as in the previous paragraph. Further, we obtain the mixed terms by taking $\bar x_3(\bar x_1 + \bar x_2),$ $\bar x_1(\bar x_4 + \bar x_5),$ $\bar x_2(\bar x_3 + \bar x_4),$ $\bar x_2(\bar x_4 + \bar x_5) - \bar x_2 \bar x_4,$ and $\bar x_3(\bar x_4 + \bar x_5),$ respectively.

For $n = 6,$ we have $I(C_n) = (x_i x_{i+1} \mid 1 \leq i \leq 5) + (x_1 x_6)$ and $$\bar \fm^2 = (\bar x_i^2 \mid 1 \leq i \leq 6) + (\bar x_i \bar x_j \mid 3 \leq i + 2 \leq j \leq 6 \text{ and } j \leq i + 4).$$ Consider the ideal $J = (\bar x_1 + \bar x_2, \bar x_3 + \bar x_4, \bar x_5 + \bar x_6)$ of $k(C_n).$ We obtain the pure squares $\bar x_i^2$ of $\bar \fm^2$ as in the previous two paragraphs; then, we obtain the mixed terms sequentially by first gathering all of the $\bar x_i \bar x_j$ such that $j = i + 2.$ Explicitly, we have that $\bar x_3(\bar x_1 + \bar x_2) = \bar x_1 \bar x_3,$ $\bar x_2(\bar x_3 + \bar x_4) = \bar x_2 \bar x_4,$ $\bar x_5(\bar x_3 + \bar x_4) = \bar x_3 \bar x_5,$ and $\bar x_6(\bar x_4 + \bar x_5) = \bar x_4 \bar x_6$ in $k(C_n).$ We use these to gather all of the $\bar x_i \bar x_j$ such that $j = i + 3.$ We have that $\bar x_1(\bar x_3 + \bar x_4) - \bar x_1 \bar x_3 = \bar x_1 \bar x_4,$ $\bar x_2(\bar x_4 + \bar x_5) - \bar x_2 \bar x_4 = \bar x_2 \bar x_5,$ and $\bar x_3(\bar x_5 + \bar x_6) - \bar x_3 \bar x_5 = \bar x_3 \bar x_6.$ Last, we gather the terms $\bar x_i \bar x_j$ such that $j + i = 4.$ We have that $\bar x_1(\bar x_5 + \bar x_6) = \bar x_1 \bar x_5$ and $\bar x_6(\bar x_1 + \bar x_2) = \bar x_2 \bar x_6$ in $k(C_n).$ We conclude that $\bar \fm^2 \subseteq \bar \fm J$ so that $\ms(k(C_6)) \leq 3.$ Considering that $\alpha(C_6) = {\left \lceil \frac 6 2 \right \rceil} = 3,$ we have that $\ms(k(C_6)) = {\left \lceil \frac 6 2 \right \rceil}.$

For $n = 7,$ we have $I(C_n) = (x_i x_{i+1} \mid 1 \leq i \leq 6) + (x_1 x_7)$ and $$\bar \fm^2 = (\bar x_i^2 \mid 1 \leq i \leq 7) + (\bar x_i \bar x_j \mid 3 \leq i + 2 \leq j \leq 7 \text{ and } j \leq i + 5).$$ Consider the ideal $J = (\bar x_1 + \bar x_2, \bar x_3 + \bar x_4, \bar x_5 + \bar x_6, \bar x_6 + \bar x_7)$ of $k(C_n).$ We obtain the pure squares $\bar x_i^2$ of $\bar \fm^2$ as in the previous three paragraphs and the mixed terms as in the previous paragraph.

Finally, if $n = 2m + 1$ for some integer $m \geq 1,$ then the ideal $J = (\bar x_1 + \bar x_i \mid 2 \leq i \leq 2m + 1)$ of $k(C_n)$ satisfies $\bar \fm^2 = J^2.$ Observe that $I(C_n) = (x_i x_{i + 1} \mid 1 \leq i \leq 2m) + (x_1 x_{2m + 1})$ and $$\bar \fm^2 = (\bar x_i^2 \mid 1 \leq i \leq 2m + 1) + (\bar x_i \bar x_j \mid 3 \leq i + 2 \leq j \leq 2m + 1).$$ We obtain the pure squares $\bar x_i^2$ of $\bar \fm^2$ for each integer $1 \leq i \leq 2m + 1$ by taking
\begin{align*}
\bar x_1^2 &= \sum_{j = 2}^{2m} (-1)^j (\bar x_1 + \bar x_j)(\bar x_1 + \bar x_{j + 1}) \text{ and} \\
\bar x_i^2 &= (\bar x_1 + \bar x_i)^2 + \sum_{j = 2}^{i - 1} (-1)^{i + j} (\bar x_1 + \bar x_j)(\bar x_1 + \bar x_{j + 1}) + \sum_{j = i}^{2m} (-1)^{i + j + 1} (\bar x_1 + \bar x_j)(\bar x_1 + \bar x_{j + 1}).
\end{align*}
Considering that the pure squares of $\bar \fm^2$ belong to $J^2,$ we obtain the mixed terms as follows.
\begin{enumerate}

\item[(i.)] We have that $\bar x_1 \bar x_3 = (\bar x_1 + \bar x_2)(\bar x_1 + \bar x_3) - \bar x_1^2.$

\vspace{0.25cm}

\item[(ii.)] We have that $\bar x_1 \bar x_4 = (\bar x_1 + \bar x_3)(\bar x_1 + \bar x_4) - \bar x_1^2 - \bar x_1 \bar x_3.$

\vspace{0.25cm}

\item[(iii.)] Continuing in this manner, we obtain all mixed terms $\bar x_1 \bar x_j$ with $2 \leq j \leq 2m + 1.$

\vspace{0.25cm}

\item[(iv.)] We obtain the remaining mixed terms $\bar x_i \bar x_j$ for some integers $4 \leq i + 2 \leq j \leq 2m + 1$ by observing that $\bar x_i \bar x_j = (\bar x_1 + \bar x_i)(\bar x_1 + \bar x_j) - \bar x_1^2 - \bar x_1 \bar x_i - \bar x_1 \bar x_j.$

\end{enumerate}
We conclude that $\cs(k(C_n)) = \cs(k(C_{2m + 1})) = 2m = n - 1.$
\end{proof}

\begin{cor}\label{ms(P_5) and ms(P_6)}
We have that $\ms(k(P_5)) = 3$ and $3 \leq \ms(k(P_6)) \leq 4.$
\end{cor}

\begin{proof}
Observe that $P_n$ is isomorphic to the induced subgraph $C_{n + 1}[\{1, 2, \dots, n\}].$ By Propositions \ref{cs and ms of induced subgraphs}, \ref{cs(P_n) and ms(P_n)}, and \ref{cs(C_n) and ms(C_n)}, we conclude that $\left \lceil \frac n 2 \right \rceil \leq \ms(k(P_n)) \leq \ms(k(C_{n + 1})) \leq \left \lceil \frac{n + 1} 2 \right \rceil$ for each integer $n \leq 6.$
\end{proof}

\begin{remark}\label{obstruction to ms(C_n) upper bound}
Unfortunately, the technique of the proof of Proposition \ref{cs(C_n) and ms(C_n)} fails to produce an ideal that witnesses $\ms(k(C_n))$ for $n \geq 8.$ Consider the cycle graph $C_8$ with edge ideal $I(C_8) = (x_i x_{i + 1} \mid 1 \leq i \leq 7) + (x_1 x_8).$ Using the same approach as in the proof, the ideal $J = (\bar x_1 + \bar x_2, \bar x_3 + \bar x_4, \bar x_5 + \bar x_6, \bar x_7 + \bar x_8)$ of $k(C_8)$ is a prospective witness of $\ms(k(C_8));$ however, one can show that the element $\bar x_1 \bar x_5$ of $\bar \fm^2$ is not contained in $J.$
\end{remark}

\begin{remark}\label{better bounds for cs(C_n) and ms(C_n)}
We believe that $\ms(k(C_n)) \leq {\left \lceil \frac{n}{2} \right \rceil}$ does not hold for all integers $n \geq 8$ whenever $k$ has characteristic zero. Using the following code in Macaulay2, one can generate any number of random ideals in $\QQ(C_8)$ and subsequently test whether such an ideal witnesses $\ms(\QQ(C_8)).$
\begin{verbatim}
loadPackage "EdgeIdeals";
loadPackage "RandomIdeals";

setRandomSeed(currentTime())

-- Declare the file to which a witness ideal will be written.
file = "msWitnessIdealsC_n";

-- Declare the number of variables for the cycle.
n = 8

-- Establish the polynomial ring and its homogeneous maximal ideal.
R = QQ[x_1 .. x_n];
m = ideal(vars R);

-- Declare the number of random ideals to test.
numRuns = 1000000;

-- Define the graph G and the edge ideal of G in R.
G = cycle R;
I = edgeIdeal G;

-- Specify the generators. The first entry is the degree of a monomial.
B = basis(1, R);

-- Create a list whose length is the total number of possible generators.
-- The integers in the list prescribe how Macaulay2 will randomly choose a
-- monomial of this degree (randomly means that 0 is a possible coefficient,
-- so the 0 polynomial could appear). E.g., in the following list, Macaulay2
-- will randomly choose four degree one terms.
L = {1,1,1,1};

-- This begins the loop. First, we generate a random ideal J; then, we test
-- if m^2 is a subset of I + J. If it is, then J is written to the file, and
-- a new line is created. The loop ends after numRuns iterations.
for iter from 1 to numRuns do (
  J = randomIdeal(L, B);
  if isSubset(m^2, I + J) then {
    file << J << endl;
  };);

file << close;
\end{verbatim}
Consistently, for one million random ideals, we could not find an ideal $J$ of $k(C_8)$ with $\mu(J) = 4$ and $\bar \fm^2 \subseteq J.$ Certainly, further testing is required to determine if this is the case for larger values.

On the other hand, consider the linear forms $y_i$ of $k(C_n)$ defined by $$y_i = \begin{cases} \bar x_i + \bar x_{n - 2} + \bar x_{n - 1} + \bar x_n &\text{if } i \equiv 0 \modulo[4], \\ \bar x_i - \bar x_{n - 2} - \bar x_{n - 1} - \bar x_n &\text{if } i \equiv 1 \modulo[4], \\ \bar x_i - \bar x_{n - 2} - \bar x_{n - 1} + \bar x_n &\text{if } i \equiv 2 \modulo[4], \text{ and} \\ \bar x_i - \bar x_{n - 2} + \bar x_{n - 1} + \bar x_n &\text{if } i \equiv 3 \modulo[4]. \end{cases}$$ We have found that for all integers $8 \leq n \leq 40,$ the ideal $J = (y_i \mid 1 \leq i \leq n - 3)$ satisfies $\bar \fm^2 \subseteq J$ in $k(C_n).$ Consequently, we have that $\ms(k(C_n)) \leq n - 3$ for all integers $8 \leq n \leq 40.$

Even more, the proof of Proposition \ref{cs(C_n) and ms(C_n)} does not settle the case of $\cs(k(C_n))$ when $n$ is even. Using the following code in Macaulay2, we have not found an ideal $J$ of $\QQ(C_{2n})$ with $\mu(J) = 2n - 1$ such that $\bar \fm^2 = J^2,$ hence we believe that $\cs(k(C_{2n})) = 2n$ if $k$ has characteristic zero.
\begin{verbatim}
loadPackage "EdgeIdeals";
loadPackage "RandomIdeals";

setRandomSeed(currentTime())

-- Declare the file to which a witness ideal will be written.
file = "csWitnessIdealsC_2n";

-- Declare the largest value of n to test.
n = 20

-- Declare the number of random ideals to test in each iteration.
numRuns = 10000;

-- Run a loop.
for i from 2 to n do (
  n = 2*i;
  R = QQ[x_1 .. x_n];
  m = ideal(vars R);
  G = cycle R;
  I = edgeIdeal G;
  B = basis(1, R);
  -- Define the witness ideal J.
  L = {};
  for j from 1 to n - 1 do (
    L = append(L, 1));
  for iter from 1 to numRuns do (
    J = randomIdeal(L, B);
    if isSubset(m^2, I + J^2) then {
        file << i << endl;
  };););

file << close;
\end{verbatim}
\end{remark}

Our next proposition aims for a partial explanation of the difficulty of computing $\ms(k(C_n))$ for $n$ sufficiently large. Before we are able to state it, we must recall the following terminology. Let $G$ be a simple graph on the vertex set $[n] = \{1, 2, \dots, n\}.$ We say that a nonempty set $Q \subseteq [n]$ forms a {\it clique} in $G$ whenever $\{i, j\}$ is an edge of $G$ for any vertices $i, 
j \in Q.$ Consequently, $Q$ is a clique of $G$ if and only if the induced subgraph of $G[Q]$ is isomorphic to the complete graph $K_{\lvert Q \rvert}.$ Likewise, a nonempty set $\mathcal Q \subseteq [n]$ is a {\it coclique} if the vertices of $\mathcal Q$ form a clique of $\overline G.$ Consequently, $\mathcal Q$ is a coclique of $G$ if and only if the induced subgraph $\overline G[\mathcal Q]$ is isomorphic to the complete graph $K_{\lvert \mathcal Q \rvert}$ if and only if the induced subgraph $G[\mathcal Q]$ is isomorphic to the empty graph $\overline{K_{\lvert \mathcal Q \rvert}}$ if and only if $\mathcal Q$ is an independent vertex set of $G.$ Consequently, we have that $$\alpha(G) = \max \{\lvert \mathcal Q \rvert : \mathcal Q \text{ is a coclique of } G \} = \max \{\lvert \mathcal Q \rvert : \mathcal Q \text{ is a clique of } \overline G \} = \omega(\overline G),$$ where $\omega(G)$ denotes the {\it clique number} of $G.$ Crucially, any edge of $G$ is trivially a clique.

One other important invariant of the graph $G$ is its {\it vertex clique cover number}, i.e., the minimum number of cliques in $G$ needed to cover all of the vertices of $G$ $$\theta(G) = \min {\left\{\ell \mid Q_1, \dots, Q_\ell \text{ are cliques of } G \text{ and } [n] = \bigcup_{i = 1}^\ell Q_i \right\}}.$$ For instance, the vertex clique cover number of $P_n$ and $C_n$ is ${\left \lceil \frac{n}{2} \right \rceil}.$ For if $n = 2 \ell,$ then a minimum clique covering is achieved by the edges $\{1, 2\}, \{3, 4\}, \dots, \{2 \ell - 1, 2 \ell\};$ on the other hand, if $n = 2 \ell + 1,$ then a minimum clique covering is achieved by the edges $\{1, 2\}, \{3, 4\}, \dots, \{2 \ell - 1, 2 \ell\}, \{2 \ell + 1\}.$

\begin{comment}
We recall the fact that the clique covering number of $G$ is equal to the {\it chromatic number} $\chi(\overline G).$
\end{comment}

\begin{definition}\label{clique-adjacent definition}
Let $G$ be a simple graph on the vertices $[n]$ with vertex clique cover $[n] = \bigcup_{i = 1}^\ell Q_i.$ We say that two cliques $Q_i$ and $Q_j$ are \textit{clique-adjacent} if one of the following conditions holds.
\begin{enumerate}

\item[(a.)] There exist vertices $v \in Q_i$ and $w \in Q_j$ such that $\{v, w\}$ is an edge of $G.$

\vspace{0.25cm}

\item[(b.)] There exists a vertex $v \in Q_i \cap Q_j,$ i.e., the cliques $Q_i$ and $Q_j$ ``overlap'' at a vertex.

\end{enumerate}
Further, if each pair $Q_i$ and $Q_j$ of cliques are clique-adjacent, we say that $\{Q_i\}_{i = 1}^\ell$ is $K_\ell$-\textit{connected}. We say that $G$ is $K_\ell$-connected if it admits a vertex clique cover $\{Q_i\}_{i = 1}^\ell$ that is $K_\ell$-connected.
\end{definition}

\begin{definition}\label{clique graph definition}
Let $G$ be a simple graph on the vertex set $[n]$ with a clique covering $[n] = \bigcup_{i = 1}^\ell Q_i.$ We define the \textit{clique graph} $\cc(G, X)$ of $G$ induced by the vertex clique cover $X = \{Q_i\}_{i = 1}^\ell$ to be the simple graph whose vertices are the cliques $Q_1, \dots, Q_\ell$ together with the edges $\{Q_i, Q_j\}$ for all integers $1 \leq i < j \leq \ell$ such that $Q_i$ and $Q_j$ are clique-adjacent (as defined in Definition \ref{clique-adjacent definition}).
\end{definition}

\begin{example}
Observe that $Q_1 = \{1, 2\},$ $Q_2 = \{3, 4\},$ and $Q_3 = \{5, 6\}$ is a minimum clique covering of the cycle graph $C_6.$ Because the edges $\{2, 3\},$ $\{4, 5\},$ and $\{1, 6\}$ belong to $C_6,$ the clique $Q_1$ is clique-adjacent to both $Q_2$ and $Q_3,$ and the clique $Q_2$ is clique-adjacent to $Q_3.$ Consequently, the clique graph $\cc(C_6, X)$ of $C_6$ induced by the clique covering $X = \{Q_1, Q_2, Q_3\}$ has vertices $Q_1, Q_2,$ and $Q_3$ and edges $\{Q_1, Q_2\}, \{Q_1, Q_3\},$ and $\{Q_2, Q_3\}.$ Put another way, we have that $\cc(C_6, X) \cong K_3,$ i.e., $C_6$ is $K_3$-connected.
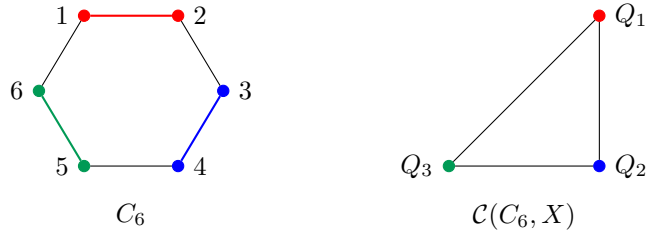
\begin{center}
\begin{tikzpicture}
\tikzset{enclosed/.style={draw, circle, inner sep=0pt, minimum size=.15cm, fill=black}}
\node[enclosed, label={left: 6}, color=ForestGreen] (6) at (-5.6,-1) {};
\node[enclosed, label={left: 1}, color=red] (1) at (-5,0) {};
\node[enclosed, label={left: 5}, color=ForestGreen] (5) at (-5,-2) {};
\node[enclosed, label={right: 2}, color=red] (2) at (-3.75,0) {};
\node[enclosed, label={right: 4}, color=blue] (4) at (-3.75,-2) {};
\node[enclosed, label={right: 3}, color=blue] (3) at (-3.15,-1) {};
\draw[thick, color=red] (1) -- (2);
\draw (2) -- (3);
\draw[thick, color=blue] (3) -- (4);
\draw (4) -- (5) node[midway, label={below, yshift=-.25cm: $C_6$}] (edge4) {};
\draw[thick, color=ForestGreen] (5) -- (6);
\draw (1) -- (6);
\node[enclosed, label={right: $Q_1$}, color=red] (Q1) at (1.85,0) {};
\node[enclosed, label={right: $Q_2$}, color=blue] (Q2) at (1.85,-2) {};
\node[enclosed, label={left: $Q_3$}, color=ForestGreen] (Q3) at (-0.15,-2) {};
\draw (Q1) -- (Q2);
\draw (Q2) -- (Q3) node[midway, label={below, yshift=-.25cm: $\cc(C_6, X)$}] (edge3) {};
\draw(Q1) -- (Q3);
\end{tikzpicture}
\end{center}
\end{example}

Crucially, the cycle graph $C_n$ on $n \leq 6$ vertices is $K_{{\left \lceil n/2 \right \rceil}}$-connected; however, for any integer $n \geq 7,$ we have that $\cc(C_n, X) \cong C_{{\left \lceil n/2 \right \rceil}}$, where $X$ is the clique covering from the paragraph preceding Definition \ref{clique-adjacent definition}. Because $X$ is a minimum clique covering of $C_n,$ it follows that $C_n$ is not $K_\ell$-connected for any integer $\ell \geq 1.$

Every connected graph has a trivial vertex clique covering by all of its edges. We refer to a vertex clique cover of $G$ by edges as an {\it edge cover} of $G.$ Consequently, the clique cover $X$ of the paragraph preceding Definition \ref{clique-adjacent definition} is an edge cover of $C_n.$ Generalizing the idea of the proof of Proposition \ref{cs(C_n) and ms(C_n)} yields the following observation.

\begin{prop}\label{K_ell-connected edge cover}
If a finite simple graph $G$ admits an edge cover $X = \{E_i\}_{i = 1}^\ell$ such that the edges of $X$ are all clique-adjacent (i.e., $X$ is $K_\ell$-connected), then we have that $\ms(k(G)) \leq \ell.$
\end{prop}

\begin{proof}
Observe that $\bar \fm^2 = (\bar x_i^2 \mid 1 \leq i \leq n) + (\bar x_e \bar x_f \mid \{e,f\} \text{ is not an edge of } G).$ We claim that the ideal $J = \sum_{i = 1}^\ell (\bar x_r + \bar x_s \mid E_i = \{r, s\})$ satisfies $J \supseteq \bar \fm^2.$ By hypothesis that $X$ is an edge cover of $G,$ for each integer $1 \leq i \leq n,$ there exists an integer $j$ such that $\{i, j\}$ belongs to $X.$ Consequently, the terms $\bar x_i^2 = \bar x_i (\bar x_i + \bar x_j)$ of $\bar \fm^2$ belong to $\bar \fm J.$ We obtain all of the mixed terms $\bar x_e \bar x_f$ such that $\{e, f\}$ is not an edge of $G$ as follows.
\begin{enumerate}[\rm (i.)]

\item By hypothesis that the edges of $X$ are clique-adjacent, any two edges $\{e, e'\}$ and $\{f, f'\}$ of $X$ are either connected by an edge $\{e, f'\},$ $\{e', f\},$ or $\{e', f'\},$ or they ``overlap'' so that $e' = f'.$

\vspace{0.25cm}

\item If $\{e, f'\}$ is an edge, then $\bar x_e \bar x_f = \bar x_e(\bar x_f + \bar x_{f'})$ belongs to $\bar \fm J$; a similar argument shows that $\bar x_e \bar x_f$ belongs to $\bar \fm J$ if $\{e', f\}$ is an edge. If $\{e', f'\}$ is an edge, then $\bar x_e \bar x_{f'} = \bar x_{f'} (\bar x_e + \bar x_{e'})$ belongs to $\bar \fm J$ so that $\bar x_e \bar x_f = \bar x_e(\bar x_f + \bar x_{f'}) - \bar x_e \bar x_{f'}$ belongs to $\bar \fm J.$ If $e' = f',$ then $\bar x_e \bar x_f = \bar x_e(\bar x_f + \bar x_{e'}) = \bar x_e(\bar x_f + \bar x_{f'})$ belongs to $\bar \fm J.$ Observe that in any case, we conclude that $\bar x_e \bar x_f$ belongs to $\bar \fm J$ (and so must belong to $J$).

\end{enumerate}
We conclude that $\bar \fm^2 \subseteq \bar \fm J \subseteq J$ so that $\ms(k(G)) \leq \mu(J) = \ell,$ as desired.
\end{proof}

\begin{remark}\label{Wagner graph obstruction}
Recall that the {\it diameter} of a finite simple graph is the maximum distance of a shortest path connecting any two vertices. By the proof of Proposition \ref{K_ell-connected edge cover}, any finite simple graph satisfying the hypotheses of the proposition must have diameter at most three. Explicitly, the maximum occurs precisely when there exist vertices $e$ and $f$ such that $\{e, f\}$ is not an edge and the edges $\{e, e'\}$ and $\{f, f'\}$ do not overlap.

Observe that in a finite simple graph of diameter two, any two vertices $e$ and $f$ are connected by a path of length at most two, hence either $\{e, f\}$ is an edge or $\{e, e'\}$ and $\{e', f\}$ are distinct edges. Consequently, it is natural to wonder if $G$ has diameter two, then must any ${\left \lceil \frac{n - 1} 2 \right \rceil}$ edges of $G$ constitute an edge cover of $G$ such that the hypotheses of Proposition \ref{K_ell-connected edge cover} hold? For if this were the case, then we would have that $\ms(k(G)) \leq {\left \lceil \frac{n - 1} 2 \right \rceil}.$

Unfortunately, the answer is no. Even more, it is not the case that in a finite simple graph of diameter two, every pair of edges in an edge cover must be clique-adjacent. Consider the Wagner graph $M_8$ pictured below.
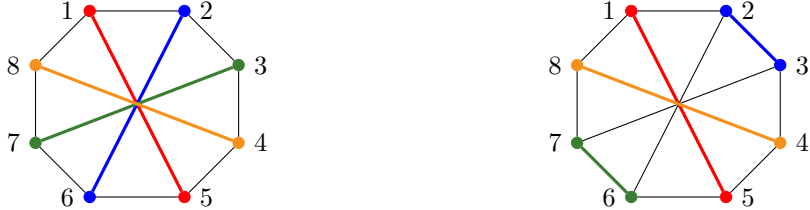
\begin{center}
\begin{tikzpicture}[scale=0.9]
\tikzset{enclosed/.style={draw, circle, inner sep=0pt, minimum size=.15cm, fill=black}}
\node[enclosed, label={left: 1}, color=red] (1) at (-8.7,3.75) {};
\node[enclosed, label={right: 2}, color=blue] (2) at (-7.3,3.75) {};
\node[enclosed, label={right: 3}, color=OliveGreen] (3) at (-6.5,2.95) {};
\node[enclosed, label={right: 4}, color=BurntOrange] (4) at (-6.5,1.8) {};
\node[enclosed, label={right: 5}, color=red] (5) at (-7.3,1) {};
\node[enclosed, label={left: 6}, color=blue] (6) at (-8.7,1) {};
\node[enclosed, label={left: 7}, color=OliveGreen] (7) at (-9.5,1.8) {};
\node[enclosed, label={left: 8}, color=BurntOrange] (8) at (-9.5,2.95) {};
\draw (1) -- (2);
\draw[very thick, color=red] (1) -- (5);
\draw (2) -- (3);
\draw[very thick, color=blue] (2) -- (6);
\draw (3) -- (4);
\draw[very thick, color=OliveGreen] (3) -- (7);
\draw (4) -- (5);
\draw[very thick, color=BurntOrange] (4) -- (8);
\draw (5) -- (6);
\draw (6) -- (7);
\draw (7) -- (8);
\draw (8) -- (1);
\node[enclosed, label={left: 1}, color=red] (1') at (-0.7,3.75) {};
\node[enclosed, label={right: 2}, color=blue] (2') at (0.7,3.75) {};
\node[enclosed, label={right: 3}, color=blue] (3') at (1.5,2.95) {};
\node[enclosed, label={right: 4}, color=BurntOrange] (4') at (1.5,1.8) {};
\node[enclosed, label={right: 5}, color=red] (5') at (0.7,1) {};
\node[enclosed, label={left: 6}, color=OliveGreen] (6') at (-0.7,1) {};
\node[enclosed, label={left: 7}, color=OliveGreen] (7') at (-1.5,1.8) {};
\node[enclosed, label={left: 8}, color=BurntOrange] (8') at (-1.5,2.95) {};
\draw (1') -- (2');
\draw[very thick, color=red] (1') -- (5');
\draw[very thick, color=blue] (2') -- (3');
\draw (2') -- (6');
\draw (3') -- (4');
\draw (3') -- (7');
\draw (4') -- (5');
\draw[very thick, color=BurntOrange] (4') -- (8');
\draw (5') -- (6');
\draw[very thick, color=OliveGreen] (6') -- (7');
\draw (7') -- (8');
\draw (8') -- (1');
\end{tikzpicture}
\end{center}
Upon inspection, we find that $M_8$ has diameter two. Further, the colored edges of both figures give a minimum clique covering (by maximal cliques); however, the pair of red and blue edges in the left-hand graph are not clique-adjacent. On the other hand, if we ``amend'' the left-hand edge cover to obtain the figure on the right, we have found a minimum clique covering of $M_8$ (by maximal cliques) that is clique-adjacent.

Crucially, every finite simple graph $H$ is an induced subgraph of a finite simple graph $G$ of diameter two. Explicitly, for any vertex $v$ that is not in $V(H),$ we may define $G = H*K_1,$ where $K_1$ is the complete graph on the vertex $v,$ i.e., it is simply the isolated vertex $v$ (cf. the second paragraph following Corollary \ref{ms(overline S_n)} for a definition of the graph operation $*$). Observe that $H \cong G[V(H)].$ Further, by Proposition \ref{cs(G*H) and ms(G*H)}, we have that $\ms(G) = \ms(H),$ hence it suffices to understand this invariant for finite simple graphs of diameter two.

Bearing all of these observations in mind, we ask the following question.
\end{remark}

\begin{question}\label{existence of K_ell-connected edge cover in diameter 2 triangle-free graph question}
If $G$ is a finite simple graph of diameter two, must it admit an edge cover that is $K_\ell$-connected? In particular, can any edge cover of $G$ be ``amended'' to an edge cover that is $K_\ell$-connected?
\end{question}

\begin{comment}
\begin{question}\label{upper bound for ms(G) in diameter 2 triangle-free graph question}
If $G$ is a finite simple graph of diameter two and vertex clique cover number two, must it hold that $\ms(k(G)) \leq 2?$ Must this bound hold if we assume that $G$ is connected with arbitrary diameter?
\end{question}
\end{comment}

Often, conjectures in graph theory are given a litmus test against the Petersen graph $P,$ as it is renowned among graph theorists for its consistent ability to produce counterexamples to many expected properties of graphs. Crucially, the Petersen graph is a simple connected graph of diameter two with 10 vertices and 15 edges. Even more, we have that $\alpha(P) = 4$ (cf. \cite[1.1.12]{West}). Our next proposition confirms that the Petersen graph does not contradict our findings thus far, hence in particular, Question \ref{existence of K_ell-connected edge cover in diameter 2 triangle-free graph question} remains open.

\begin{prop}\label{cs and ms of Petersen graph}
Let $P$ be the Petersen graph. We have that $8 \leq \cs(k(P)) \leq 10$ and $4 \leq \ms(k(P)) \leq 5.$
\end{prop}

\begin{proof}
By the preceding commentary and Remark \ref{dim k(G)}, we have that $\ms(k(P)) \geq \alpha(P) \geq 4.$ Conversely, we may realize the Petersen graph in the plane as a five-cycle connected to a pentagram by some ``spokes'' (cf. \cite{Graf}).
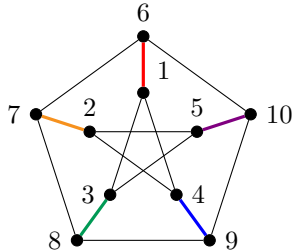
\begin{center}
\begin{tikzpicture}[scale=0.3]
\tikzset{enclosed/.style={draw, circle, inner sep=0pt, minimum size=.15cm, fill=black}}
\begin{scope}[rotate=90]
\foreach \x/\y in {0/1}{
\node[enclosed, label={above right: \y}] (\y) at (canvas polar cs: radius=2.5cm, angle=\x){};}
\foreach \x/\y in {216/4}{
\node[enclosed, label={right: \y}] (\y) at (canvas polar cs: radius=2.5cm, angle=\x){};}
\foreach \x/\y in {72/2, 288/5}{
\node[enclosed, label={\y}] (\y) at (canvas polar cs: radius=2.5cm, angle=\x){};}
\foreach \x/\y in {144/3}{
\node[enclosed, label={left: \y}] (\y) at (canvas polar cs: radius=2.5cm, angle=\x){};}
\foreach \x/\y in {0/6}{
\node[enclosed, label={\y}] (\y) at (canvas polar cs: radius=5cm,angle=\x){};}
\foreach \x/\y in {216/9, 288/10}{
\node[enclosed, label={right: \y}] (\y) at (canvas polar cs: radius=5cm,angle=\x){};}
\foreach \x/\y in {72/7, 144/8}{
\node[enclosed, label={left: \y}] (\y) at (canvas polar cs: radius=5cm,angle=\x){};}
\end{scope}
\foreach \x/\y in {1/6}{\draw[very thick, color=red] (\x) -- (\y);}
\foreach \x/\y in {2/7}{\draw[very thick, color=BurntOrange] (\x) -- (\y);}
\foreach \x/\y in {3/8}{\draw[very thick, color=ForestGreen] (\x) -- (\y);}
\foreach \x/\y in {4/9}{\draw[very thick, color=blue] (\x) -- (\y);}
\foreach \x/\y in {5/10}{\draw[very thick, color=violet] (\x) -- (\y);}
\foreach \x/\y in {1/3, 2/4, 3/5, 4/1, 5/2}{\draw (\x) -- (\y);}
\foreach \x/\y in {6/7, 7/8, 8/9, 9/10, 10/6}{\draw (\x) -- (\y);}
\end{tikzpicture}
\end{center}
One can readily verify that the colored edges pictured above induce a $K_5$-connected edge cover of $P$: indeed, the red edge shares a ``common edge'' with each of the other colored edges, hence it is clique-adjacent to each of the colored edges; the rest are clique-adjacent by symmetry. We conclude that $4 \leq \ms(P) \leq 5 = {\left \lceil \frac 9 2 \right \rceil}.$ Last, we have that $\lvert E(P) \rvert = 15 < 19$ so that $8 \leq \cs(k(P)) \leq 10$ by Propositions \ref{main proposition about cs and ms} and \ref{cs of k[x_1, ..., x_n] modulo t homogeneous quadratics} with $s = 1.$
\end{proof}

Consider two finite simple graphs $G$ and $H$ on the disjoint vertex sets $V(G)$ and $V(H)$ with respective edge sets $E(G)$ and $E(H).$ We recall that the {\it graph union} of $G$ and $H$ is the graph $G + H$ on the vertex set $V(G) \cup V(H)$ and edge set $E(G) \cup E(H).$ We illustrate this below for the blue graph $G$ and the red graph $H.$
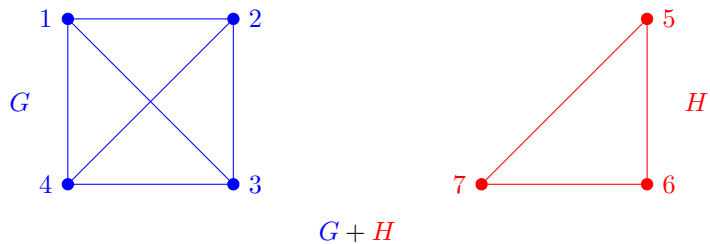
\begin{center}
\begin{tikzpicture}[scale=1.1]
\tikzset{enclosed/.style={draw, circle, inner sep=0pt, minimum size=.15cm, fill=black}}
\node[enclosed, label={left: \textcolor{blue}{1}}, color=blue] (1) at (-5,1) {};
\node[enclosed, label={right: \textcolor{blue}{2}}, color=blue] (2) at (-3,1) {};
\node[enclosed, label={right: \textcolor{blue}{3}}, color=blue] (3) at (-3,-1) {};
\node[enclosed, label={left: \textcolor{blue}{4}}, color=blue] (4) at (-5,-1) {};
\draw[color=blue] (1) -- (2);
\draw[color=blue] (2) -- (3);
\draw[color=blue] (3) -- (4);
\draw[color=blue] (1) -- (4) node[midway, label={left, xshift=-.25cm: $\textcolor{blue}{G}$}] {};
\draw[color=blue] (1) -- (3);
\draw[color=blue] (2) -- (4);
\node[enclosed, color=red, label={right: \textcolor{red}{5}}] (5) at (2,1) {};
\node[enclosed, color=red, label={right: \textcolor{red}{6}}] (6) at (2,-1) {};
\node[enclosed, color=red, label={left: \textcolor{red}{7}}] (7) at (0,-1) {};
\draw[color=red] (5) -- (6) node[midway, label={right,xshift=.25cm: $\textcolor{red}{H}$}] {};
\draw[color=red] (6) -- (7);
\draw[color=red] (5) -- (7);
\draw[color=white] (3) -- (7) node[midway, label={below, yshift=-0.25cm: $\textcolor{black}{\textcolor{blue}{G} + \textcolor{red}{H}}$}] {};
\end{tikzpicture}
\end{center}
Observe that $G$ and $H$ are both induced subgraphs of $G + H.$ Further, there are no edges between $G$ and $H$ in $G + H,$ hence every independent vertex set of $G + H$ is the disjoint union of some independent vertex set of $G$ and some independent vertex set of $H.$ Consequently, we obtain the following bounds.

\begin{prop}\label{a lower bound for cs(G + H) and ms(G + H)}
We have that $\ms(k(G + H)) \geq \max \{\ms(k(G)), \ms(k(H)), \alpha(G) + \alpha(H)\}$ and $\cs(k(G + H)) \geq \max \{\cs(k(G)), \cs(k(H))\}$ for any finite simple graphs $G$ and $H.$
\end{prop}

\begin{proof}
By Proposition \ref{cs and ms of induced subgraphs}, Remark \ref{dim k(G)}, and the exposition preceding the statement of this proposition, we have that $\ms(k(G + H)) \geq k(G),$ $\ms(k(G + H)) \geq k(H),$ and $\ms(k(G + H)) \geq \alpha(G + H) = \alpha(G) + \alpha(H),$ so the lower bound for $\ms(k(G + H))$ holds. Likewise, the lower bound for $\cs(k(G + H))$ holds by Proposition \ref{cs and ms of induced subgraphs}.
\end{proof}

\begin{cor}\label{ms(overline S_n)}
We have that $\ms(k(K_n + K_1)) = 2$ for all integers $n \geq 1.$
\end{cor}

\begin{proof}
By Proposition \ref{a lower bound for cs(G + H) and ms(G + H)}, it follows that $\ms(k(K_n + K_1)) \geq \alpha(K_n) + \alpha(K_1) = 2.$ Conversely, we have that $\ms(k(K_n + K_1)) \leq 1 + \ms(k(K_n)) = 2$ by Propositions \ref{ms(K_n)} and \ref{ms in terms of number of isolated vertices and nontrivial connected components}.
\end{proof}

\begin{comment}
Conversely, observe that $$\bar \fm^2 = (\bar x_i^2 \mid 1 \leq i \leq n + 1) + (\bar x_j \bar x_{n + 1} \mid 1 \leq j \leq n).$$ Consequently, the ideal $J = (\bar x_1 + \cdots + \bar x_n, \bar x_{n + 1})$ of $k(K_n + K_1)$ contains all of the pure squares $\bar x_i^2$ for each integer $1 \leq i \leq n + 1,$ as we have that $\bar x_i^2 = \bar x_i(\bar x_1 + \cdots + \bar x_n)$ for any integer $1 \leq i \leq n.$ On the other hand, the mixed terms $\bar x_j \bar x_{n + 1}$ belong to $\bar \fm J$ and so also to $J.$ We conclude that $\ms(k(K_n + K_1)) \leq 2.$
\end{comment}

Generally, the invariants $\ms(k(G + H))$ and $\cs(k(G + H))$ are difficult to understand because there are no relations between the vertices of $G$ and the vertices of $H.$ Explicitly, we have that $k(G + H) \cong k(G) \otimes_k k(H)$ (cf. Question \ref{ms and cs of tensor product of k-algebras}). Even still, the Macaulay2 code provided below is a good starting point to study these rings.
\begin{verbatim}
-- This script takes a quintuple (int_1, int_2, graph_1, graph_2, numLinearForms) and
-- defines the edge ring of the disjoint union G + H on int_1 + int_2 vertices. It defines
-- also a list L with numLinearForms copies of 1 from which a random ideal K can be created.
-- The inputs graph_1 and graph_2 are graphs from the EdgeIdeals package, e.g., the cycle
-- graph (cycle), the anticycle graph (antiCycle), the complete graph (completeGraph),
-- and the complete multipartite graph (completeMultiPartite).

loadPackage "EdgeIdeals";
loadPackage "RandomIdeals";

setRandomSeed(currentTime())

-- Declare the file to which a witness ideal will be written.
file = "msWitnessIdealsDisjointUnion";

-- Declare the number of variables of the first graph.
int_1 = read "How many vertices does your first graph have? ";
n_1 = value int_1;

-- Establish the polynomial ring for the first graph.
R = QQ[x_1 .. x_(n_1)];

-- Define the graph G and the edge ideal of G in R.
graph_1 = read "What type of graph would you like to consider? ";
graph_1 = value graph_1;
G = graph_1 R;
I = edgeIdeal G;

-- Declare the number of variables of the second graph.
int_2 = read "How many vertices does your second graph have? ";
n_2 = value int_2;

-- Establish the polynomial ring for the second graph.
S = QQ[x_(n_1 + 1) .. x_(n_1 + n_2)];

-- Define the graph H and the edge ideal of H in S.
graph_2 = read "What type of graph would you like to consider? ";
graph_2 = value graph_2;
H = graph_2 S;
J = edgeIdeal H;

-- Define the ambient polynomial ring and its homogeneous maximal ideal.
T = QQ[x_1 .. x_(n_1 + n_2)];
m = ideal(vars T);

-- Identify I and J as ideals of T.
f_1 = map(T, R);
I = f_1(I);
f_2 = map(T, S);
J = f_2(J);

-- Specify the generators. The first entry is the minimum degree of a monomial.
B = basis(1, T);

-- Create a list whose length is the total number of possible generators.
-- The integers in the list prescribe how Macaulay2 will randomly choose a
-- monomial of this degree (randomly means that 0 is a possible coefficient,
-- so the 0 polynomial could appear).
numLinearForms = read "How many linear forms would you like to consider? ";
len = value numLinearForms;
L = {};
for i from 1 to len do (
  L = append(L, 1));

-- This begins the loop. First, we generate a random ideal K; then, we test
-- if m^2 is a subset of I + J + K. If it is, then K is written to the file,
-- and a new line is created. The loop ends after numRuns iterations.
numTests = read "How many random ideals would you like to test? ";
numRuns = value numTests;
for iter from 1 to numRuns do (
  K = randomIdeal(L, B);
  if isSubset(m^2, I + J + K) then {
    file << K << endl;
  };);

file << close;
\end{verbatim}

We define the {\it graph join} $G*H$ of $G$ and $H$ as the graph union $G + H$ together with all edges joining a vertex of $G$ with a vertex of $H.$ For example, the graph join of two paths on $n = 2$ vertices is the complete graph $K_4.$
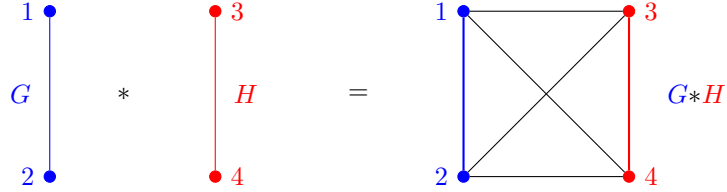
\begin{center}
\begin{tikzpicture}[scale=1.1]
\tikzset{enclosed/.style={draw, circle, inner sep=0pt, minimum size=.15cm, fill=black}}
\node[enclosed, label={left: \textcolor{blue}{1}}, color=blue] (1) at (-5,1) {};
\node[enclosed, label={left: \textcolor{blue}{2}}, color=blue] (2) at (-5,-1) {};
\draw[color=blue] (1) -- (2) node[midway, label={left: $\textcolor{blue}{G}$}] {} node[midway, label={right, xshift=.65cm: $\textcolor{black}{*}$}] {};
\node[enclosed, label={right: \textcolor{red}{3}}, color=red] (3) at (-3,1) {};
\node[enclosed, label={right: \textcolor{red}{4}}, color=red] (4) at (-3,-1) {};
\draw[color=red] (3) -- (4) node[midway, label={right: $\textcolor{red}{H}$}] (edge2) {};
\node[enclosed, label={left: \textcolor{blue}{1}}, color=blue] (11) at (0,1) {};
\node[enclosed, label={right: \textcolor{red}{3}}, color=red] (33) at (2,1) {};
\node[enclosed, label={right: \textcolor{red}{4}}, color=red] (44) at (2,-1) {};
\node[enclosed, label={left: \textcolor{blue}{2}}, color=blue] (22) at (0,-1) {};
\draw (11) -- (33);
\draw[thick, color=red] (33) -- (44) node[midway, label={right, xshift=.25cm: $\textcolor{blue}{G} \textcolor{black}{*} \textcolor{red}{H}$}] {};
\draw (44) -- (22);
\draw[thick, color=blue] (11) -- (22) node[midway, label={left, xshift=-1cm: $\textcolor{black}{=}$}] {};
\draw (11) -- (44);
\draw (22) -- (33);
\end{tikzpicture}
\end{center}
Observe that $G = (G*H)[V(G)]$ and $H = (G*H)[V(H)]$ are induced subgraphs of $G*H.$

Even more, the complement graph of the graph join $G*H$ is the graph union of the complements of $G$ and $H,$ i.e., we have that $\overline{G*H} = \overline G + \overline H.$ Explicitly, the pair $\{i, j\}$ is an edge of $G*H$ if and only if (a.) $\{i, j\}$ is an edge of $G$ or (b.) $\{i, j\}$ is an edge of $H$ or (c.) $i \in V(G)$ and $j \in V(H)$ or vice-versa, hence $\{i, j\}$ is an edge of $\overline{G*H}$ if and only if $\{i, j\}$ is neither an edge of $G$ nor an edge of $H$ nor an edge connecting some vertex of $G$ to some vertex of $H$ if and only if $\{i, j\}$ is an edge of $\overline G$ or $\overline H,$ i.e., $\{i, j\}$ is an edge of $\overline G + \overline H.$ Consequently, the graph invariants of $G*H$ can be described in terms of those of $\overline G + \overline H.$

Let $m$ and $n$ be positive integers. Let $G$ and $H$ be simple graphs on the vertex sets $V(G) = [m]$ and $V(H) = \{m + 1, m + 2, \dots, m + n\}.$ We have the edge ideals $I(G)$ and $I(H),$ edge rings $k(G) = k[x_1, \dots, x_m] / I(G)$ and $k(H) = k[x_{m + 1}, \dots, x_{m + n}] / I(H),$ and their respective irrelevant maximal ideals $\fm_G = (x_1, \dots, x_m)$ and $\fm_H = (x_{m + 1}, \dots, x_{m + n}).$ By definition of $G*H,$ observe that $V(G*H) = V(G) \cup V(H)$ and $$E(G*H) = E(G) \cup \{\{i, j\} \mid 1 \leq i \leq m \text{ and } m + 1 \leq j \leq n\} \cup E(H).$$ Consequently, the edge ideal of $G*H$ is given by $$I(G*H) = I(G) + (x_i x_j \mid 1 \leq i \leq m \text{ and } m + 1 \leq j \leq n) + I(H) = I(G) + \fm_G \fm_H + I(H).$$ Using this notation, we make the following observation.

\begin{prop}\label{cs(G*H) and ms(G*H)}
Let $G$ and $H$ be simple graphs on the vertex sets $V(G)$ and $V(H)$ as above. We have that $\ms(k(G*H)) = \max \{\ms(k(G)), \ms(k(H))\}$ and $\max \{\cs(k(G)), \cs(k(H))\} \leq \cs(k(G*H)) \leq \cs(k(G)) + \cs(k(H)).$
\end{prop}

\begin{proof}
Let $R = k[x_1, \dots, x_m, x_{m + 1}, \dots, x_n].$ Observe that the edge ideals $I(G)$ and $I(H)$ and the maximal ideals $\fm_G = (x_1, \dots, x_m)$ and $\fm_H = (x_{m + 1}, \dots, x_n)$ lie in $R$; they satisfy $I(G) + \fm_G \fm_H + I(H) = (I(G) + \fm_H) \cap (I(H) + \fm_G)$ by \cite[Proposition 2.21]{Gel}. Observe that $k(G) \cong R / (I(G) + \fm_H)$ and $k(H) \cong R / (I(H) + \fm_G).$ We will henceforth identify $k(G)$ and $k(H)$ with their quotients of $R.$ By the analog of \cite[Proposition 5.11]{Rot} in the category of commutative rings, the fiber product of $k(G)$ and $k(H)$ with respect to $k$ is the subring $$k(G) \times_k k(H) \stackrel{\text{def}} = \{(f + I(G) + \fm_H, g + I(H) + \fm_G) \mid f + \fm_G + \fm_H = g + \fm_G + \fm_H\} \subseteq k(G) \times k(H)$$ together with the restriction of the first-coordinate projection map $\pi_1 : k(G) \times_k k(H) \to k(G)$ and the restriction of the second-coordinate projection map $\pi_2 : k(G) \times_k k(H) \to k(H).$ Observe that the triple $(R, \pi_{k(G)}, \pi_{k(H)})$ with the ring homomorphisms defined by $\pi_{k(G)}(f) = f + I(G) + \fm_H$ and $\pi_{k(H)}(f) = f + I(H) + \fm_G$ satisfies the identity $\pi_G \circ \pi_{k(G)} = \pi_H \circ \pi_{k(H)}$ for the canonical surjections $\pi_G : k(G) \to k$ and $\pi_H : k(H) \to k,$ hence the analog of \cite[Proposition 5.11]{Rot} in the category of commutative rings yields a unique ring homomorphism $\theta : R \to k(G) \times_k k(H)$ defined by $\theta(f) = (f + I(G) + \fm_H, f + I(H) + \fm_G)$ with $\ker \theta = (I(G) + \fm_H) \cap (I(H) + \fm_G).$ By \cite[Proposition 2.1]{Gel}, we find that $\theta$ is surjective with $\ker \theta = I(G) + \fm_G \fm_H + I(H) = I(G*H),$ hence it induces a ring isomorphism $\bar \theta : k(G*H) \to k(G) \times_k k(H)$ by the First Isomorphism Theorem.

We claim that for any element $(g + I(G) + \fm_H, h + I(H) + \fm_G)$ of the fiber product $k(G) \times_k k(H)$ such that $g, h \in R$ are homogeneous polynomials of same degree, there exists a homogeneous polynomial $f \in R$ such that $\theta(f) = (g + I(G) + \fm_H, h + I(H) + \fm_G).$ Indeed, for any homogeneous polynomials $g, h \in R$ of positive degree, we may write $g = g_{G \setminus H} + g_{H \setminus G} + g_{G \cap H}$ and $h = h_{G \setminus H} + h_{H \setminus G} + h_{G \cap H}$ for some polynomials $g_{G \setminus H}, h_{G \setminus H} \in \fm_G \setminus \fm_H,$ $g_{H \setminus G}, h_{H \setminus G} \in \fm_H \setminus \fm_G,$ and $g_{G \cap H}, h_{H \cap G} \in \fm_G \cap \fm_H.$ Consequently, the polynomial $f = g_{G \setminus H} + h_{H \setminus G}$ satisfies $\theta(f) = (g + I(G) + \fm_H, h + I(H) + \fm_G).$ Even more, if $g$ and $h$ have the same degree, then $f$ is homogeneous.

Let $a = \max \{\ms(k(G)), \ms(k(H))\}.$ By Proposition \ref{ms equivalent to m^2 = mI inductive step}, there exist homogeneous linear forms $s_1, \dots, s_a$ in the variables $x_1, \dots, x_m$ whose images in $k(G)$ satisfy $\fm_G^2 = (s_1, \dots, s_a) \fm_G.$ Likewise, there exist homogeneous linear forms $t_1, \dots, t_a$ in the variables $x_{m + 1}, \dots, x_{m + n}$ whose images in $k(H)$ satisfy $\fm_H^2 = (t_1, \dots, t_a) \fm_H.$ By the surjectivity of $\theta,$ there exist polynomials $f_i \in R$ such that $\theta(f_i) = (s_i + I(G) + \fm_H, t_i + I(H) + \fm_G)$ for each integer $1 \leq i \leq a.$ Crucially, the elements $s_i$ and $t_i$ are homogeneous linear forms, so we may assume that the polynomials $f_i$ are homogeneous by the construction of the previous paragraph. By the proof of Proposition \ref{cs and ms of fiber product}, we have that $\fm^2_{k(G) \times_k k(H)} \subseteq ((s_1, t_1), \dots, (s_a, t_a)).$ Going modulo $I(G*H) = I(G) + \fm_G \fm_H + I(H),$ we find that $\fm^2_{k(G*H)} \subseteq (\bar f_1, \dots, \bar f_a).$ Because the polynomials $f_i$ are homogeneous, we conclude that $\ms(k(G*H)) \leq a$. Last, $\ms(k(G*H)) \geq a$ holds by Proposition \ref{comparison theorem for cs and ms under surjective ring homomorphism} since there are surjections $k(G*H) \to k(G)$ and $k(G*H) \to k(H).$ 

Likewise, the $\cs(k(G*H))$ bounds follow by the proof of Proposition \ref{cs and ms of fiber product} and the second paragraph here.
\end{proof}

Combining Propositions \ref{ms(k(G)) leq n - mcn(G) + 3} and \ref{cs(G*H) and ms(G*H)} yields two immediate corollaries.

\begin{cor}\label{ms(G*H) when overline G and overline H are chordal}
Let $G$ and $H$ be graphs on disjoint vertex sets $V(G)$ and $V(H)$ such that $\overline G$ and $\overline H$ are chordal. If $k$ is an infinite field, then $\ms(k(G*H)) = \max \{\alpha(G), \alpha(H)\}.$
\end{cor}

\begin{comment}
\begin{proof}
Observe that $\omega(\overline G + \overline H) = \max \{\omega(\overline G), \omega(\overline H)\}.$ If $\overline G$ and $\overline H$ are chordal, then $\overline{G*H} = \overline G + \overline H$ is chordal. If $k$ is infinite, the result follows by Proposition \ref{ms(k(G)) leq n - mcn(G) + 3} and $\alpha(G) = \omega(\overline G).$
\end{proof}
\end{comment}

\begin{cor}\label{ms of graph join of complete graphs minus an edge}
Let $m$ and $n$ be positive integers. Let $G$ be the graph obtained from the complete graph $K_m$ by removing $1 \leq i \leq {\left \lfloor \frac m 2 \right \rfloor}$ non-adjacent edges. Let $H$ be the graph obtained from the complete graph $K_n$ by removing $1 \leq j \leq {\left \lfloor \frac n 2 \right \rfloor}$ non-adjacent edges. If $k$ is infinite, then $\ms(k(G*H)) = 2.$
\end{cor}

\begin{proof}
Observe that $\overline G$ consists of $m - 2i$ isolated vertices and $i$ non-adjacent edges. Likewise, $\overline H$ consists of $n - 2j$ isolated vertices and $j$ non-adjacent edges. Consequently, the graphs $\overline G$ and $\overline H$ are chordal so that $\ms(k(G*H)) = \max \{\alpha(G), \alpha(H)\} = 2$ by Proposition \ref{ms(K_n) minus non-adjacent edges} and Corollary \ref{ms(G*H) when overline G and overline H are chordal}.
\end{proof}

For any positive integers $n_1, \dots, n_t,$ the graph $K_{n_1, n_2, \dots, n_t} = \overline{K_{n_1}} * \overline{K_{n_2}} * \cdots * \overline{K_{n_t}}$ is called the {\it complete $t$-partite graph} on $n_1, \dots, n_t.$ Considering that $k(\overline K_n) = k[x_1, \dots, x_n]$ is regular, we obtain the following.

\begin{cor}\label{ms(K_n_1, n_2, dots, n_t)}
We have that $\ms(K_{n_1, n_2, \dots, n_t}) = \max \{n_1, n_2, \dots, n_t\}.$
\end{cor}

\begin{proof}
Observe that $\ms(k(K_{n_1, \dots, n_t})) = \max \{\ms(k(\overline{K_{n_1}})), \dots, \ms(k(\overline{K_{n_t}}))\}$ by Proposition \ref{cs(G*H) and ms(G*H)}. By the observation preceding the statement of the corollary, the latter value is $\max \{n_1, \dots, n_t\}.$
\end{proof}

\begin{comment}
\begin{prop}\label{ms(G*H) leq max |V(G)|, |V(H)|}
Let $G$ and $H$ be graphs on the disjoint vertex sets $V(G)$ and $V(H)$ defined as above. We have that $\ms(k(G*H)) \leq \max \{m, n\}.$
\end{prop}

\begin{proof}
We may assume without loss of generality that $m < n.$ By Proposition \ref{K_ell-connected edge cover}, it suffices to find an edge cover with $n$ edges such that any pair of edges is clique-adjacent. Consider any edge cover $L$ of $G*H$ consisting of some $m$ edges from a vertex of $G$ to a vertex of $H$ and some $n - m$ edges from a vertex of $G$ to the remaining $n - m$ vertices of $H$ that have not already been covered. We claim that any two edges $E$ and $F$ of $L$ are clique-adjacent. If $E$ and $F$ are both incident to some vertex, then we are done; otherwise, the endpoints of $E$ and $F$ are distinct. By definition of $L,$ an endpoint $v$ of $E$ belongs to $G,$ and an endpoint $w$ of $F$ belongs to $H.$ Considering that $\{v, w\}$ is an edge of $G*H$ by definition, we conclude that $E$ and $F$ are clique-adjacent.
\end{proof}
\end{comment}

Given a finite simple graph $G$ on $n$ vertices, one naturally wonders if $\ms(k(G)) + \ms(k(\overline G)) = n.$ Our next proposition illustrates that this is not the case.

\begin{prop}\label{ms(S_n) + ms(overline S_n) = n + 1}
We have that $\ms(k(S_n)) + \ms(k(\overline{S_n})) = n + 1.$
\end{prop}

\begin{proof}
Observe that $S_n = \overline{K_{n - 1}} * K_1,$ hence $\overline{S_n} = K_{n - 1} + K_1$ so that $\ms(k(\overline{S_n})) = 2$ by Corollary \ref{ms(overline S_n)}. We conclude the result, as $\ms(k(S_n)) = n - 1$ by Proposition \ref{cs(S_n) and ms(S_n)}.
\end{proof}

\begin{comment}
We note that we have not yet given non-trivial bounds for $\cs(k(K_n)).$ Primarily, this is due to the fact that Proposition \ref{cs of k[x_1, ..., x_n] modulo t homogeneous quadratics} fails to produce a meaningful lower bound and $\ms(k(K_n)) = 1.$ Our next proposition uses the mechanics of the graph join to produce an upper bound.

\begin{cor}\label{cs(K_3n+r)}
We have that $2 \leq \cs(k(K_{3n + r})) \leq 2n + r$ for each integer $0 \leq r \leq 2.$
\end{cor}

\begin{proof}
Observe that $K_{3n + r}$ is the graph join of $n$ copies of $K_3$ and $K_r.$ By Proposition \ref{cs(C_n) and ms(C_n)}, we have that $\cs(k(K_3)) = 2 \geq \cs(k(K_r))$ for each integer $0 \leq r \leq 2,$ hence by the first part of Proposition \ref{main proposition about cs and ms} and Proposition \ref{cs(G*H) and ms(G*H)}, we conclude that $2 = \max \{\cs(k(K_3)), \cs(k(K_r))\} \leq \cs(k(K_{3n + r})) \leq n \cs(k(K_3)) + \cs(k(K_r)) \leq 2n + r.$
\end{proof}
\end{comment}

We provide bounds on our invariants for the wheel graph on $n \geq 4$ vertices. We will see that in turn, the following proposition yields a slightly improved upper bound on the cycle graph.

\begin{prop}\label{cs(W_n) and ms(W_n)}
Let $W_n$ be the wheel graph on $n \geq 4$ vertices, i.e., the graph join of the complete graph $K_1$ and the cycle graph $C_{n - 1}.$ The following inequalities hold.
\begin{align*}
n - 2 &\leq \cs(k(W_n)) \leq \begin{cases} n - 1 &\text{if } n \text{ is even and} \\ n &\text{if } n \text{ is odd} \end{cases} \\ \\
{\left \lfloor \frac{n - 1} 2 \right \rfloor} &\leq \ms(k(W_n)) \leq \begin{cases}
\displaystyle {\left \lceil \frac{n - 1} 2 \right \rceil} &\text{if } n \leq 7 \text{ and} \\
n - 3 &\text{if } n \geq 8. \end{cases}
\end{align*}
If $n = 4, 5, $ or 7, then $\ms(k(W_n))$ achieves its lower bound.
\end{prop}

\begin{proof}
Observe that $W_n = K_1 * C_{n - 1},$ hence Proposition \ref{cs(G*H) and ms(G*H)} implies that $$\cs(k(C_{n - 1})) = \max \{\cs(k(K_1)), \cs(k(C_{n - 1}))\} \leq \cs(k(W_n)) \leq \cs(k(K_1)) + \cs(k(C_{n - 1})) = \cs(k(C_{n - 1})) + 1$$ and $\ms(k(W_n)) = \max \{\ms(k(K_1)), \ms(k(C_{n - 1}))\} = \ms(k(C_{n - 1})).$ By Proposition \ref{cs(C_n) and ms(C_n)}, the stated lower bounds for $\cs(k(W_n))$ and $\ms(k(W_n))$ hold. If $n$ is even, then $n - 1$ is odd so that $\cs(C_{n - 1}) = n - 2;$ otherwise, we have that $\cs(k(C_{n - 1})) \leq n - 1,$ hence the upper bound for $\cs(k(W_n))$ holds. Likewise, if $n \leq 7,$ then the upper bound for $\ms(k(W_n))$ holds by Proposition \ref{cs(C_n) and ms(C_n)}. Last, we have that $\ms(k(W_4)) = \ms(k(C_3)) = 1$; $\ms(k(W_5)) = \ms(k(C_4)) = 2$; and $\ms(k(W_7)) = \ms(k(C_6)) = 3$ by Proposition \ref{cs(C_n) and ms(C_n)}.

For $n \geq 8,$ we construct an edge cover $\{E_i\}_{i = 1}^{n - 3}$ that is $K_{n - 3}$-connected. First, we cover two clique-adjacent ``perimeter'' edges of $W_n$; then, we cover the remaining $n - 5$ ``perimeter'' vertices and the ``hub'' with the $n - 5$ edges connecting the ``hub'' to each of these ``perimeter'' vertices. Ultimately, we obtain an edge cover with $n - 3 = 2 + (n - 5)$ edges. It is $K_{n - 3}$-connected because the two ``perimeter'' edges are clique-adjacent, and both of these edges are clique-adjacent to an edge connecting the ``hub'' and a ``perimeter'' vertex because the ``hub'' is adjacent to all ``perimeter'' vertices. By Proposition \ref{K_ell-connected edge cover}, we conclude that $\ms(k(W_n)) \leq n - 3.$
\end{proof}

\begin{cor}\label{ms(C_n) revisited}
We have that $\ms(k(C_n)) \leq n - 2$ for all integers $n \geq 8.$
\end{cor}

\begin{proof}
By Proposition \ref{cs(W_n) and ms(W_n)}, we have that $\ms(k(C_n)) = \ms(k(W_{n + 1})) \leq n - 2$ for all $n \geq 8.$
\end{proof}

We conclude with a discussion of another familiar graphical construction.

\begin{definition}\label{graph wedge definition}
Consider any finite simple graphs $G$ and $H$ with respective vertex sets $V(G)$ and $V(H)$ such that $V(G) \cap V(H) = \{v\}$ and respective edge sets $E(G)$ and $E(H).$ We define the \textit{wedge graph} $G \vee_v H$ with respect to $v$ as the graph with vertices $V(G) \cup V(H)$ and edges $E(G) \cup E(H).$ Put another way, $G \vee_v H$ is obtained by ``gluing'' $G$ and $H$ together at their common vertex $v.$
\end{definition}

Generally, the vertex $v$ determines the corresponding wedge graph $G \vee_v H.$ Explicitly, for any labelling of the vertices of $G,$ the wedge graph $G \vee_v H$ depends upon the labeling of the vertices of $H.$ Below is a diagram of two non-isomorphic graphs that are obtained by wedging two copies of $P_3$ with different labelings.
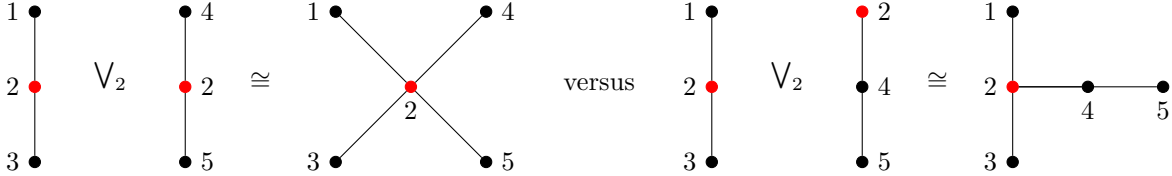
\begin{center}
\begin{tikzpicture}
\tikzset{enclosed/.style={draw, circle, inner sep=0pt, minimum size=.15cm, fill=black}}
\node[enclosed, label={left: 1}] (1) at (-10,1) {};
\node[enclosed, color=red, label={left: 2}] (2) at (-10,0) {};
\node[enclosed, label={left: 3}] (3) at (-10,-1) {};
\node[enclosed, label={right: 4}] (4) at (-8,1) {};
\node[enclosed, color=red, label={right: 2}] (5) at (-8,0) {};
\node[enclosed, label={right: 5}] (6) at (-8,-1) {};
\node[enclosed, color=white, label={above: $\textcolor{black}{\bigvee_2}$}] (v) at (-9,-0.25) {};
\node[enclosed, color=white, label={above: $\textcolor{black}{\cong}$}] (v) at (-7,-0.25) {};
\node[enclosed, label={left: 1}] (1') at (-6,1) {};
\node[enclosed, color=red, label={below: 2}] (2') at (-5,0) {};
\node[enclosed, label={left: 3}] (3') at (-6,-1) {};
\node[enclosed, label={right: 4}] (4') at (-4,1) {};
\node[enclosed, label={right: 5}] (6') at (-4,-1) {};
\draw (1) -- (2);
\draw (2) -- (3);
\draw (4) -- (5);
\draw (5) -- (6);
\draw (1') -- (2');
\draw (2') -- (3');
\draw (2') -- (4');
\draw (2') -- (6');
\node[enclosed, color=white, label={above: \textcolor{black}{versus}}] (v) at (-2.5,-0.25) {};
\node[enclosed, label={left: 1}] (12) at (-1,1) {};
\node[enclosed, color=red, label={left: 2}] (22) at (-1,0) {};
\node[enclosed, label={left: 3}] (32) at (-1,-1) {};
\node[enclosed, color=red, label={right: 2}] (42) at (1,1) {};
\node[enclosed, label={right: 4}] (52) at (1,0) {};
\node[enclosed, label={right: 5}] (62) at (1,-1) {};
\node[enclosed, color=white, label={above: $\textcolor{black}{\bigvee_2}$}] (v) at (0,-0.25) {};
\node[enclosed, color=white, label={above: $\textcolor{black}{\cong}$}] (v) at (2,-0.25) {};
\node[enclosed, label={left: 1}] (12') at (3,1) {};
\node[enclosed, color=red, label={left: 2}] (22') at (3,0) {};
\node[enclosed, label={left: 3}] (32') at (3,-1) {};
\node[enclosed, label={below: 4}] (52') at (4,0) {};
\node[enclosed, label={below: 5}] (62') at (5,0) {};
\draw (12) -- (22);
\draw (22) -- (32);
\draw (42) -- (52);
\draw (52) -- (62);
\draw (12') -- (22');
\draw (22') -- (32');
\draw (22') -- (52');
\draw (22') -- (62');
\end{tikzpicture}
\end{center}
On the other hand, for the complete graph $K_n,$ the wedge vertex is irrelevant, as every labeling of the vertices of $K_n$ induces a graph automorphism of $K_n.$ Because the diameter of $K_n$ is one, the wedge graphs $K_m \vee K_n$ form a family of graphs of diameter two. By the discussion of Remark \ref{Wagner graph obstruction}, graphs of diameter two are of particular interest because they are highly connected yet ostensibly exhibit subtle behavior with respect to the invariants.

Our first result on the wedge of complete graphs gives non-trivial bounds on $\ms(K_m \vee K_n).$

\begin{prop}\label{ms(K_m vee K_n)}
We have that $2 \leq \ms(k(K_m \vee K_n)) \leq \min \{m, n\}.$
\end{prop}

\begin{proof}
We will assume that $m < n$ so that $\min \{m, n\} = m$ and $\max \{m, n\} = n.$ Further, we will assume that $K_m$ is the complete graph on the vertices $[m] = \{1, 2, \dots, m\}$ and $K_n$ is the complete graph on the vertices $\{m, m + 1, \dots, m + n - 1\}.$ By definition, the graph $K_m \vee K_n$ is obtained by gluing $K_m$ and $K_n$ along the common vertex $m$ of $K_m$ and $K_n,$ hence we have that
\begin{align*}
I(K_m \vee K_n) &= (x_i x_j \mid 1 \leq i < j \leq m \text{ or } m \leq i < j \leq m + n - 1) \text{ and} \\
\bar \fm^2 &= (\bar x_i^2 \mid 1 \leq i \leq m + n - 1) + (\bar x_i \bar x_j \mid 1 \leq i \leq m - 1 \text{ and } m + 1 \leq j \leq m + n - 1).
\end{align*}
Consider the ideal $J = (\bar x_1, \bar x_2, \dots, \bar x_{m - 1}, \bar x_m + \bar x_{m + 1} + \cdots + \bar x_{m + n - 1}).$ We obtain the pure squares $\bar x_i^2$ for each integer $m \leq i \leq m + n - 1$ by taking $\bar x_i(\bar x_m + \bar x_{m + 1} + \cdots + \bar x_{m + n - 1}),$ and the remaining pure squares $\bar x_i^2$ for each integer $1 \leq i \leq m - 1$ are clearly contained in $\bar \fm J.$ We obtain the mixed terms $\bar x_i \bar x_j$ such that $1 \leq i \leq m - 1$ and $m + 1 \leq j \leq m + n - 1$ in the following manner.
\begin{enumerate}

\item[(i.)] Observe that $\bar x_i(\bar x_m + \bar x_{m + 1} + \cdots + \bar x_{m + n - 1})$ contains $\bar x_i \bar x_j$ as a summand.

\vspace{0.25cm}

\item[(ii.)] Given any integer $\ell \in \{m, m + 1, \dots, m + n - 1\} \setminus \{j\},$ observe that the monomial $\bar x_i \bar x_\ell$ is an element of $\bar \fm J$ by hypothesis that $1 \leq i \leq m - 1.$

\vspace{0.25cm}

\item[(iii.)] Consequently, we have that $\bar x_i \bar x_j = \bar x_i(\bar x_m + \bar x_{m + 1} + \bar x_{m + n - 1}) - \sum_{\ell \in S} \bar x_i \bar x_\ell$ is an element of $\bar \fm J,$ where $S$ is the set $\{m, m + 1, \dots, m + n - 1\} \setminus \{j\}.$

\end{enumerate}
We conclude that $\ms(k(K_m \vee K_n)) \leq \min \{m, n\}.$ On the other hand, we have that $\alpha(K_m \vee K_n) = 2$ so that $\ms(k(K_m \vee K_n)) \geq 2.$ Every vertex $1 \leq i \leq m - 1$ is incident to all vertices $1 \leq i < j \leq m,$ hence the set $\{i, j\}$ with $m + 1 \leq j \leq m + n - 1$ is a maximum independent vertex set.
\end{proof}

\begin{remark}
We note that the proof of Proposition \ref{ms(K_m vee K_n)} comes from the following observation. Let $G$ be a simple graph on $n$ vertices. Relabelling $G,$ if necessary, if the vertices $\{1, 2, \dots, m\}$ form an independent vertex set in $\overline G$ and the induced subgraph on $\{m + 1, m + 2, \dots, n\}$ is connected in $\overline G,$ then the ideal $\bar \fm^2$ of $k(G)$ is contained in $J = (\bar x_1 + \cdots + \bar x_m, \bar x_{m + 1}, \dots, \bar x_n).$ Indeed, if the vertices $\{1, 2, \dots, m\}$ are independent in $\overline G,$ then the graph $G$ must contain all the edges $\{i, j\}$ such that $1 \leq i < j \leq i.$ Consequently, the pure squares $\bar x_i^2$ such that $1 \leq i \leq m + 1$ belong to the ideal $\bar \fm J,$ as they can be written as $\bar x_i^2 = \bar x_i(\bar x_1 + \cdots + \bar x_m).$ Clearly, the pure squares $\bar x_{m + 1}^2, \dots, \bar x_n^2$ all belong to $\bar \fm J.$ Further, the mixed terms $\bar x_i \bar x_j$ such that $m + 1 \leq i < j \leq n$ belong to the ideal $\bar \fm J,$ as the monomials $\bar x_{m + 1}, \dots, \bar x_n$ all belong to $J$ by construction.
\end{remark}

\begin{comment}
Unfortunately, this observation is useless if $G$ is not very connected because then $\overline G$ is very connected, and we have that $\mu(J)$ is quite large.
\end{comment}

\begin{cor}\label{ms(T_n)}
Let $T_n$ be the kite graph on $n \geq 5$ vertices, i.e., the wedge graph of the complete graphs $K_{n - 2}$ and $K_2.$ We have that $\ms(k(T_n)) = 2.$
\end{cor}

\begin{remark}\label{lollipop graph}
The kite graph of Corollary \ref{ms(T_n)} is also called the {\it lollipop graph} $L_{n - 2, 1}.$
\end{remark}

Our next goal is to extend Corollary \ref{ms(T_n)} to the wedge graph of $K_m$ and $n$ copies of $K_2.$

\begin{definition}\label{jellyfish graph definition}
Let $m$ and $n$ be positive integers. We define the \textit{jellyfish graph} $J_{m, n}$ with body of size $m$ and $n$ tentacles as the simple graph $J_{m, n} = K_m \vee (\vee_{i = 1}^n K_2)$ on $m + n$ vertices with an edge $\{i, j\}$ whenever $1 \leq i < j \leq m$ or $i = 1$ and $m + 1 \leq j \leq m + n.$
\end{definition}

\begin{prop}\label{complement graph of J_m,n}
Let $m$ and $n$ be positive integers. We have that $\overline{J_{m, n}} \cong (\overline{K_{m - 1}} * K_n) + K_1.$ Particularly, the complement graph $\overline{J_{m, n}}$ is chordal with $\omega(\overline{J_{m, n}}) = n + 1.$
\end{prop}

\begin{proof}
Observe that $\{i, j\}$ is an edge of $\overline{J_{m, n}}$ if and only if $2 \leq i \leq m$ and $m + 1 \leq j \leq m + n$ or $m + 1 \leq i < j \leq m + n.$ Consequently, the vertex $i = 1$ of $\overline{J_{m, n}}$ forms a copy of $K_1$; the vertices $2 \leq i \leq m$ of $\overline{J_{m, n}}$ form a copy of $\overline{K_{m - 1}}$; the vertices $m + 1 \leq i \leq m + n$ of $\overline{J_{m, n}}$ form a clique $K_n$; and each vertex of $\overline{K_{m - 1}}$ is adjacent to a vertex of $K_n.$ We conclude that $\overline{J_{m, n}} \cong (\overline{K_{m - 1}} * K_n) + K_1.$ Considering that this is a union of chordal graphs, it follows that $\overline{J_{m, n}}$ is chordal. Further, observe that $\omega(\overline{J_{m, n}}) = \omega(\overline{K_{m - 1}} * K_n) = \alpha(K_{m - 1} + \overline{K_n}) = n + 1.$
\end{proof}

\begin{prop}\label{ms of J_m,n}
Let $m$ and $n$ be positive integers. We have that $\ms(J_{m, n}) = n + 1.$
\end{prop}

\begin{proof}
By Proposition \ref{complement graph of J_m,n}, we have that $\overline{J_{m, n}}$ is chordal with $\omega(\overline{J_{m, n}}) = n + 1.$ Consequently, it follows that $\ms(J_{m, n}) = \alpha(J_{m, n}) = \omega(\overline{J_{m, n}}) = n + 1$ by Proposition \ref{ms(k(G)) leq n - mcn(G) + 3}.
\end{proof}

\begin{comment}
\begin{question}\label{cs(K_m vee K_n)}
What is $\cs(k(K_m \vee K_n))?$ Can we find a non-trivial bound?
\end{question}
\end{comment}

\section{Further directions}\label{Further Directions}

One lingering question concerns the tensor product of standard graded algebras over a field $k.$ Consider $R = k[x_1, \dots, x_m] / I$ for some homogeneous quadratic ideal $I$ of $k[x_1, \dots, x_m]$ and $S = k[y_1, \dots, y_n] / J$ for some homogeneous quadratic ideal $J$ of $k[y_1, \dots, y_n].$ Observe that $I + J$ is an ideal of $k[x_1, \dots, x_m, y_1, \dots y_n]$ and $$R \otimes_k S \cong \frac{k[x_1, \dots, x_m, y_1, \dots, y_n]}{I + J}.$$ We note that the quadratic squarefree monomials $\bar x_i \bar y_j$ of $R \otimes_k S$ do not vanish.

\begin{question}\label{ms and cs of tensor product of k-algebras}
Let $R$ and $S$ be defined as above. What are $\ms(R \otimes_k S)$ and $\cs(R \otimes_k S)$?
\end{question}

Earlier, in Section \ref{Computing ms(R) and cs(R) for the edge ring of a finite simple graph}, we saw that even if $I$ and $J$ are quadratic squarefree monomial ideals, the above question is quite subtle. We provided some Macaulay2 code toward verifying this in the paragraph after Corollary \ref{ms(overline S_n)}.

One other interesting graphical invariant related to Section \ref{Computing ms(R) and cs(R) for the edge ring of a finite simple graph} can be defined as follows. Let $\mathbf X = (X_1, \dots, X_m)$ be an $m$-dimensional random vector with multivariate normal (or Gaussian) distribution $X \sim \mathcal N_m(\boldsymbol \mu, \mathbf \Sigma),$ where $\mathbf \Sigma$ is an $m \times m$ positive-semidefinite matrix known as the {\it covariance matrix}. Consider the finite simple graph $G$ on the vertices $[m]$ with an edge $\{i, j\}$ if and only if the random variables $X_i$ and $X_j$ are conditionally dependent given all of the other random variables (cf. \cite[Corollary 2.2]{Uhl}). By the paragraph following \cite[Problem 1.2]{GrSu}, we may define the {\it maximum likelihood threshold} $\operatorname{mlt}(G) = \min \{\#\text{i.i.d. samples} \mid \mathbf \Sigma \text{ exists with probability one}\}.$

\begin{prop}\label{connection between ms(G) and mlt(G)}
Let $G$ be the finite simple graph corresponding to an $m$-dimensional Gaussian random vector. Let $I(G)$ be the edge ideal of $G$ in the polynomial ring $\RR[x_1, \dots, x_m].$ We have that $\ms(\RR(G)) = \operatorname{mlt}(\overline G)$ if \vspace{0.25cm}
\begin{enumerate}[\rm (1.)]

\item $\overline G$ is chordal; \vspace{0.25cm}

\item $\overline G$ is complete; \vspace{0.25cm}

\item $G$ is complete; or \vspace{0.25cm}

\item $\overline G$ has no induced cycles (i.e., $\overline G$ is a tree).

\end{enumerate}
\end{prop}

\begin{proof}
By \cite[Proposition 1.3]{GrSu} and Proposition \ref{ms(k(G)) leq n - mcn(G) + 3}, if $\overline G$ is chordal, then $\operatorname{mlt}(\overline G) = \omega(\overline G) = \alpha(G) = \ms(\RR(G)).$ If $\overline G$ is complete, then $G$ has no edges, hence $\RR(G) = \RR[x_1, \dots, x_m]$ is a regular standard graded local ring, from which it follows that $\ms(\RR(G)) = m$ by Propositions \ref{ms(R) = ms(R_m) for a standard graded local ring} and \ref{main proposition about cs and ms}. Conversely, if $G$ is complete, then $\ms(\RR(G)) = 1$ by Proposition \ref{ms(K_n)}. Last, if $\overline G$ has no cycles, then $\ms(\RR(G)) = 2$ by Corollary \ref{ms for the complement of a tree}. By the paragraph preceding \cite[Proposition 1.3]{GrSu}, we have that $\ms(\RR(G)) = \operatorname{mlt}(\overline G)$ in each of these cases.
\end{proof}

\begin{question}
Let $G$ be the finite simple graph corresponding to an $m$-dimensional Gaussian random vector. Let $I(G)$ be the edge ideal of $G$ in the polynomial ring $\RR[x_1, \dots, x_m].$ Does it hold that $\ms(\RR(G)) = \operatorname{mlt}(\overline G)$?
\end{question}

\section*{Acknowledgements}

We express our gratitude to Hailong Dao for suggesting this problem to us and for many productive conversations regarding this work. We appreciate the useful comments of Grigoriy Blekherman toward a possible connection with the maximum likelihood threshold of a graph. We thank the creators of the Macaulay2 computer algebra software and especially those who contributed to the {\tt EdgeIdeals} and {\tt RandomIdeals} packages.

\end{document}